\documentclass[generic]{amsart}

\usepackage{geometry}
\usepackage[english]{babel}
\usepackage[utf8]{inputenc}
\usepackage{amsmath,amsthm, amssymb, latexsym, color}
\usepackage{enumitem} 
\usepackage{verbatim}
\usepackage{bbm}
\usepackage{algorithm}
\usepackage[noend]{algpseudocode}
\usepackage{mathrsfs}
\usepackage{comment}
\usepackage[frak=mma]{mathalfa}
\usepackage{graphicx}
\usepackage{ulem}
\usepackage[numbers, sort&compress]{natbib}
\usepackage{xcolor}

\RequirePackage[numbers]{natbib}
\RequirePackage[colorlinks=blue,citecolor=blue,urlcolor=blue]{hyperref}

\DeclareMathOperator{\spn}{span}

\newcommand{\N}{\mathbb N}

\newcommand{\eps}{\varepsilon}

\newcommand{\R}{\mathbb R}
\newtheorem{thm}{Theorem}[section]
\newtheorem{defin}[thm]{Definition}
\newtheorem{lem}[thm]{Lemma}
\newtheorem{cor}[thm]{Corollary}

\newtheorem{rem}[thm]{Remark}
\newtheorem{prop}[thm]{Proposition}

\newtheorem{cond}[thm]{Condition}
\newtheorem{ass}[thm]{Assumption}
\begin{document}

	\title[Langevin algorithms]{On polynomial-time computation of high-dimensional posterior measures by Langevin-type algorithms}
	\author[R. Nickl and S. Wang]{Richard Nickl \text{ and } Sven Wang \\ \\ University of Cambridge $^\dagger$ \\ \\ \today}
	\thanks{\textit{$\dagger$~Department of Pure Mathematics \& Mathematical Statistics, Wilberforce Road, CB3 0WB Cambridge, UK.} Email: nickl@maths.cam.ac.uk, svenwang@mit.edu. We gratefully acknowledge support by the \textit{European Research Council}, ERC grant agreement 647812 (UQMSI) }
	%\runtitle{Langevin algorithms}

	\begin{abstract}
		The problem of generating random samples of high-dimensional posterior distributions is considered. The main results consist of non-asymptotic computational guarantees for Langevin-type MCMC algorithms which scale polynomially in key quantities such as the dimension of the model, the desired precision level, and the number of available statistical measurements. As a direct consequence, it is shown that posterior mean vectors as well as optimisation based maximum a posteriori (MAP) estimates are computable in polynomial time, with high probability under the distribution of the data. These results are complemented by statistical guarantees for recovery of the ground truth parameter generating the data.
		
		Our results are derived in a general high-dimensional non-linear regression setting (with Gaussian process priors) where posterior measures are not necessarily log-concave, employing a set of local `geometric' assumptions on the parameter space, and assuming that a good initialiser of the algorithm is available. The theory is applied to a representative non-linear example from PDEs involving a steady-state Schr\"odinger equation. 
	\end{abstract}

\setcounter{tocdepth}{2}
\maketitle
\tableofcontents

\section{Introduction}

Markov chain Monte Carlo (MCMC) type algorithms are a key methodology in computational mathematics and statistics. The main idea is to generate a Markov chain $(\vartheta_k: k \in \mathbb N)$ whose laws $\mathcal L(\vartheta_k)$ on $\mathbb R^D$ approximate its invariant measure. In Bayesian inference the relevant invariant measure has a probability density of the form
\begin{equation} \label{dergral}
\pi(\theta|Z^{(N)}) \propto e^{\ell_N(\theta)} \pi(\theta),~\theta \in \mathbb R^D.
\end{equation}
Here $\pi$ is a \textit{prior density function} for a parameter $\theta \in \mathbb R^D$ and the map $\ell_N: \mathbb R^D \to \mathbb R$ is the `data-log-likelihood' based on $N$ observations $Z^{(N)}$ from some statistical model, so that $\pi(\cdot|Z^{(N)})$ is the density of the Bayesian \textit{posterior probability distribution} on $\mathbb R^D$ arising from the observations.

\smallskip

It can be challenging to give performance guarantees for MCMC algorithms in the increasingly complex and high-dimensional statistical models relevant in contemporary data science. By `high-dimensional' we mean that the model dimension $D$ may be large (e.g., proportional to a power of $N$). Without any further assumptions accurate sampling from $\pi(\cdot|Z^{(N)})$ in high dimensions can then be expected to be intractable (see below for more discussion).  For MCMC methods the computational hardness typically manifests itself in an \textit{exponential} (or worse) dependence in $D$ or $N$ of the `mixing time' of the Markov chain $(\vartheta_k: k \in \mathbb N)$ towards its equilibrium measure (\ref{dergral}).

\smallskip

 In this work we develop mathematical techniques which allow to overcome such computational hardness barriers. We consider diffusion-based MCMC algorithms targeting the Gibbs-type measure with density $\pi(\cdot|Z^{(N)})$ from (\ref{dergral}) in a non-linear and high-dimensional setting. The prior $\pi$ will be assumed to be Gaussian -- the main challenge thus arises from the non-convexity of $-\ell_N$. We will show how local geometric properties of the statistical model can be combined with recent developments in Bayesian nonparametric statistics \cite{N17, MNP21} and the non-asymptotic theory of Langevin algorithms \cite{D17, DM17, DM18} to justify the `\textit{polynomial time}' feasibility of such sampling methods. 
 
\smallskip

While the approach is general, it crucially takes advantage of the particular geometric structure of the statistical model at hand. In a large class of high-dimensional non-linear inference problems arising throughout applied mathematics, such structure is described by \textit{partial differential equations} (PDEs). Examples that come to mind are inverse and data assimilation problems, and in particular since influential work by A. Stuart \cite{S10}, MCMC-based Bayesian methodology is frequently used in such settings, especially for the task of uncertainty quantification. We refer the reader to \cite{KKSV00, KS04, HLLST04, CMR05, LSS09, CDRS09, S10, MH12, SS12, CRSW13, RC15, GKNSSS15, DS16, CLM16, BGLFS17, AMOS19, G19, BCMW20} and the references therein. A main contribution of this paper is to demonstrate the feasibility of our proof strategy in a (for such PDE problems) prototypical non-linear example where the parameter $\theta$ models the potential in a steady-state Schr\"odinger equation. This PDE arises in various applications such as photo-acoustics, e.g., \cite{BU10, BR11}, and provides a suitable framework to lay out the main mathematical ideas underpinning our proofs.

\subsection{Basic setting and contributions}
 
To summarise our key results we now introduce a more concrete setting. For  $\mathcal O$ a bounded subset of $\R^d, d \in \mathbb N,$ and $\Theta$ some parameter space, consider a family of real-valued bounded `regression' functions $\{\mathcal G(\theta): \theta \in \Theta\}$ defined on $\mathcal O$. If $L^2(\mathcal O)$ denotes the usual space of square Lebesgue-integrable functions, this induces a `forward map'
 \begin{equation}\label{eq:fwd}
 \mathcal G: \Theta \to L^2(\mathcal O),
 \end{equation}
and we suppose that $N$ observations $Z^{(N)}=(Y_i, X_i: i=1, \dots N)$ arising via
 \begin{equation}\label{eq:data:intro}
 Y_i= \mathcal G(\theta)(X_i)+\eps_i, ~~~~~i=1,...,N,
 \end{equation}
 are given, where $\eps_i\sim N(0,1)$ are independent noise variables, and design variables $X_i$ are drawn uniformly at random from the domain $\mathcal O$ (independently of $\eps_i$). While natural parameter spaces $\Theta$ can be infinite-dimensional, in numerical practice a $D$-dimensional discretisation of $\Theta$ is employed, where $D$ can possibly be large. The log-likelihood function of the data $(Y_i, X_i)$ then equals, up to additive constants, the usual least squares criterion
 \begin{equation}\label{eq:llh}
 	\ell_N(\theta)= -\frac 12 \sum_{i=1}^N \big[Y_i-\mathcal G(\theta)(X_i) \big]^2,~~ \theta \in \mathbb R^D.
 \end{equation}
The aim is to recover $\theta$ from $Z^{(N)}$. A wide-spread practice in statistical science is to employ Gaussian (process) priors $\Pi$ with multivariate normal probability densities $\pi$ on $\mathbb R^D$; from a numerical point of view the Bayesian approach to inference in such problems is then precisely concerned with (approximate) evaluation of the posterior measure (\ref{dergral}). 

As discussed above, in important physical applications the forward map $\mathcal G$ is described by a \textit{partial differential equation}. For example suppose that $\mathcal G(\theta)=u_{f_\theta}$ arises as the solution $u=u_{f_\theta}$ to the following elliptic boundary value problem for a Schr\"odinger equation (with $\Delta$ the Laplacian)
\begin{equation}\label{thecat}
\begin{cases}
\frac 12 \Delta u - f_\theta u =0 \text{ on }\mathcal O,\\
u=g \text{ on }\partial \mathcal O,
\end{cases}
\end{equation}
with a suitable parameterisation $\theta\mapsto f_\theta>0$, $\theta\in\R^D$ (see (\ref{fwdG}) below for details).  In such cases, the map $\mathcal G$ is non-linear  and $-\ell_N(\theta)$ is not convex. The probability measure with density $\pi(\cdot|Z^{(N)})$ given in (\ref{dergral}) may then be highly complex to evaluate in a high-dimensional setting, with computational cost scaling exponentially as $D \to \infty$. For instance, complexity theory for high-dimensional numerical integration (see \cite{NW08i, NW08ii} for general references) implies that computing the integral of a $D$-dimensional real-valued Lipschitz function -- such as the normalising factor implicit in (\ref{dergral}) -- by a deterministic algorithm has worst case cost scaling as $D^{D/5}$ \cite{S79,HNUW14}. Relaxing a worst case analysis, Monte Carlo methods can in principle obtain dimension-free guarantees (with high probability under the randomisation scheme). However, a curse of dimensionality may persist as one typically is only able to sample \textit{approximately} from the target measure, and since the approximation error incurred, e.g., by the mixing time of a Markov chain, could scale exponentially in dimension. The references \cite{BBL08, BLB08, BC09, RvH15, YWJ16, MCJFJ18, AGJ20, BAWZ20} discuss this issue in a variety of contexts. In addition, since the distribution becomes increasingly `spiked' as the statistical information increases (i.e., $N\to\infty$), commonly used iterative algorithms can take an exponential in $N$ time to exit neighbourhoods of local optima of the posterior surface $\pi(\cdot|Z^{(N)})$ (e.g., \cite{E15}, Example 4). 

In light of the preceding discussion one may ask whether the approximate calculation of basic aspects of $\pi(\cdot|Z^{(N)})$ -- such as its mean vector (expected value), real-valued functionals $\int_{\mathbb R^D} H(\theta) \pi(\theta|Z^{(N)})d\theta$, or mode -- is feasible at a computational cost which grows at most \textit{polynomially in} $D,N$ and the desired (inverse) precision level. While answering this question in the affirmative may not directly identify a practical algorithm, it clarifies a fundamental aspect of the computational complexity of the problem at hand. Very few rigorous results providing even just partial such guarantees appear to be available. The notable exception Hairer, Stuart and Vollmer \cite{HSV14} along with some other important references will be discussed below. 

\smallskip

Let us describe the scope of the methods to be developed in this article in the problem of approximate computation of the high-dimensional \textit{posterior mean vector} in the PDE model (\ref{thecat}) with the Schr\"odinger equation. We will require mild regularity assumptions on $D, \Pi$ and on the ground truth $\theta_0$ generating the data (\ref{eq:data:intro}) -- full details can be found in Section \ref{sec:schrres}. If $\Pi$ is a $D$-dimensional Gaussian process prior with covariance equal to a rescaled inverse Laplacian raised to some large enough power $\alpha \in \mathbb N$, if the model dimension grows at most as $D\lesssim N^{d/(2\alpha+d)}$, and if $\theta_0$ is sufficiently well-approximated by its `discretisation' in $\R^D$ (see (\ref{dimbias})), we obtain the following main result.

\begin{thm}\label{Cmajor}
	Suppose that data $Z^{(N)}=(Y_i,X_i:i=1,...,N)$ arise through (\ref{eq:data:intro}) in the Schr\"odinger model (\ref{thecat}) and let $P>0$. Then, for any precision level $\eps \ge N^{-P}$ there exists a (randomised) algorithm whose output $\hat\theta_\eps\in \R^D$ can be computed with computational cost
	\begin{equation}\label{OMG}
	O(N^{b_1}D^{b_2}\eps^{-b_3}) ~~~(b_1,b_2,b_3>0),
	\end{equation}
	and such that with high probability (under the joint law of $Z^{(N)}$ and the randomisation mechanism),
	\[ \big\|\hat\theta_\eps-E^{\Pi}[\theta|Z^{(N)}]\big\|_{\R^D}\le \eps, \]
	where $E^{\Pi}[\theta|Z^{(N)}]=\int_{\R^D}\theta \pi(\theta|Z^{(N)})d\theta$ denotes the mean vector of the posterior distribution $\Pi(\cdot|Z^{(N)})$ with density (\ref{dergral}).
\end{thm}

We further show in Theorem \ref{brainfreeze} that $\hat \theta_\eps$ also recovers the ground truth $\theta_0$, within precision $\eps$. The method underlying Theorem \ref{Cmajor} consists of an initialisation step which requires solving a standard convex optimisation problem, followed by iterations $(\vartheta_k)$ of a discretised gradient based Langevin-type MCMC algorithm, at each step requiring a single evaluation of $\nabla \ell_N$ (which itself amounts to solving a standard linear elliptic boundary value problem). In particular our results will imply that the posterior mean can be computed by ergodic averages $(1/J) \sum_{k\le J} \vartheta_k$ along the MCMC chain (after some burn-in time), see Theorem \ref{thm:post:mean} (which implies Theorem \ref{Cmajor}). The laws $\mathcal L(\vartheta_k)$ of the iterates $(\vartheta_k)$ in fact provide a \textit{global} approximation
\[ W_2(\mathcal L(\vartheta_k),\Pi(\cdot|Z^{(N)}))\le \eps, ~~~k\ge k_{mix},\]
of the high-dimensional posterior measure on $\R^D$, in Wasserstein-distance $W_2$. Our explicit convergence guarantees will ensure that both the `mixing time' $k_{mix}$ and the number of required iterations $J$ to reach precision level $\eps$ scales polynomially in $D,N,\eps^{-1}$. Similar statements hold true for the computation of real-valued functionals $\int_{\R^D}H(\theta)\pi(\theta|Z^{N})d\theta$ for Lipschitz maps $H:\R^D\to \R$ and of maximum a posteriori (MAP) estimates. See Theorems \ref{thm:schr:wass}, \ref{MAP} as well as Proposition \ref{thm:schr:func} for precise statements. 

\smallskip

The main ideas of this article can be summarised as follows. We first demonstrate that, with high probability under the law generating the data $Z^{(N)}$, the target measure $\Pi(\cdot|Z^{(N)})$ from (\ref{dergral}) is locally log-concave on a region in $\mathbb R^D$ where most of its mass concentrates.  Then we show that a `localised' Langevin-type algorithm, when initialised into the region of log-concavity, possesses polynomial time convergence guarantees in `moderately' high-dimensional models. That sufficiently precise initialisation is possible has to be shown in each problem individually (for the Schr\"odinger model, see Section \ref{sec:initpfs}). Our proofs provide a template (outlined in Section \ref{barocco}) that can be used in principle also in general settings as long as the linearisation $\nabla_\theta G(\theta_0)$ of $\mathcal G$ at the ground truth parameter $\theta_0$ satisfies a suitable stability estimate (i.e., a quantitative injectivity property related to the `information' operator of the statistical model). We note that this `gradient stability' hypothesis remains entirely `local' and is hence weaker than the `Polyak-Lojasiewicz' gradient condition used in non-convex optimisation \cite{L63, P63}, see also \cite{KNS16}. We verify our local stability property for the Schr\"odinger equation using elliptic PDE techniques (see Lemma \ref{wundervoncordoba}) but our approach may succeed in a variety of other non-linear forward models arising in inverse problems \cite{KKL01, U09, S10, MS12}, integral $X$-ray geometry \cite{PSU12, MNP21, IM20, PSU22}, and also in the context of data assimilation and filtering \cite{CDRS09, MH12, RC15}. In fact, the very recent reference \cite{BN21} achieves this for the non-linear inverse problem considered in \cite{PSU12, MNP21}. Further advancing our understanding of the computational complexity of such PDE-constrained high-dimensional inference problems poses a formidable challenge for future research.

\subsection{Discussion of related literature}\label{subsec:related}

Both the statistical and computational aspects of high-dimensional Bayes procedures have been subject of great interest in recent years. Frequentist convergence properties of high- and infinite-dimensional Bayes procedures were intensely studied in the last two decades. For `direct' statistical models we refer to the recent monograph \cite{GV17} (and references therein), and in the non-linear (PDE) setting relevant here to \cite{NS17, GN19, AN19,N17,NS19,NVW18,MNP21, MNP21b, BGK20, K22,B21,AW21,BN21, S22}.

We now discuss a variety of mixing time results of MCMC algorithms in high-dimensional settings, and refer to the references cited in these articles for further important results. 

\subsubsection{Mixing times for pCN-type algorithms} The important contribution \cite{HSV14} by Hairer, Stuart and Vollmer derives dimension-independent convergence guarantees for the preconditioned Crank-Nicolson (pCN) algorithm, using ergodicity results for infinite-dimensional Markov chains from Hairer, Mattingly and Scheutzow \cite{HMS11}. The task of sampling from a general measure arising from a Gaussian process prior and a general likelihood function $\exp(-\Phi(\theta))$ is considered there. Their results are hence naturally compatible with the setting considered in this paper, where $\Phi$ is given by (\ref{eq:llh}), i.e. $\Phi=\Phi_N=-\ell_N$
and it is natural to ask (a) whether the bounds from \cite{HSV14} apply to this class of problems and (b) if they apply, how they quantitatively depend on $N$ and model dimension.
\par 
The key Assumptions 2.10, 2.11, and 2.13 made in \cite{HSV14} can be summarised as (A) a global lower bound on the acceptance probability of the pCN as well as (B) a (local) Lipschitz continuity requirement on $\Phi$. In non-linear PDE problems, part (B) can usually be verified (e.g., \cite{NVW18}), while part (A) is more challenging: due to the global nature of the assumption, it seems that verification of (A) will typically require bounds for likelihood ratios $\exp(\Phi(\theta)-\Phi(\bar\theta))$ with $\theta,\bar\theta$ arbitrarily far apart. Of course, in some specific problems an initial bound may be obtained by invoking inequalities like (\ref{ubd}). However the resulting lower bounds on the acceptance probabilities in the pCN scheme will decrease exponentially in $N$. We also note that though dimension-independent, the main Theorems 2.12 and 2.14 from \cite{HSV14} remain implicit (non-quantitative) in the relevant quantities from Assumptions (A) and (B); this seems to stem both from the utilised proof techniques, such as considerations regarding level sets of Lyapunov functions (cf.~\cite{HSV14}, p.2474), as well as the qualitative nature of the key underlying probabilistic weak Harris theorem proved by \cite{HMS11}.  Summarising, while it would be very exciting to see the results \cite{HSV14} be extended to yield quantitative bounds which are polynomial in both $N,D$, serious technical and conceptual innovations seem to be required. These remarks apply as well to recent dimension-free mixing time bounds on Hamiltonian Monte Carlo (HMC) methods in \cite{BREZ20, BRE21, GHM21}, which scale exponentially in $N$ via the Lipschitz constant of $\ell_N$. In our context, when exploiting local average curvature of the likelihood surface arising from PDE structure, it is more promising to investigate gradient based MCMC schemes.

\subsubsection{Computational guarantees for Langevin-type algorithms}
For the important gradient-based class of Langevin Monte Carlo (LMC) algorithms, nonasymptotic convergence guarantees which are suited for high-dimensional settings were obtained by Dalalyan \cite{D17} for log-concave densities, shortly after to be extended by Durmus and Moulines \cite{DM17,DM18} to closely related cases. Our proofs rely substantially on these convergence results for the strongly log-concave case (see Appendix \ref{app:ULA} for a review). We emphasize that the fundamental ideas underpinning the fast mixing of `hypercontractive' Langevin diffusions in high dimensions go back to earlier seminal work \cite{BE85,JKO98}, see also the monograph \cite{BGL14}.

Very recently further extensions have emerged, notably \cite{MCJFJ18}, \cite{VW19, MCCFBJ21} and also \cite{CELSZ21}, which estabish convergence guarantees assuming that either the density to be sampled from is convex outside of some region, or that the target measure satisfies functional inequalities of log-Sobolev and Poincar\'e type. However, it appears that both of these results, when applied to (\ref{eq:llh}) without any further substantial work, yield bounds that scale exponentially in $N$. Indeed, the bound in Theorem 1 of \cite{MCJFJ18} evidently depends exponentially on the Lipschitz constant of the gradient $\nabla \ell_N$; and ad hoc verification of assumptions from \cite{VW19} would utilise the Holley-Stroock perturbation principle \cite{HS87} (and (\ref{ubd})), exhibiting the same exponential dependence. Alternative, more elaborate ways of verifying functional inequalities in this context would be highly interesting, but this is not the approach we take here.  

\subsubsection{Relationship to Bernstein-von Mises theorems}

A key idea in our proofs is to use approximate curvature of $\ell_N(\theta)$ `near' the ground truth $\theta_0$. On a deeper level this idea is related to the possibility of a Bernstein-von Mises theorem which would establish precise Gaussian (`Laplace') approximations to posterior distributions, see \cite{L1812, LC86, vdV98} for the classical versions of such results in `low-dimensional' statistical models, and \cite{G00, CN13, CN14, CR15} for high- or infinite-dimensional versions.

Such an approach is taken by \cite{BC09} who attempt to exploit the asymptotic `normality' of the posterior measure to establish bounds on the computation time of MCMC-based posterior sampling, building on seminal work by Lovasz, Simonovits and Vempala \cite{LS93, LV07} on the complexity of general Metropolis-Hastings schemes. While \cite{BC09} allow potentially for moderately high-dimensional situations (by appealing to high-dimensional Bernstein-von Mises theorems from \cite{G00}), their sampling guarantees hold for rescaled posterior measures arising as laws of $\sqrt N(\theta-\tilde \theta)|Z^{(N)}$ where $\tilde \theta=\tilde \theta(Z^{(N)})$ is an initial `semi-parametrically efficient centring' of the posterior draws $\theta|Z^{(N)}$. In our setting such a centring is not generally available (in fact to show that one can compute such centrings, such as the posterior mode or mean, in polynomial time, is a main aim of our analysis). The setting in \cite{BC09} thus appears somewhat unnatural for the problems studied here, also because the conditions there do not appear to permit Gaussian priors.

For the Schr\"odinger equation example considered in the present paper, Bernstein-von Mises theorems were obtained in \cite{N17} -- see also the more recent \cite{MNP21b}. While we follow \cite{N17} in using elliptic PDE theory to quantify the amount of curvature expressed in the `limiting information operator' arising from the Schr\"odinger model, our proofs are in fact not based on an asymptotic Gaussian approximation of the posterior distribution (via Le Cam theory, as in \cite{N17, MNP21b}). Rather we use tools from high-dimensional probability to deduce local curvature bounds directly for the likelihood surface, and then show that the posterior measure is approximated, in Wasserstein distance, by a globally log-concave measure that concentrates around the posterior mode (see Theorem \ref{waterstone}). While one can think of this as a `non-asymptotic' version of a Bernstein-von Mises theorem, the underlying techniques do not require the full inversion of the information operator (as in \cite{N17} or also in \cite{MNP17, NS19, MNP21b}), but solely rely on a `stability estimate' for the local linearisation of the forward map, and hence are likely to apply to a larger class of PDEs (a PDE model where this difference matters is discussed in \cite{NP22}). A further key advantage of our approach is that we do not require the initialiser for the algorithm to be a `semi-parametrically efficient' estimator (as \cite{BC09} do), instead only a sufficiently fast `nonparametric' convergence rate is required, which substantially increases the class of admissible initialisation strategies.

\subsubsection{Regularisation/optimisation literature}

Regularisation-driven optimisation methods have been studied for a long time in applied mathematics, see for instance the monographs  \cite{EHN96, KNS08}. In the setting of non-linear operator equations in Hilbert spaces and with deterministic noise, `local' convergence guarantees for iterative (gradient or `Landweber') methods have been obtained in \cite{HNS95, KNS08}, assuming that optimisation is performed over a (sufficiently small) neighbourhood of a maximum. The proof techniques underlying our main results allow as well to derive guarantees for gradient descent algorithms targeting, for instance, maximum a posteriori (MAP) estimates, see Section \ref{optimap}. Specifically, in Theorem \ref{MAP}, global convergence guarantees for the computation of MAP estimates over a high-dimensional discretisation space are given, in our genuine statistical framework, paralleling our main results for Langevin sampling methods, which can be regarded as randomised versions of classical gradient methods. A main attraction of studying such randomised algorithms, and more generally of solving the problem of Bayesian computation, is of course that one can access \textit{entire} posterior distributions, which is required for quantifying the statistical uncertainty in the reconstruction provided by point estimates such as posterior mean or mode.

\subsection{Notations and conventions}\label{notatio}
Throughout, $N$ will denote the number of observations in (\ref{eq:data:intro}) and $D$ will denote the dimension of the model from (\ref{eq:llh}). For a real-valued function $f:\R^D\to \R$, its gradient and Hessian are denoted by $\nabla f$ and $\nabla ^2f$, respectively, while $\Delta = \nabla^T \nabla$ denotes the Laplace operator. For any matrix $A\in \R^{D\times D}$, we denote the operator norm by 
\[\|A\|_{op}:=\sup_{\psi: \|\psi\|_{\R^D}\le 1}\|A\psi\|_{\R^D}. \]
If $A$ is positive definite and symmetric, then we denote the minimal and maximal eigenvalues of $A$ by $\lambda_{min}(A)$ and $\lambda_{max}(A)$ respectively, with condition number $\kappa(A):=\lambda_{max}(A)/\lambda_{min}(A)$. The Euclidean norm on $\R^D$ will be denoted by $\|\cdot\|_{\R^D}$. The space $\ell^2(\mathbb N)$ denotes the usual sequence space of square-summable sequence $(a_n:n \in \mathbb N)$, normed by $\|\cdot\|_{\ell^2}$. For any $a\in \R$, we write $a_+=\min\{a,0\}$. Throughout, $\lesssim, \gtrsim, \simeq$ will denote (in-)equalities up to multiplicative constants.

For a Borel subset $\mathcal O\subseteq \R^d$, $d\in\N$, let $L^p=L^p(\mathcal O)$ be the usual spaces of functions endowed with the norm $\|\cdot\|_{L^p}^p=\int_{\mathcal O}|h(x)|^pdx$, where $dx$ is Lebesgue measure. The usual $L^2(\mathcal O)$ inner product is denoted by $\langle\cdot,\cdot\rangle_{L^2(\mathcal O)}$. If $\mathcal O$ is a smooth domain in $\R^d$, then $C(\mathcal O)$ denotes the space of bounded continuous functions $h:\mathcal O\to \R$ equipped with the supremum norm $\|\cdot\|_\infty$ and $C^\alpha(\mathcal O), \alpha \in \mathbb N$, denote the usual spaces of $\alpha$-times continuously differentiable functions on $\mathcal O$ with bounded derivatives. Likewise we denote by $H^\alpha(\mathcal O)$ the usual order-$\alpha$ Sobolev spaces of weakly differentiable functions with square integrable partial derivatives up to order $\alpha \in \mathbb N$, and this definition extends to positive $\alpha \notin \mathbb N$ by interpolation \cite{TI}. We also define $(H^2_0(\mathcal O))^*$ as the topological dual space of $$\big(H^2_0(\mathcal O) = \big\{h \in H^2(\mathcal O): tr(h)=0\big\}, \|\cdot\|_{H^2(\mathcal O)}\big),$$ where $tr(\cdot)$ denotes the usual trace operator on $\mathcal O$. We will repeatedly use the inequalities
\begin{equation}
\|gh\|_{H^\alpha}\le c(\alpha, \mathcal O) \|g\|_{H^\alpha}\|h\|_{H^\alpha},\qquad \alpha>d/2, \label{eq:h-mult}
\end{equation}
\begin{equation}
\|h\|_{H^\beta} \le c(\beta, \alpha, \mathcal O) \|h\|^{(\alpha-\beta)/\alpha}_{L^2} \|h\|^{\beta/\alpha}_{H^\alpha},~~0 \le \beta \le \alpha \label{krankl}
\end{equation}
for $g,h \in H^\alpha$, see, e.g., \cite{LM72}. For Borel probability measures $\mu_1,\mu_2$ on $\R^D$ with finite second moments we define the Wasserstein distance 
\begin{equation}\label{verymuchso}
	W^2_2(\mu_1, \mu_2)= \inf_{\nu \in \Gamma(\mu_1, \mu_2)}\int_{\R^D\times\R^D} \|\theta-\vartheta\|^2_{\R^D}d\nu(\theta,\vartheta),
\end{equation}
where $\Gamma(\mu_1, \mu_2)$ is the set of all `couplings' of $\mu_1$ and $\mu_2$ (see, e.g., \cite{V09}). Finally we say that a map $H: \mathbb R^D \to \mathbb R$ is Lipschitz if it has finite Lipschitz norm 
\begin{equation}\label{labello}
\|H\|_{Lip} := \sup_{x\neq y, x,y \in \mathbb R^D} \frac{|H(x)-H(y)|}{\|x-y\|_{\R^D}}.
\end{equation}

\section{Main results for the Schr\"odinger model}\label{sec:schrres}

Our object of study in this section is a nonlinear forward model arising with a (steady state) Schr\"odinger equation. Throughout, let $\mathcal O \subset \mathbb R^d$ be a bounded domain with smooth boundary $\partial \mathcal O$. For convenience, we restrict throughout to $d \le 3$, dimensions $d \ge 4$ could be considered as well at the expense of further technicalities. Moreover, without loss of generality we assume $vol(\mathcal O)=1$.

Suppose that $g\in C^\infty(\partial \mathcal O)$ is a given function prescribing boundary values $g\ge g_{min}>0$ on $\partial O$. For an `attenuation potential' $f\in H^\alpha(\mathcal O)$, consider solutions $u=u_f$ of the PDE
\begin{equation}\label{eq:schr}
	\begin{cases}
		\frac{1}{2}\Delta u-fu=0 ~~~ \text{on}~\mathcal O,\\
		u=g ~~~ \text{on}~ \partial\mathcal O.
	\end{cases}
\end{equation}
If $\alpha>d/2$ and $f \ge 0$ then standard theory for elliptic PDEs (see Chapter 6 of \cite{GT98} or Chapter 4 in \cite{CZ95}) implies that a unique classical solution $u_f\in C^{2}(\mathcal O)\cap C(\bar{\mathcal O})$ to the Schr\"odinger equation (\ref{eq:schr}) exists. The non-linearity of the map $f \mapsto u_f$ becomes apparent from the classical Feynman-Kac formula (e.g., Theorem 4.7 in \cite{CZ95}) 
\begin{equation} \label{fkac}
u_f(x) =u_{f,g}(x)= E^x\left[g(X_{\tau_\mathcal O}) e^{-\int_0^{\tau_\mathcal O} f(X_s)ds} \right],~x \in \mathcal O,
\end{equation} 
where $(X_s: s \ge 0)$ is a $d$-dimensional Brownian motion started at $x$ with exit time $\tau_\mathcal O$ from $\mathcal O$. This PDE appears in various settings in applied mathematics; for example an application to photo-acoustics is discussed in Section 3 in \cite{BU10}. 

\subsection{Bayesian inference with Gaussian process priors}

\subsubsection{The Dirichlet-Laplacian and Gaussian random fields}\label{ONB}

In Bayesian statistics popular choices of prior probability measures arise from Gaussian random fields whose covariance kernels are related to the Laplace operator $\Delta$, see, e.g., Section 2.4 in \cite{S10} and also Example 11.8 in \cite{GV17} (where the closely related `Whittle-Mat\'ern' processes are considered). 

For $\psi\in L^2(\mathcal O)$, let  $v\equiv \mathbb V[\psi]$ denote the (unique) solution in $H^2_0$ to the Poisson equation $\Delta v/2 = \psi$ on $\mathcal O$. By standard results (Section 5.A in \cite{TI}) the compact $\langle \cdot, \cdot \rangle_{L^2(\mathcal O)}$-self-adjoint operator $\mathbb V$ has eigenfunctions $(e_k: k \in \mathbb N)$ forming an orthonormal basis of $L^2(\mathcal O)$ such that $\mathbb V[\psi] = \sum_{k=1}^\infty \mu_k \langle e_k, \psi \rangle_{L^2(\mathcal O)} e_k$, with (negative) eigenvalues $\mu_k$ satisfying the Weyl asymptotics (e.g., Corollary 8.3.5 in \cite{TII})
\begin{equation}\label{weyl}
\lambda_k = \frac{1}{|\mu_k|} \simeq  k^{2/d}~\text{ as } k \to \infty,~~ 0<\lambda_k<\lambda_{k+1},~~k \in \mathbb N.
\end{equation}
The `spectrally defined' Sobolev-type spaces $\mathcal H^\alpha = \{F\in L^2(\mathcal O): \sum_{k=1}^\infty \lambda_k^{\alpha} \langle F, e_k \rangle_{L^2(\mathcal O)}^2<\infty  \}$ are isomorphic to corresponding Hilbert sequence spaces
\[ h^\alpha:=\Big\{ \theta \in \ell^2(\mathbb N): \|\theta\|_{h^\alpha}^2=\sum_{k=1}^\infty \lambda_k^{\alpha} \theta_k^2<\infty\Big\},~~~ h^0=: \ell^2(\mathbb N).\]
One shows that $\mathcal H^\alpha$ is a closed subspace of $H^\alpha(\mathcal O)$ and that the sequence norm $\|\cdot\|_{h^\alpha}$ is equivalent to $\|\cdot\|_{H^\alpha(\mathcal O)}$ on $\mathcal H^\alpha$. For $\alpha$ even, this follows from the usual isomorphism theorems for the $\alpha/2$-fold application of the inverse Dirichlet-Laplacian, and extends to general $\alpha$ by interpolation, see Section 5.A in \cite{TI}. One also shows that any $F \in H^\alpha(\mathcal O)$ supported strictly inside of $\mathcal O$ belongs to $\mathcal H^\alpha$.

\smallskip

A centred Gaussian random field $\mathcal M_\alpha$ on $\mathcal O$ can be defined by the infinite random series
\begin{equation}\label{matern}
\mathcal M_\alpha(x) = \sum_{k=1}^\infty \lambda_k^{-\alpha/2} g_k e_k(x),~x \in \mathcal O,~~g_k \sim^{i.i.d.}N(0,1).
\end{equation}
For $\alpha>d/2$ one shows that $\mathcal M_\alpha$ defines a Gaussian Borel random variable in $C(\mathcal O)\cap \{h \text{ uniformly continuous}: h=0 ~\text{on } \partial \mathcal O\}$ with reproducing kernel Hilbert space equal to $\mathcal H^\alpha$ (see Example 2.6.15 in \cite{GN16}), thus  providing natural priors for $\alpha$-regular functions vanishing at $\partial \mathcal O$. Such Dirichlet boundary conditions could be replaced by Neumann conditions at the expense of minor changes (see p.473 in \cite{TI}). Our techniques in principle may extend to other classes of priors such as exponential Besov-type priors considered in \cite{AW21, LSS09}, but we focus our development here on the most commonly used class of $\alpha$-regular Gaussian process priors.

\subsubsection{Re-parameterisation, regular link functions, and forward map}

To use Gaussian random fields such as $\mathcal M_\alpha$ to model a potential $f \ge 0$ featuring in the Schr\"odinger equation (\ref{eq:schr}), we need to enforce positivity by use of a `link function' $\Phi$. While $\Phi =\exp$ is common, for technical convenience (following \cite{NVW18}) we choose a function that is globally Lipschitz.

\begin{defin}[Regular link function]\label{def:reg:link}
	Let $K_{min}\in [0,\infty)$. We say that $\Phi:\R\to (K_{min},\infty)$ is a \textit{regular link function} if it is bijective, smooth, strictly increasing (i.e. $\Phi'>0$ on $\R$) and if for any $k\ge 1$, the $k$-th derivative of $\Phi$ satisfies $\sup_{x\in \R}\big|\Phi^{(k)}(x)\big|<\infty.$
\end{defin}
For a simple example of a regular link function $\Phi$, see e.g. Example 3.2 of \cite{NVW18}. We denote the composition operator associated to $\Phi$ by
\begin{equation}\label{fifa}
	\Phi^*:L^2(\mathcal O)\to L^2(\mathcal O),~~~~ F\mapsto \Phi\circ F = \Phi^*(F).
\end{equation}
Now to describe a natural parameter space for $f$, we will first expand functions $F \in L^2(\mathcal O)$ in the orthonormal basis from Section \ref{ONB}, 
\begin{equation} \label{Ft}
F=F_\theta = \sum_{k=1}^\infty \theta_k e_k,~~(\theta_k: k=1, 2, \dots) \in \ell^2(\mathbb N),
\end{equation}
and denote by $\Psi(\theta)=F_\theta$ the map $\Psi: \ell^2(\mathbb N) \to L^2(\mathcal O)$ that associates to the vector $\theta$ the `Fourier' series of $F_\theta$.  We then apply a regular link function $\Phi$ to $F_\theta$ and set $f_\theta := \Phi \circ F_\theta$. For $\alpha>d/2$, one shows (see (\ref{links}) below) that $F_\theta \in H^\alpha(\mathcal O)$ implies $f_\theta \in H^\alpha(\mathcal O)$ and hence solutions of the Schr\"odinger equation (\ref{eq:schr}) exist for such $f$. If we denote the solution map $f \mapsto u_f$ from (\ref{eq:schr}) by $G$, then the overall forward map describing our parametrisation is given by
\begin{equation} \label{fwdG}
\mathcal G: h^\alpha \to L^2(\mathcal O),~\mathcal G(\theta) = u_{f_\theta} = [G\circ \Phi^*\circ \Psi](\theta).
\end{equation}
We shall frequently regard $\mathcal G$ as a map on the closed linear subspace $\R^D$ of $h^\alpha$ consisting of the first $D$ coefficients $(\theta_1, \dots, \theta_D)$ of $\theta \in h^\alpha$, and suppress the dependence of $\mathcal G$ on $\Phi$ in the notation. We also note that the solutions of (\ref{eq:schr}) are uniformly bounded by a constant independent of $\theta \in h^\alpha$, specifically
\begin{equation} \label{ubd}
\|\mathcal G(\theta)\|_\infty = \|u_{f_\theta}\|_\infty \le \|g\|_\infty,
\end{equation}
as follows from (\ref{fkac}) and $f_\theta \ge 0$. This `bounded range' property of $\mathcal G$ is relative to the norm employed; for instance the $\|u_{f_\theta}\|_{H^\alpha}$-norms are \textit{not} uniformly bounded in $\theta\in h^\alpha$ for general $\alpha$.

\subsubsection{Measurement model, prior, likelihood and posterior}

For the forward map $\mathcal G$ from (\ref{fwdG}), we now consider the measurement model
\begin{equation}\label{model}
Y_i = \mathcal G(\theta)(X_i) + \varepsilon_i,~~ i=1, \dots, N,~~\varepsilon_i \sim^{i.i.d.} N(0,1), ~~X_i\sim^{i.i.d.} \text{Uniform}(\mathcal O).
\end{equation}
The i.i.d.~random vectors
\begin{equation} \label{ZN}
Z^{(N)}=(Z_i)_{i=1}^N =(Y_i, X_i)_{i=1}^N
\end{equation}
are drawn from a product measure on $(\mathbb R \times \mathcal O)^N$ that we denote by $P^N_\theta=\otimes_{i=1}^N P_\theta$. The coordinate (Lebesgue) densities $p_\theta$ of the joint probability density $p_\theta^N = \prod_{i=1}^N p_\theta$ of $P_\theta^N$ are of the form
\begin{equation} \label{lik}
p_\theta(y,x) := \frac{1}{\sqrt{2 \pi}} \exp \Big\{-\frac{1}{2}[y-\mathcal G(\theta)(x)]^2 \Big\},~~ y \in \mathbb R, x \in \mathcal O,
\end{equation}
(recalling $vol(\mathcal O)=1$) and we can define the \textit{log-likelihood function} as 
\begin{equation}\label{loglik}
\ell_N(\theta) \equiv \log p_\theta^N + N\log \sqrt{2\pi} = -\frac{1}{2} \sum_{i=1}^N\big(Y_i - \mathcal G(\theta)(X_i)\big)^2.
\end{equation}

When using Gaussian process prior models in Bayesian statistics, a common discretisation approach is to truncate the (`Karhunen-Lo\'eve' type) expansion of the prior in a suitable basis, cf. \cite{LSS09, S10, HSV14, DS11}. In our context this will mean that we truncate the series defining the random field $\mathcal M_\alpha$ in (\ref{matern}) at some finite dimension $D$ to be specified. For integer $\alpha$ to be chosen, and recalling the eigenvalues $(\lambda_k: k \in \N)$ of the Dirichlet Laplacian from (\ref{weyl}), we thus consider priors
\begin{equation}\label{schrottprior}
\theta \sim \Pi=\Pi_N \sim N \big(0, N^{-d/(2\alpha+d)} \Lambda^{-1}_\alpha\big), ~~ \Lambda_\alpha = diag(\lambda^{\alpha}_1, \dots, \lambda^{\alpha}_{D}),
\end{equation}
supported in the subspace $\R^D$ of $h^\alpha$ consisting of its first $D$ coordinates. The Lebesgue density  $d\Pi$ of $\Pi$ on $\R^D$ will be denoted by $\pi$. The posterior measure $\Pi(\cdot|Z^{(N)})$ on $\mathbb R^D$ then arises from data $Z^{(N)}$ in (\ref{model}) via Bayes' formula, with probability density function
\begin{align}\label{postbus}
\pi(\theta|Z^{(N)}) &\propto e^{\ell_N(\theta)} \pi(\theta) \\
& \propto \exp\left\{-\frac{1}{2}\sum_{i=1}^N\big(Y_i - \mathcal G(\theta)(X_i)\big)^2 - \frac{N^{d/(2\alpha+d)}}{2} \|\theta\|^2_{h^\alpha} \right\},~\theta \in \mathbb R^D. \notag
\end{align}

\subsection{Polynomial time guarantees for Bayesian posterior computation}

\subsubsection{Description of the algorithm}

We now describe the Langevin-type algorithm targeting the posterior measure $\Pi(\cdot|Z^{(N)})$. It requires the choice of an initialiser $\theta_{init}$ and of constants $\epsilon, K, \gamma$. Our goal is merely to exhibit its polynomial runtime and we do not attempt to optimize the constants involved.

Throughout, we use the initialiser $\theta_{init}=\theta_{init}(Z^{(N)})\in \R^D$ constructed in Theorem \ref{triebelei} in Section \ref{sec:initpfs} (computable in $O(N^{b_0})$ polynomially many steps, for some $b_0>0$). For $\epsilon>0$ to be chosen we define the high-dimensional region
\begin{equation} \label{calbschrott}
\hat {\mathcal B}=\{ \theta \in \mathbb R^D: \|\theta-\theta_{init}\|_{\R^D}\le \epsilon D^{-4/d}/2\}.
\end{equation}
We then construct a proxy function $\tilde \ell_N: \mathbb R^D \to \mathbb R$ which agrees on $\hat {\mathcal B}$  with the log-likelihood function $\ell_N$ from (\ref{loglik}). Specifically, take the cut-off function $\alpha=\alpha_\eta$ from (\ref{eq:alphaeta:def}) and the convex function $g=g_\eta$ from (\ref{eq:geta:def}) with choice $\eta=\epsilon D^{-4/d}$ and $|\cdot|_1=\|\cdot\|_{\R^D}$. Note that $\alpha$ is compactly supported and identically one on $\hat{\mathcal B}$ and that $g$  vanishes on $\hat{\mathcal B}$. Then for $K$ to be chosen, $\tilde \ell_N$ takes the form
\begin{equation}\label{eq:localization}
	\tilde \ell_N(\theta):=\alpha(\theta)\ell_N(\theta) - K g(\theta),~~~\theta \in \mathbb R^D.
\end{equation}
This induces a  proxy probability measure, correspondingly denoted by $\tilde \Pi(\cdot|Z^{(N)})$, with log-density
\begin{equation} \label{surrod}
\log \tilde \pi(\theta|Z^{(N)}) =  \tilde \ell_N (\theta) - N^{\frac{d}{2\alpha+d}}\|\theta\|_{h^\alpha}^2/2 + const.,~\theta \in \mathbb R^D.
\end{equation}
Note that $\tilde \pi(\cdot|Z^{(N)})$ coincides with the posterior density $\pi(\cdot|Z^{(N)})$ on the set $\hat{\mathcal B}$ up to a (random) normalising constant. The MCMC scheme we consider is then given in Algorithm \ref{alg:ULA} and the law of the resulting Markov chain $(\vartheta_k) \in \mathbb R^D$ will be denoted by $\mathbf P_{\theta_{init}}$.

\begin{algorithm}
	\caption{}\label{alg:ULA}
	\begin{flushleft}
	 \hspace*{\algorithmicindent} \textbf{Input:} Initialiser $\theta_{init}\in \R^D$, convexification parameters $\epsilon, K>0$, step size $\gamma>0$, $i.i.d.$ sequence $\xi_k\sim N(0,I_{D\times D})$. \\
	 
	 \smallskip
	 
	\hspace*{\algorithmicindent} \textbf{Output:} Markov chain $\vartheta_1, \dots,\vartheta_k, \dots  \in \R^D$.
	\end{flushleft}
\vspace{0.1cm}
	\begin{algorithmic}[1]
%		\Procedure{localULA}{$\psi^*,A,\gamma,|\cdot|_1,\eta,K$}
		\State \textbf{initialise} $\vartheta_0=\theta_{init}$
		\For{$k=0,...$}
		\State $\vartheta_{k+1}= \vartheta_k+\gamma \nabla \log \tilde \pi (\vartheta_k|Z^{(N)})+\sqrt{2\gamma}\xi_{k+1}$
		\EndFor 
		\State  \textbf{return} $(\vartheta_k:k=1,\dots)$
%		\EndProcedure
	\end{algorithmic}
\end{algorithm}

\smallskip

While the algorithm is related to stochastic optimisation methods based on gradient descent, the diffusivity term is of constant order in $k$, allowing $(\vartheta_k)$ to explore the entire support of the target measure. It coincides with the \textit{unadjusted Langevin algorithm} (see Appendix \ref{app:ULA}) targeting $\pi(\cdot|Z^{(N)})$ as long as the iterates $(\vartheta_k)$ stay within the region $\hat{\mathcal B} \subset \R^D$ we have initialised to. When $(\vartheta_k)$ exits $\hat {\mathcal B}$, the Markov chain is forced by the `proxy' function $\tilde \ell_N$ to eventually return to $\hat{\mathcal B}$.  This procedure is justified since most of the posterior mass will be shown to concentrate on $\hat {\mathcal B}$ with high probability under the law of $Z^{(N)}$. [In fact a key step of our proofs is to control the Wasserstein-distance between the measures induced by the densities $\pi(\cdot|Z^{(N)}), \tilde \pi(\cdot|Z^{(N)})$, cf.~Theorem \ref{waterstone}.] Note that while the ball in (\ref{calbschrott}) shrinks as dimension $D \to \infty$, relative to the step-sizes $\gamma$ permitted below, $\hat {\mathcal B}$ has asymptotically \textit{growing} diameter. The results that follow show that the Markov chain $(\vartheta_k)$ mixes sufficiently fast to reconstruct the posterior surface on $\hat {\mathcal B}$ with arbitrary precision after a polynomial runtime.

\smallskip

To demonstrate the performance of Algorithm \ref{alg:ULA} in a large $N,D$ scenario, we now make the following specific choices of the key algorithm parameters $\epsilon, K, \gamma$.

\begin{cond}\label{asymptopia} 
Let $\theta_{init}$ be the initialiser from Theorem \ref{triebelei} and suppose that
\begin{equation*}
\epsilon := \frac{1}{\log N},~~~K:= ND^{8/d}(\log N)^3 ,~~~\gamma\le \frac{1}{N D^{8/d} (\log N)^4}.
\end{equation*}
\end{cond}

\subsubsection{Conditions involving $\theta_0$}

The convergence guarantees obtained below hold for high-dimensional models where $D$ is permitted to grow polynomially in $N$, and under the frequentist assumption that the data $Z^{(N)}$ from (\ref{model}) is generated from a fixed ground truth $\theta_0$ inducing the law $P_{\theta_0}^N$. Note that we do \textit{not} assume that $\theta_0 \in \mathbb R^D$, but rather that $\theta_0 \in h^\alpha$ is sufficiently well approximated by its $\ell^2(\mathbb N)$-projection $\theta_{0,D}$ onto $\mathbb R^D$. The precise condition, which is discussed in more detail in Remark \ref{alpha} below, reads as follows.
 \begin{cond}\label{FAQ}
For integers $d\le 3$ and $\alpha> 6$, suppose data $Z^{(N)}$ from (\ref{ZN}) arise in the Schr\"odinger model (\ref{model}) for some fixed $\theta_0 \in h^\alpha$. Moreover, suppose that $D \in \mathbb N$ is such that for some constants $c_0>0, 0<c_0'<1/2$, and $\theta_{0,D}=((\theta_{0})_1,...,(\theta_0)_D)$,
\begin{equation} \label{dimbias}
\begin{split}
D \le c_0 N^{d/(2\alpha+d)},~~\|\mathcal G(\theta_{0,D})-\mathcal G (\theta_0)\|_{L^2(\mathcal O)}&\le c_0' N^{-\alpha/(2\alpha+d)}.
\end{split}
\end{equation}
\end{cond}
Though it will be left implicit, the results we obtain in this section depend on $\theta_0$ only through $c_0'$ and an upper bound $S\ge\|\theta_0\|_{h^\alpha}$.

 \subsubsection{Computational guarantees for ergodic MCMC averages}
 
We first present a concentration inequality for ergodic averages along the Markov chain $(\vartheta_k)$. Proposition \ref{thm:schr:func} is non-asymptotic in nature; hence its statement necessarily involves various constants whose dependence on $D$ and $N$ is tracked. Theorems \ref{thm:post:mean} and \ref{brainfreeze} then demonstrate how the desired polynomial time computation guarantees, including Theorem \ref{Cmajor}, can be deduced from it.

\smallskip

For `burn-in' time $J_{in} \in \mathbb N$ and MCMC samples $(\vartheta_k:k=J_{in}+1,...,J_{in}+J)$ from Algorithm \ref{alg:ULA}, define
 \[ \hat \pi_{J_{in}}^J(H)=\frac 1J\sum_{k=J_{in}+1}^{J_{in}+J}H(\vartheta_k),~~~ H: \mathbb R^D \to \mathbb R.\]
We also set, for $c_1>0$ to be chosen,
 \begin{equation}\label{eq:totallybiased}
 B(\gamma):= c_1\Big[\gamma D^{(d+24)/d} (\log N)^6+ \gamma^2ND^{(d+44)/d}(\log N)^{12} \Big]+ 2\exp(-N^{-\frac{d}{2\alpha+d}}).
 \end{equation}
The quantity $B(\gamma)$ is an upper bound for the error incurred by the discretisation of the Langevin dynamics (see (\ref{eq:ULA}) below) and by the `proxy' construction (\ref{surrod}).
 
 \begin{prop}\label{thm:schr:func}
 	Assume Condition \ref{FAQ} is satisfied and consider iterates $\vartheta_k$ of the Markov chain from Algorithm \ref{alg:ULA} with $\theta_{init}, \epsilon, K, \gamma$ satisfying Condition \ref{asymptopia}. Then there exist constants $c_1, c_2,...,c_5>0$ such that for all $N \in \mathbb N$, any Lipschitz function $H:\R^D\to \R$, any burn-in period
 	\begin{equation}\label{fire}
 	J_{in}\ge \frac{\log N}{\gamma ND^{-4/d}} \times \log\big(D + B(\gamma)^{-1}),
 	\end{equation}
	any $J\in \N$, any $t\ge 2\|H\|_{Lip}\sqrt{B(\gamma)}$ and on events $\mathcal E_N$ (measurable subsets of $(\R \times \mathcal O)^N$) of probability $P_{\theta_0}^N(\mathcal E_N)\ge 1-c_2\exp(-c_3N^{d/(2\alpha+d)})$,
 	\begin{equation*}
 	\mathbf P_{\theta_{init}}\big(\big|\hat\pi_{J_{in}}^J-E^\Pi(H|Z^{(N)})\big| \ge t \big)\le c_5\exp\Big( -c_4\frac{t^2N^2J\gamma}{D^{8/d}\|H\|_{Lip}^2(1+D^{4/d}/(NJ\gamma))} \Big).
 	\end{equation*}
 \end{prop}

The next result concerns computation of the posterior mean vector  $$E^\Pi[\theta|Z^{(N)}]=\int_{\R^D} \theta \pi(\theta|Z^{(N)})d\theta$$ 
by ergodic averages
\[\bar\theta_{J_{in}}^J:=\frac{1}{J}\sum_{k=J_{in}+1}^{J_{in}+J}\vartheta_k,~~~ J_{in},J\in \N, \]
within prescribed precision level $\varepsilon$. For convenience we assume $\eps\ge N^{-P}$ for some $P>0$, which is natural in view of the statistical error to be considered in Theorem \ref{brainfreeze} below. To this end, we make an explicit choice for the step size parameter
\begin{equation}\label{gammel}
\gamma=\gamma_\varepsilon = \min\Big(\frac{\varepsilon^2}{D^{(d+24)/d}}, \frac{\varepsilon}{\sqrt ND^{(22+d/2)/d}}, \frac 1{ND^{8/d}}\Big) \times (\log N)^{-7}.
\end{equation}

 \begin{thm}\label{thm:post:mean} 
 Assume Condition \ref{FAQ} is satisfied. Fix $P>0$ and let $\varepsilon \ge N^{-P}$. Consider iterates $\vartheta_k$ of the Markov chain from Algorithm \ref{alg:ULA} with $\theta_{init}, \epsilon, K$ satisfying Condition \ref{asymptopia} and with $\gamma=\gamma_\varepsilon$ as in (\ref{gammel}). Then there exist $c_6,c_7,c_8>0$ and at most polynomially growing constants 
 \begin{equation}\label{polybahn}
 g_{D,N,\varepsilon} = O(D^{\bar b_1} N^{\bar b_2} \varepsilon^{-\bar b_3}), ~~\bar b_1, \bar b_2, \bar b_3>0,
 \end{equation}
such that for all $N \in \mathbb N$, $J_{in} \ge g_{D,N,\eps}$, $J \in \mathbb N$, and on events $\mathcal E_N$ of probability $P_{\theta_0}^N(\mathcal E_N)\ge 1-c_7\exp(-c_8N^{d/(2\alpha+d)})$,
 	\begin{equation}\label{eq:mean:conc}
 	\mathbf{P}_{\theta_{init}}\Big( \big\|\bar\theta_{J_{in}}^J-E^\Pi[\theta|Z^{(N)}]\big\|_{\R^D} \ge \varepsilon \Big)\le c_6 D\exp\Big(-\frac{J}{g_{D,N, \varepsilon}}  \Big).
 	\end{equation}
 \end{thm}
 
Theorem \ref{thm:post:mean} implies that for $J_{in} \wedge J \gg g_{D,N,\varepsilon} \times \log D$, one can compute the posterior mean vector within precision $\eps>0$ with probability as close to one as desired. Using this and Theorem \ref{triebelei} (whose hypotheses are implied by those of Theorem \ref{thm:post:mean}), we have in particular also proven Theorem \ref{Cmajor}. Similar bounds for computation of $E^\Pi(H|Z^{(N)})$ can be obtained as long as $\|H\|_{Lip}$ grows at most polynomially in $D$. 
 
 We conclude this subsection with a result concerning recovery of the actual target of statistical inference, that is, the ground truth $\theta_0$. It combines Theorem \ref{thm:post:mean} with a statistical rate of convergence of $E^\Pi[\theta|Z^{N}]$ to $\theta_0$, obtained by adapting recent results from \cite{MNP21} to the present situation.

\begin{thm}\label{brainfreeze}
Consider the setting of Theorem \ref{thm:post:mean} with $P=\alpha^2/((2\alpha+d)(\alpha+2))$. There exist further constants $c_9, c_{10}, c_{11}, c_{12}>0$ such that for all $N\in \N$, all $\varepsilon \ge c_{11}N^{-\frac{\alpha}{2\alpha+d}\frac{\alpha}{\alpha+2}}$, with $g_{D,N,\eps}$ from (\ref{polybahn}) and on events  $\mathcal E_N$ of probability $P_{\theta_0}^N(\mathcal E_N)\ge 1-c_9\exp(-c_{10}N^{d/(2\alpha+d)})$,
 	\begin{equation}\label{eq:meanzero}
 	\mathbf{P}_{\theta_{init}}\Big(\big\|\bar\theta_{J_{in}}^J-\theta_0\big\|_{\ell^2} \ge \varepsilon \Big) \le c_{12}\exp\Big( -\frac{J}{4g_{D,N, \varepsilon}} \Big).
 	\end{equation}
\end{thm}

While the statistical minimax-optimal rate towards $\theta_0\in h^\alpha$ in this problem can be expected to be faster than $N^{-P}$ (see \cite{N17}), it appears unclear how to obtain this rate when $F_\theta$ is discretised by means of the (for the purposes of the present paper essential) spectral decomposition of the Dirichlet-Laplacian from Section \ref{ONB}. The difficulty arises with the approximation theory of the space $H^\alpha_c(\mathcal O)$ (equal to the completion of $C_c^\infty(\mathcal O)$ in $H^\alpha(\mathcal O)$) and is not discussed further here.

 %[\textcolor{red}{CAUTION: (1) same as the caution after Theorem \ref{thm:schr:func}, and (2) slight abuse of notation with $\R^D$ and $\ell^2(\mathbb N)$ -- to be fixed!}]
 
 %\begin{proof}
 %	We just need to check that the second term in (\ref{eq:mean:conc}) is of at most the same order as the first term, when $t\asymp N^{-\frac{\alpha}{2\alpha+d}}$. Indeed, this can be seen by noting $D\lesssim N^{\frac{d}{2\alpha+4+d}}$, and hence
 %	\[ D\exp\Big(-\frac{c\gamma J N^2t^2 }{D(1+\frac{1}{\gamma NJ})}\Big)\lesssim N^{\frac{d}{2\alpha+4+d}}e^{-c'N^2t^2/D}\le N^{\frac{d}{2\alpha+4+d}}e^{-c''N}, \]
 %	which proves the claim.
 %\end{proof}
 
 \subsubsection{Global bounds for posterior approximation in Wasserstein distance}
 
 The previous theorems concern the computation of specific posterior characteristics; one may also be interested in \textit{global} mixing properties of the laws $\mathcal L(\vartheta_k)$ induced by the Markov chain $(\vartheta_k: k \in \mathbb N)$ towards the target $\Pi(\cdot|Z^{(N)})$, for instance in the Wasserstein distance from (\ref{verymuchso}).
 
 \begin{thm}\label{thm:schr:wass}
 	Assume Condition \ref{FAQ} is satisfied, let $\mathcal L(\vartheta_k)$ denote the law of the $k$-th iterate $\vartheta_k$ of the Markov chain from Algorithm \ref{alg:ULA} with $\theta_{init}, \epsilon, K, \gamma$ satisfying Condition \ref{asymptopia}, and let $B(\gamma), c_1$ be as in (\ref{eq:totallybiased}). For any $P>0$ there exist constants $c_1, c_{13},c_{14},c_{15},c_{16}>0$ such that on events $\mathcal E_N$ of probability $P_{\theta_0}^N(\mathcal E_N)\ge 1-c_{13}\exp(-c_{14}N^{d/(2\alpha+d)})$ and for all $N\in\N$, the following holds.
 	\begin{enumerate}[label=\textbf{\roman*)}]
		\item For any $k\ge 1$,
		 \begin{equation}\label{eq:schr:wass}
		\begin{split}
		W^2_2\big(\mathcal L(\vartheta_k),\Pi[\cdot|Z^{(N)}]\big)\le c_{15} D^{2\alpha/d} (1-c_{16}ND^{-4/d}\gamma)_+^k+ B(\gamma).
		\end{split}
		\end{equation}
		\item For any `precision level' $\eps\ge N^{-P}$ and for $\gamma = \gamma_\eps$ from (\ref{gammel}), there exists
		\begin{equation}\label{mixer}
		k_{mix}= O(N^{\tilde b_1}D^{\tilde b_2}\eps^{-\tilde b_3}),~~~~ \tilde b_1, \tilde b_2, \tilde b_3>0,
		\end{equation}
		such that for any $k\ge k_{mix}$,
		\[ W_2\big(\mathcal L(\vartheta_k),\Pi[\cdot|Z^{(N)}]\big) \le \eps. \]
 	\end{enumerate} 
 \end{thm}

The first term on the right hand side of (\ref{eq:schr:wass}) characterises the rate of geometric convergence towards equilibrium of $(\vartheta_k)$; the factor $ND^{-4/d}\gamma$ can be thought of as a spectral gap of the Markov chain (related to the `average local curvature' of $\ell_N(\cdot)$ near $\theta_0$ in the Schr\"odinger model). Choosing $\gamma=\gamma_\eps$ as in (\ref{gammel}), part ii) further establishes `polynomial-time' mixing of the MCMC scheme towards the posterior measure.

\subsubsection{Computation of the MAP estimate}\label{optimap}

Our techniques also imply the following guarantees for the computation of \textit{maximum a posteriori} (MAP) estimates
\[\hat\theta_{MAP}\in \arg\max_{\theta\in\R^D} \pi (\theta|Z^{(N)}) \]
by a classical gradient (ascent) method applied to the `proxy' posterior surface (\ref{surrod}). 

\begin{thm}\label{MAP}
	Assume Condition \ref{FAQ} is satisfied and let $\theta_{init}$ denote the initialiser from Theorem \ref{triebelei}. For $k=0,1,2,\dots$, consider the gradient algorithm 
	\[ \vartheta_0=\theta_{init}, ~~~ \vartheta_{k+1}=\vartheta_k + \gamma \nabla \log \tilde \pi(\vartheta_k|Z^{(N)}),~~~\gamma = \frac{1}{N D^{8/d} (\log N)^4}.\]
	There exist constants $c_{17},c_{18},c_{19},c_{20},c_{21}>0$ such that for all $N \in \N$ and on events $\mathcal E_N$ of probability at least $P_{\theta_0}^N(\mathcal E_N)\ge 1-c_{17}\exp(-c_{18}N^{d/(2\alpha+d)})$ we have the following:
	\begin{enumerate}[label=\textbf{\roman*)}]
		\item A unique maximiser $\hat\theta_{MAP}$ of $\pi (\theta|Z^{(N)})$ over $\R^D$ exists.
		\item For all $k\ge 1$, we have the geometric convergence
		\[ \|\vartheta_k-\hat\theta_{MAP}\|_{\R^D}^2\le c_{19} D^{4/d} \Big(1- \frac{c_{20}}{D^{12/d}(\log N)^4} \Big)_+^k. \]
		\item Finally, we can choose $k=O(D^{12/d} (\log N)^5)$ such that
		\[ \|\vartheta_k-\theta_0\|_{\ell^2}\le c_{21} N^{-\frac{\alpha}{2\alpha+d}\frac{\alpha}{\alpha+2}}. \]
	\end{enumerate}
\end{thm}

\begin{rem}[About Condition \ref{FAQ}]\label{alpha} \normalfont 
In principle the upper bound for $D$ required in Condition \ref{FAQ} could be replaced by general conditions on $D$ (alike those from Lemma \ref{youtube}) which do not become more stringent as $\alpha$ increases. From a statistical point of view, however, a choice $D\leq c_0 N^{d/(2\alpha+d)}$ is natural as it corresponds to the optimal `bias-variance' tradeoff underpinning the convergence rate towards $\theta_0 \in h^\alpha$ from Theorem \ref{brainfreeze}. [In fact, the second requirement in (\ref{dimbias}) can be checked for $\theta_0\in h^\alpha$ and $D\simeq N^{d/(2\alpha+d)}$, since $\mathcal G$ is $\ell^2(\N) - L^2(\mathcal O)$ Lipschitz.] Moreover, combined with $\alpha>6$, such a choice of $D$ provides a convenient sufficient condition throughout our proofs: It is used critically when showing (in Theorem \ref{waterstone}) that the proxy posterior measure $\tilde \Pi(\cdot|Z^{(N)})$ contracts about a $\|\cdot\|_{\R^D}$-neighbourhood of $\theta_0$ of radius $D^{-4/d}$ on which the information in the Schr\"odinger model has a  stable behaviour (see (\ref{important})). It is also required for our initialiser $\theta_{init}$ to lie in this neighbourhood (Theorem \ref{triebelei}). While it is conceivable that the condition on $\alpha$ could be weakened (as discussed, e.g., in the next remark), it would come at the expense of considerable further technicalities that we wish to avoid here. 
\end{rem}

\section{General theory for random design regression}\label{barocco}

In proving the results from Section \ref{sec:schrres}, we will first develop some theory which applies to general nonlinear regression models. We thus consider in this section the measurement model (\ref{eq:data:intro}) for a general forward model $\mathcal G$ that satisfies a set of analytic conditions to be detailed below. Let $\Theta$ be a (measurable) linear subspace of $\ell^2(\mathbb N)$ which itself admits a subspace $\R^D \subseteq \Theta$ for some $D \in \N$. Let $\mathcal O$ be a Borel subset of $\R^d, d\ge 1$, and consider a model of regression functions $\{\mathcal G(\theta):\theta \in \Theta\}$ via a Borel-measurable forward map $\mathcal G: \Theta\to C(\mathcal O)$. While we regard each $\mathcal G(\theta)$ as a continuous real-valued function, the results of this section readily extend to vector or matrix fields over manifolds $\mathcal O$, see Remark \ref{chopin}. Our data is given by $Z_i=(Y_i, X_i)$ arising from
\begin{equation}\label{eq:data}
Y_i=\mathcal G(\theta)(X_i)+\eps_i,~~~~~ i=1,...,N,
\end{equation} 
where $X_i\sim^{i.i.d.}P^X$, $P^X$ a Borel probability measure on $\mathcal O$, and where $\eps_i\sim^{i.i.d.} N(0,1)$, independently of the $X_i$'s. We write $Z^{(N)}=(Z_1,...,Z_N)$ for the full data vector with joint distribution $P_\theta^N = \otimes_{i=1}^N P_\theta$ on $(\R \times \mathcal O)^N$, with expectation operator $E_\theta^N=\otimes_{i=1}^N E_\theta$. Then the log-likelihood functions of the data $Z^{(N)}$ and of a single observation $Z=(Y,X)\sim P_\theta$ are given by
\begin{equation}\label{eq:gen:llh}
\ell_N(\theta) \equiv \ell_N(\theta, Z^{(N)}) = -\frac 12 \sum_{i=1}^N [Y_i-\mathcal G(\theta)(X_i)]^2, ~~~~
\ell(\theta)\equiv \ell(\theta, Z) = -\frac 12 [Y-\mathcal G(\theta)(X)]^2,
\end{equation}
respectively. If we regard these maps as being defined on $\mathbb R^D \subseteq \Theta$, and if $\Pi$ is a Gaussian prior $\Pi$ supported in $\R^D$, then we obtain the posterior measure $\Pi(\cdot|Z^{(N)})$ with probability density $\pi(\cdot|Z^{(N)})$ on $\mathbb R^D$ as in (\ref{postbus}). 
\smallskip

The main results of this section are Theorems \ref{thm:gen:wass} and \ref{thm:gen:func}, providing convergence guarantees for a Langevin sampling method for the posterior distribution that depend polynomially on model dimension $D$ and number $N$ of measurements, and which hold on an \textit{event} (i.e., a measurable subset $\mathcal E$ of the sample space $(\R \times \mathcal O)^N$ supporting the data $Z^{(N)}$) of the form
\[ \mathcal E:=\mathcal E_{conv}\cap \mathcal E_{init}\cap \mathcal E_{wass}.\] On $\mathcal E_{conv}$  the negative log-likelihood $-\ell_N(\theta)$ will be strongly convex in some region $\mathcal B \subseteq \mathbb R^D$, while $\mathcal E_{init}$ is the event that allows one to initialise the method at some (data-driven) $\theta_{init}=\theta_{init}(Z^{(N)})$ in that set $\mathcal B$. Finally, intersection with $\mathcal E_{wass}$ further guarantees that the posterior measure $\Pi(\cdot|Z^{(N)})$ is close in Wasserstein distance to a \textit{globally} log-concave surrogate probability measure $\tilde \Pi(\cdot|Z^{(N)})$ which locally coincides with $\Pi(\cdot|Z^{(N)})$ up to proportionality factors. In applying the results of this section to a concrete sampling problem, one needs to show that all the events $\mathcal E_{conv}, \mathcal E_{init}, \mathcal E_{wass}$ have sufficiently high frequentist $P_{\theta_0}^N$-probability, where $\theta_0$ is the ground truth parameter generating data (\ref{eq:data}). For the event $\mathcal E_{conv}$ we provide a generic method in Lemma \ref{youtube}, based on a stability estimate for the linearisation of the map $\mathcal G$ combined with high-dimensional concentration of measure techniques. Techniques for controlling the respective probabilities of $\mathcal E_{init}$ and $\mathcal E_{wass}$ are discussed in Remark \ref{cleverchap}.

\smallskip

We will assume the set $\mathcal B \subseteq \mathbb R^D$ of local convexity to be of \textit{ellipsoidal} form.
\begin{defin}\label{def:ellip:norm}
	A norm $|\cdot|$ on $\R^D$ is called ellipsoidal if there exists a positive definite, symmetric matrix $M\in \R^{D\times D}$ such that $|\theta|^2=\theta^TM\theta$  for any $\theta\in\R^D$.
\end{defin}
Throughout this section, for some centring $\theta^* \in \mathbb R^D$, scalar $\eta>0$ and ellipsoidal norm $|\cdot|_1$ with associated matrix $M$, let $\mathcal B$ denote the open subset of $\mathbb R^D$ given by
\begin{equation}\label{eq:B:ass}
\mathcal B:= \big\{ \theta\in\R^D: |\theta-\theta^*|_1< \eta \big \}.
\end{equation}
One may think of $\theta^*$ as the projection of $\theta_0$ onto $\mathbb R^D$, but at this stage this is not necessary. While for the Schr\"odinger model with $d \le 3$ we can choose $|\cdot|_1=\|\cdot\|_{\R^D}$, in general (e.g., when $d \ge 4$ or in other non-linear problems) it may be convenient to consider other (ellipsoidal) localisation regions.

\subsection{Local curvature bounds for the likelihood function}

In what follows, $\theta_0 \in \Theta$ is an arbitrary `ground truth' and the gradient operator $\nabla=\nabla_\theta$ will always act on $\mathcal G, \ell, \ell_N$ viewed as maps on the subspace $\mathbb R^D \subseteq \Theta$. Specifically we shall write $(\nabla \mathcal G(\theta)(x):x\in \mathcal O)$ and $(\nabla^2 \mathcal G(\theta)(x):x\in \mathcal O)$ for the following vector and matrix fields
\begin{equation*}
\nabla \mathcal G(\theta):  \mathcal O \to \R^D,~~~~\nabla^2 \mathcal G(\theta):  \mathcal O \to \R^{D\times D},\end{equation*}
respectively. The following condition summarises some quantitative regularity conditions on the map $\mathcal G$. These have to hold locally on the set $\mathcal B$ (and are satisfied, for instance, for any smooth $\mathcal G$). To formulate them we equip $\R^D$ and $\R^{D\times D}$ with the Euclidean norm $\|\cdot\|_{\R^D}$ and the operator norm $\|\cdot\|_{op}=\|\cdot\|_{\mathbb R^D \to \mathbb R^D}$ (for linear maps from $\mathbb R^D \to \mathbb R^D$) respectively, and the functional norms of $\mathbb R^D$- or $\R^{D \times D}$-valued fields are understood relative to these norms. [So for instance, in (\ref{eq:bddness}), one requires a bound $k_2$ for $\sup_{x \in \mathcal O}\|\nabla^2 \mathcal G(\theta)(x)\|_{\mathbb R^D \to \mathbb R^D}$ that is uniform in $\theta \in \mathcal B$.]

\begin{ass}[Local regularity]\label{ass:model} Let $\mathcal B$ be given in (\ref{eq:B:ass}).

	\begin{enumerate}[label=\textbf{\roman*)}]

		\item For any $x\in\mathcal O$, the map $\theta \mapsto \mathcal G(\theta)(x)$ is twice continuously differentiable on $\mathcal B$.
		\item For some $k_0,k_1,k_2>0$,
		\begin{equation}\label{eq:bddness}
		\begin{split}
			\sup_{\theta\in \mathcal B} \|\mathcal G(\theta)-\mathcal G(\theta_0)\|_{\infty} \le k_0,\\
			\sup_{\theta\in \mathcal B}\| \nabla \mathcal G(\theta) \|_{L^\infty(\mathcal O,\R^D)} \le k_1,\\
			\sup_{\theta\in \mathcal B}\| \nabla^2 \mathcal G(\theta) \|_{L^\infty(\mathcal O,\R^{D\times D})}\le k_2.
		\end{split}
		\end{equation}
		\item  For some $m_0,m_1,m_2>0$ and any $\theta,\bar\theta\in \mathcal B$, we have
		\begin{equation*}
		\begin{split}
		 \| \mathcal G(\theta)-\mathcal G(\bar \theta) \|_\infty&\le m_0|\theta-\bar\theta|_1,\\
		\| \nabla \mathcal G(\theta)- \nabla \mathcal G(\bar \theta)  \|_{L^\infty(\mathcal O,\R^D)}&\le m_1|\theta-\bar\theta|_1,\\
		\| \nabla^2 \mathcal G(\theta)- \nabla^2 \mathcal G(\bar \theta)  \|_{L^\infty(\mathcal O,\R^{D\times D})}& \le m_2|\theta-\bar\theta|_1.
		\end{split}
		\end{equation*}
	\end{enumerate}
\end{ass}

\bigskip

We now turn to the central condition underlying the results in this section in terms of a local curvature bound on $E_{\theta_0}[-\nabla^2\ell (\theta,Z)]$, with $\ell(\theta):\R^D \to \R$ from (\ref{eq:gen:llh}). To motivate it, notice that
\begin{equation} \label{hessen}
-\nabla^2 \ell(\theta, Z) = [\nabla \mathcal G(\theta)(X)] [\nabla \mathcal G(\theta)(X)]^T + [\mathcal G(\theta)(X) - Y] \nabla^2 [\mathcal G(\theta)(X)].
\end{equation}
If the design distribution $P^X$ is uniform on a bounded domain $\mathcal O$ (say, of unit volume) then at $\theta=\theta_0$,  the $E_{\theta_0}^N$-expectation of the last expression can be represented as
\begin{equation}
v^TE_{\theta_0}[-\nabla^2\ell (\theta_0,Z)]v = \|\nabla \mathcal G(\theta_0)^Tv\|_{L^2(\mathcal O)}^2,~~v \in \mathbb R^D.
\end{equation}
Therefore, if a suitable `$L^2(\mathcal O)$-stability estimate' for the linearisation $\nabla \mathcal G$ of $\mathcal G$ at $\theta_0$ is available, the key condition (\ref{eq:ass:lb}) below holds at $\theta_0$; by regularity of $\mathcal G$ this should extend to $\theta$ sufficiently close to $\theta_0$. In the example with the Schr\"odinger equation studied in Section \ref{sec:schrres}, such a stability estimate indeed follows from elliptic PDE theory, see Lemma \ref{wundervoncordoba}, and the recent reference \cite{BN21} verifies this condition for the non-Abelian $X$-ray transform considered in \cite{MNP21}.

Note that the Hessian $E_{\theta_0}[-\nabla^2\ell (\theta,Z)]$ is symmetric (by (\ref{hessen}) and Assumption \ref{ass:model}i)), and recall that $\lambda_{min}(A)$ denotes the smallest eigenvalue of a symmetric matrix $A$.

\begin{ass}[Local curvature]\label{ass:geom}
	Let $\mathcal B$ be given in (\ref{eq:B:ass}) and let $\ell: \R^D \to \R$ be as in (\ref{eq:gen:llh}).
		\begin{enumerate}[label=\textbf{\roman*)}]
		\item For some $c_{min}>0$, we have 
		\begin{equation}\label{eq:ass:lb}
		\inf_{\theta\in \mathcal B} \lambda_{min}\Big(E_{\theta_0}[-\nabla^2\ell (\theta,Z)]\Big)\ge c_{min}.
		\end{equation} 
		\item For some $c_{max}\ge c_{min}>0$, we have 
		\begin{equation}\label{eq:ass:ub}
		\sup_{\theta\in \mathcal B}\Big[ |E_{\theta_0}\ell(\theta, Z)|+\|E_{\theta_0}[\nabla \ell(\theta, Z) ] \|_{\R^D}+\|E_{\theta_0}[\nabla^2\ell (\theta,Z)]\|_{op} \Big] \le c_{max}.
		\end{equation}
	\end{enumerate}
\end{ass}

\medskip

The following lemma, which is based on concentration of measure arguments, shows that the local `average' curvature bound in (\ref{eq:ass:lb}) carries over to the `observed' log-likelihood function, with high frequentist $P_{\theta_0}^N$-probability, and whenever $D \le \mathcal R_N$, where the dimension constraint is explicitly quantified in terms of the constants featuring in the previous hypotheses. The expression for $\mathcal R_N$ substantially simplifies in concrete settings but, in this general form, reflects the various non-asymptotic stochastic regimes of the log-likelihood function and its derivatives.

\begin{lem}\label{youtube}
	Suppose that data arises from (\ref{eq:data}) with $\ell_N: \R^D \to \R$ given by (\ref{eq:gen:llh}). Suppose  Assumptions \ref{ass:model}, \ref{ass:geom} are satisfied. There exists a universal constant $C>0$ such that if
	\begin{equation}\label{eq:itis what itis}
	\mathcal R_N:=CN\min\bigg\{  \frac{c_{min}^2}{C_{\mathcal G}^2\eta^2},\frac{c_{min}}{C_{\mathcal G}\eta}, \frac{c_{min}^2}{C_{\mathcal G}'^2},\frac{c_{min}}{k_2}, \frac{c_{max}^2}{C_{\mathcal G}''^2\eta^2},\frac{c_{max}}{C_{\mathcal G}''\eta}, \frac{c_{max}^2}{C_{\mathcal G}'''^2},\frac{c_{max}}{k_0+k_1} \bigg\},
	\end{equation}
where
	\begin{equation}\label{eq:CG}
	\begin{split}
		C_{\mathcal G}:=k_0m_2+k_1m_1+ k_2m_0+m_2,&~~~~~ C_{\mathcal G}':=k_1^2+k_0k_2+k_2,\\
		C_{\mathcal G}'':= k_0m_1+k_1m_0+m_1+k_0m_0+m_0,& ~~~~~C_{\mathcal G}'''=k_0k_1+k_1+k_0^2+k_0,
	\end{split}
	\end{equation}
	then for any $D,N\ge 1$ satisfying $D\le \mathcal R_N$, we have	\begin{equation}\label{eq:emp:lb}
	\begin{split}
	P_{\theta_0}^N\Big(\inf_{\theta\in \mathcal B}\lambda_{min}\big[-\nabla^2\ell_N(\theta, Z^{(N)})\big]< \frac 12Nc_{min}\Big) &\le 8e^{-\mathcal R_N},
	\end{split}
	\end{equation}
	as well as 
	\begin{equation}\label{eq:emp:ub}
	\begin{split}
	P_{\theta_0}^N\Big( \sup_{\theta\in \mathcal B}\Big[|\ell_N(\theta, Z^{(N)})|+\|\nabla \ell_N(\theta, Z^{(N)})\|_{\R^D}+\|\nabla^2\ell_N(\theta, Z^{(N)})\|_{op}\Big]> N(5c_{max}+1)\Big) ~~~~~~~~&\\
	\le 24e^{-\mathcal R_N} + e^{-N/8}&.
	\end{split}
	\end{equation}
	\end{lem}

Inspection of the proof (given in Section \ref{subsec:conc:pf}) shows that for the first inequality (\ref{eq:emp:lb}), the terms involving $c_{max}$ can be removed from the definition of  $\mathcal R_N$. In the sequel we will restrict considerations to the event
\begin{equation}\label{eq:E:event}
\begin{split}
\mathcal E_{conv} &:=\Big\{  \inf_{\theta\in \mathcal B}\lambda_{min}\big[-\nabla^2\ell_N(\theta)\big]\ge Nc_{min}/2  \Big\}\\
&~~~~~~~~~~~\cap \Big\{  \sup_{\theta\in\mathcal B} \Big[|\ell_N(\theta)| + \|\nabla \ell_N(\theta)\|_{\R^D}+\|\nabla^2 \ell_N(\theta)\|_{op}\Big]\le N(5c_{max}+1) \Big\},
\end{split}
\end{equation}
whose $P_{\theta_0}^N$-probability is controlled by Lemma \ref{youtube}.

\subsection{Construction of the likelihood surrogate function}\label{subsec:surrogate}

For Bayesian computation via Langevin-type algorithms one needs to ensure recurrence of the underlying diffusion process, a sufficient condition for which is \textit{global log-concavity} (on $\R^D$) of the target measure to be sampled from, see Appendix \ref{app:ULA}. To this end we now construct a `surrogate log-likelihood function' $\tilde\ell_N: \mathbb R^D \to \mathbb R$ for the log-likelihood $\ell_N$ such that $\tilde \ell_N=\ell_N$ identically on the subset $\{\theta \in \mathbb R^D: |\theta-\theta^* |_1\le 3\eta/8\}$ of $\mathcal B$ from (\ref{eq:B:ass}), and which will be shown to be globally log-concave on the event $\mathcal E$ from (\ref{nevergonnahappen}) below.

\smallskip

In order to perform the convexification of $-\ell_N$, one needs to identify the region $\mathcal B$ up to sufficient precision. In what follows, we denote by $\theta_{init}=\theta_{init}(Z^{(N)})\in\R^D$ a (data-driven) point estimator where the sampling algorithm is initialised; and we define the event $\mathcal E_{init}$ (measurable subset of $(\R \times \mathcal O)^N$) by 
\begin{equation} \label{initevent}
 \mathcal E_{init}:=\big\{|\theta_{init}-\theta^* |_1\le \eta/8\big\},
 \end{equation}
 where $\theta_{init}$ belongs to the region $\mathcal B$. That such initialisation is possible (i.e., that $\mathcal E_{init}$ has sufficiently high $P_{\theta_0}^N$-probability for appropriate $\eta>0$) is proved for the Schr\"odinger model in Theorem \ref{triebelei}.

\smallskip

We require two auxiliary functions, $g_\eta$ (globally convex) and $\alpha_\eta$ (cut-off function): For some smooth and symmetric (about $0$) function
$\varphi:\R\to [0,\infty)$ satisfying $\text{supp}(\varphi)\subseteq [-1,1]$ and $\int_{\R}\varphi(x)dx=1$, let us define the mollifiers $\varphi_h(x):=h^{-1}\varphi(x/h), h>0$. Then, we define the functions $\tilde \gamma_\eta, \gamma_\eta: \R \to \R$ by
\begin{equation}\label{eq:gamma:tilde}
\begin{split}
\tilde \gamma_\eta (t)~&:=\begin{cases}
0 ~~&\text{if}~~~ t<5\eta/8,\\
(t-5\eta/8)^2 ~~~&\text{if}~~~ t \ge 5\eta/8,
\end{cases}\\
\gamma_\eta (t)~&:=\big[\varphi_{\eta/8}\ast \tilde \gamma_\eta\big](t),
\end{split}
\end{equation}
where $\ast$ denotes convolution, and
\begin{equation}\label{eq:geta:def}
g_\eta:\R^D\to [0,\infty), ~~~~~~ g_\eta(\theta):=\gamma_\eta(|\theta-\theta_{init}|_1).
\end{equation}
Finally, for some smooth $\alpha: [0,\infty)\to [0,1]$ which satisfies $\alpha (t)=1$ for $t\in [0,3/4]$ and $\alpha(t)=0$ for $t\in [7/8,\infty)$, we define the `cut-off' function
\begin{equation}\label{eq:alphaeta:def}
\alpha_\eta:\R^D\to [0,1], ~~~~~~ \alpha_\eta(\theta)=\alpha\big(|\theta-\theta_{init}|_{1}/\eta \big).
\end{equation}

\begin{defin}
For the auxiliary functions $g_\eta, \alpha_\eta$ from (\ref{eq:geta:def}), (\ref{eq:alphaeta:def}) and $K>0$, we define the surrogate likelihood function $\tilde \ell_N$ by
\begin{equation}\label{eq:tilde:ell}
\begin{split}
\tilde \ell_N:\R^D\to \R,& ~~~~ \tilde \ell_N(\theta):=\alpha_\eta (\theta) \ell_N (\theta)-Kg_\eta (\theta).
\end{split}
\end{equation}
\end{defin}

%[\textcolor{red}{Think about inserting a graphic, illustrating $\alpha$ and $g_\eta$ in the `radial coordinate' $|\theta-\theta^*|_1$, where key properties would be: (1) $\alpha=1$ on $[0,3/4]$, (2) that $g_\eta=0$ on $[\eta/2,\infty)$, such that indeed, $\ell=\tilde \ell$ on $\mathcal B_{1/2}$, (3) $g_\eta$ has constant curvature for outside a ball of radius $3\eta/4$, enforcing convexity over the potentially non-convexity of $\ell_N\cdot \alpha$.}]

When the choice of the constant $K>0$ is large enough relative to $c_{max}$ from Assumption \ref{ass:model}, the following global convexity property can be proved for $\tilde\ell_N$ (see Appendix \ref{sec:aux} for a proof).

\begin{prop}\label{prop:log:conv}
	On the event $\mathcal E_{conv}\cap \mathcal E_{init}$ (cf.~(\ref{eq:E:event}), (\ref{initevent})), when $\tilde \ell_N$ from (\ref{eq:tilde:ell}) is defined with any constant $K$ satisfying
	\begin{equation}\label{eq:K:cond}
	K\ge CN(c_{max} +1)\cdot\frac{1+\lambda_{max}(M)/\eta^2}{\lambda_{min}(M)},
	\end{equation}
	($C>1$ depending only on the function $\alpha$ above), we have
	\begin{equation*}
	\ell_N(\theta) = \tilde \ell_N(\theta) ~~~\text{ for all } \theta \in \mathbb R^D \text{ s.t. } |\theta-\theta^* |_1\le 3\eta/8.
	\end{equation*}
	Moreover, $\tilde \ell_N\in C^2(\R^D)$ and it holds that
	\begin{align}\label{eq:str:conv}
	\inf_{\theta \in \mathbb R^D}\lambda_{min}\big(-\nabla^2 \tilde\ell_N(\theta)\big)\ge Nc_{min}/2,
	\end{align}
	as well as
	\begin{align}\label{eq:grad:lip}
	\|\nabla \tilde \ell_N(\theta)-\nabla \tilde \ell_N(\bar \theta)\|_{\R^D}\le 7K \lambda_{max}(M) \|\theta-\bar\theta\|_{\R^D},~~~\theta,\bar\theta\in \R^D.
	\end{align}
\end{prop}

\subsection{Non-asymptotic bounds for Bayesian posterior computation}\label{subsec:genMAIN}
We now consider the problem of generating random samples from the posterior measure
\[  \Pi[B|Z^{(N)} ]=\frac{\int_B e^{\ell_N(\theta,Z^{(N)})}d\Pi(\theta)}{\int_{\R^D} e^{\ell_N(\theta,Z^{(N)})}d\Pi(\theta)}, ~~~~ B \subseteq \R^D \text{ measurable}, \] arising from data (\ref{eq:data}) with log-likelihood (\ref{eq:gen:llh}) and Gaussian $N(0,\Sigma)$ prior $\Pi$ of density $\pi$ on $\mathbb R^D$, with positive definite covariance matrix $\Sigma \in \R^{D\times D}$. 

We use the stochastic gradient method obtained from an Euler discretisation of the $D$-dimensional Langevin diffusion (see Appendix \ref{app:ULA}) with drift vector field $\nabla (\tilde \ell_N + \log \pi)$ based on the surrogate likelihood function. More precisely, for \textit{stepsize} $\gamma>0$ and auxiliary variables $\xi_{k} \sim^{i.i.d.}  N(0,I_{D \times D})$, define a Markov chain as
\begin{equation}\label{ULAgen}
\begin{split}
\vartheta_0 &= \theta_{init}, \\
\vartheta_{k+1} &= \vartheta_k + \gamma \big[\nabla \tilde \ell_N(\vartheta_k) -  \Sigma^{-1}\vartheta_k \big] + \sqrt{2 \gamma} \xi_{k+1},~~k=0,1,\dots
\end{split}
\end{equation}
Probabilities and expectations with respect to the law of this Markov chain (random only through the $\xi_{k}$, conditional on the data $Z^{(N)}$) will be denoted by $\mathbf P_{\theta_{init}}, \mathbf E_{\theta_{init}}$ respectively. The invariant measure of the underlying continuous time Langevin diffusion equals the \textit{surrogate posterior distribution} given by
\[ \tilde \Pi[B|Z^{(N)} ]:=\frac{\int_B e^{\tilde \ell_N(\theta,Z^{(N)})}d\Pi(\theta)}{\int_{\R^D} e^{\tilde \ell_N(\theta,Z^{(N)})}d\Pi(\theta)}, ~~~~  B \subseteq \R^D \text{ measurable}. \]

\smallskip

In the following results we assume that the Wasserstein distance $W_2$ between $\tilde \Pi(\cdot|Z^{(N)})$ and $\Pi(\cdot|Z^{(N)})$ can be controlled, specifically, for any $\rho>0$, let us define the event 
\begin{equation}\label{wasser}
\mathcal E_{wass}(\rho):=\big\{ W^2_2\big( \Pi\big[ \cdot|Z^{(N)} \big], \tilde\Pi\big[ \cdot|Z^{(N)} \big] \big)\le \rho/2 \big\}. 
\end{equation}
For the Schr\"odinger model this is achieved in Theorem \ref{waterstone}, for $\rho$ decaying exponentially in $N$, using that most of the posterior mass (and its mode) concentrate on the set $\mathcal B$ from (\ref{eq:B:ass}), and the ideas underlying this proof extend to general settings, see Remark \ref{cleverchap}.

\smallskip

Our first result consists of a global Wasserstein-approximation of $\Pi(\cdot|Z^{(N)})$ by the law $\mathcal L(\vartheta_k)$ on $\mathbb R^D$ of the $k$-th iterate $\vartheta_k$ arising from (\ref{ULAgen}).

\begin{thm}[Non-asymptotic Wasserstein mixing] \label{thm:gen:wass} 
Suppose that the model given by (\ref{eq:data})-(\ref{eq:gen:llh}) fulfills the Assumptions \ref{ass:model}, \ref{ass:geom} for some $0<\eta \le 1$, that $D, N \in \mathbb N$ are such that $D\leq \mathcal R_N$ with $\mathcal R_N$ from (\ref{eq:itis what itis}) and let $K$ be as in (\ref{eq:K:cond}). Further define the constants
	\[m:= Nc_{min}/2+ \lambda_{min}(\Sigma^{-1}),~~~~\Lambda:= 7K\lambda_{max}(M)+\lambda_{max}(\Sigma^{-1}). \]
Then for any $0<\gamma \le 1/ \Lambda$ and any $\rho>0$ the algorithm $(\vartheta_k:k\ge 0)$ from (\ref{ULAgen}) satisfies, on the event (i.e., measurable subset of $(\R \times \mathcal O)^N$)
\begin{equation} \label{nevergonnahappen}
\mathcal E:=\mathcal E_{conv}\cap \mathcal E_{init}\cap \mathcal E_{wass}(\rho),
\end{equation}
%	\[1- \P_0^N( \mathcal E_{init}^c)-\P_0^N( \mathcal E_{wass}(\rho)^c) - c_2e^{-\mathcal R_N}, \]
	(with $\mathcal E_{conv}, \mathcal E_{init}, \mathcal E_{wass}(\rho)$ defined in (\ref{eq:E:event}), (\ref{initevent}), (\ref{wasser}), respectively), and all $k\ge 0$,
	\begin{equation}\label{eq:gen:wass}
		W^2_2\big( \mathcal L(\vartheta_k), \Pi[ \cdot | Z^{(N)} ] \big)\le \rho + b(\gamma) +  4 \big(\tau(\Sigma, M, R)+\frac{D}{m} \big)\Big(1-\frac{\gamma m}2\Big)^k,
	\end{equation}
	where, for some universal constants $c_1,c_2>0$, any $R\ge \|\theta^*\|_{\mathbb R^D}$ and $\kappa(\Sigma)= \lambda_{max}(\Sigma)/\lambda_{min}(\Sigma)$,
	\begin{equation}\label{eq:ULAbias}
		b(\gamma)= c_1\Big[\frac{\gamma D\Lambda^2}{m^2}+\frac{\gamma^2 D\Lambda^4}{m^3} \Big],~~\tau(\Sigma, M, R) = c_2\kappa (\Sigma)\Big[1+ \frac{\eta^2}{\lambda_{min}(M)} + R^2\Big].
	\end{equation}
\end{thm}

\smallskip

From the previous theorem we can obtain the following bound on the computation of posterior functionals by ergodic averages of $\vartheta_k$ collected after some burn-in time $J_{in}\in \mathbb N$. Specifically, if we define, for any $H: \mathbb R^D \to \mathbb R$ integrable with respect to $\Pi(\cdot|Z^{(N)})$, the random variable 
\begin{equation}\label{ergopop}
\hat \pi_{J_{in}}^J(H) = \frac{1}{J}\sum_{k=J_{in}+1}^{J_{in}+J} H(\vartheta_k),
\end{equation}
we obtain the following non-asymptotic concentration bound.

\begin{thm}[Lipschitz functionals] \label{thm:gen:func}
	In the setting of the previous theorem, there exist further constants $c_3, c_4>0$ such that for any $\rho>0$, any burn-in period
	\begin{equation}\label{eq:burn}
		J_{in}\ge \frac{c_3}{m\gamma}\times \log \Big(1+ \frac{1}{\rho+b(\gamma)}+\tau(\Sigma, M, R)+\frac{D}{m}\Big),
	\end{equation}
	any $J\in \N$, any Lipschitz function $H: \mathbb R^D \to \mathbb R$, any
	\begin{equation}\label{eq:tea}
		t\ge \sqrt 8\|H\|_{Lip}\sqrt{ \rho + b(\gamma)}
	\end{equation}
	and on the event $\mathcal E$ from (\ref{nevergonnahappen}), we have 
	\begin{equation}\label{eq:gen:conc}
	\begin{split}
	&\mathbf P_{\theta_{init}}\Big(\big|\hat \pi_{J_{in}}^J(H)- E^{\Pi}[H|Z^{(N)}]\big| \ge t \Big)\le 2\exp\Big(-c_4\frac{t^2m^2J\gamma}{\|H\|_{Lip}^2(1+1/(mJ \gamma))}\Big).
	\end{split}
	\end{equation}
\end{thm}

\smallskip

From the last theorem one can obtain as a direct consequence the following guarantee for computation of the posterior mean $E^\Pi[\theta|Z^{(N)}]$ by the ergodic average accrued along the Markov chain. 

\begin{cor}\label{generallymean}
In the setting of Theorem \ref{thm:gen:func}, if we define $$\bar \theta_{J_{in}}^J = \frac{1}{J}\sum_{k=J_{in}+1}^{J_{in}+J} \vartheta_k,$$ then on the event $\mathcal E$ and for $t\ge \sqrt 8\sqrt{ \rho + b(\gamma)}$, we have for some constant $c_5>0$ that
\begin{equation}\label{eq:gen:concmean}
	\begin{split}
	&\mathbf P_{\theta_{init}}\Big(\big\|\bar \theta_{J_{in}}^J-E^{\Pi}[\theta|Z^{(N)}]\big\|_{\R^D} \ge t \Big)\le 2D\exp\Big(-c_5\frac{t^2m^2J\gamma}{D(1+1/(mJ\gamma)}\Big).
	\end{split}
	\end{equation}
\end{cor}

The two previous results imply that one can compute the posterior mean (or $E^{\Pi}[H|Z^{(N)}]$ with $\|H\|_{Lip}\le1$) within precision $\eps>0$ as long as $\epsilon\gtrsim \sqrt \rho$: For instance if $\gamma$ is chosen as
\[ \gamma\simeq \min \Big\{ \frac{\eps^2 m^2}{D\Lambda^2},\frac{\eps m^{3/2}}{D^{1/2}\Lambda^2} \Big\},  \]
%~~~ J_{in}\gtrsim \frac{\log (1+D/m)+\log (1+\eps^{-1})}{\gamma Nc_{min}}
then the overall number of required MCMC iterations $J_{in}+J$ depends polynomially on the quantities $N,D,m^{-1},\Lambda,\eps^{-1}$. When the latter three constants exhibit at most polynomial growth in $N,D$ (as is the case for the Schr\"odinger equation treated in Section \ref{sec:schrres}), we can deduce that polynomial-time computation of such posterior characteristics is feasible, on the event $\mathcal E$ from (\ref{nevergonnahappen}) at computational cost  $J_{in}+J=O(N^{b_1}D^{b_2}\eps^{-b_3}),b_1,b_2,b_3>0$, with $\mathbf P_{\theta_{init}}$-probability as close to $1$ as desired.

\begin{rem}[About the events $\mathcal E_{init}, \mathcal E_{wass}$] \label{cleverchap} \normalfont
Controlling the probability of the events $\mathcal E_{init}, \mathcal E_{wass}$ (featuring in the definition of $\mathcal E$ in (\ref{nevergonnahappen})) on which the preceding bounds hold may pose a formidable challenge in its own right when considering a concrete `forward map' $\mathcal G$. For our prototypical example of the Schr\"odinger equation from Section \ref{sec:schrres}, this is achieved in Sections \ref{sec:freqpfs} and \ref{sec:initpfs}. The proofs there give some guidance for how to proceed in other settings, too. In essence one can expect that in bounding the $P_{\theta_0}^N$-probability of the events $\mathcal E_{init}, \mathcal E_{wass}$,  \textit{global} `stability' and `range' properties of the map $\mathcal G$ will play a role. In contrast,  Assumptions \ref{ass:model}, \ref{ass:geom} employed in this section are `local'  in the sense that they concern properties of $\mathcal G$ on $\mathcal B$ from (\ref{eq:B:ass}) only. Discerning local from global requirements on $\mathcal G$ in this way appears helpful both in the proofs and in the exposition of the main ideas of this paper. Following the ideas in Section \ref{sec:freqpfs} below, the recent contribution \cite{BN21} provides a set of conditions on $\mathcal G$ under which a log-concave approximation of the posterior measure similar to Theorem \ref{waterstone} holds true.
\end{rem}
\begin{rem}[Extensions to vector-valued data]\label{chopin}\normalfont
The key results of this section apply to other settings (e.g. in \cite{MNP21, PSU22}) where the `forward' map $\mathcal G(\theta)$ defines an element of the space of continuous maps $C(\mathcal M \to V)$ from a $d$-dimensional compact manifold $\mathcal M$ (possibly with boundary) into a finite-dimensional inner product space $V$ of fixed finite dimension $dim(V)<\infty$. If we assume that the statistical errors $(\varepsilon_i: i=1, \dots, N)$ in equation (\ref{eq:data}) are i.i.d.~$N(0,\textrm{Id}_V)$ in $V$, then the log-likelihood function of the model is not given by (\ref{eq:gen:llh}) but instead of the form $$\ell_N(\theta) = -\frac{1}{2}\sum_{i=1}^N \|Y_{i}-\mathcal G(\theta)(X_i)\|_V^2, ~~\ell(\theta) = -\frac{1}{2}\|Y-\mathcal G(\theta)(X)\|_V^2, $$ 
where the $X_i, X$ are drawn i.i.d.~from a Borel measure $P^X$ on $\mathcal M$. Imposing Assumption \ref{ass:model} with the obvious modification of the norms there for $V$-valued maps, and if Assumption \ref{ass:geom} holds for the preceding definition of $\ell(\theta)$, then the conclusion of Lemma \ref{youtube} remains valid as stated (see also the proof of Lemma 5.6 in \cite{BN21}).
\end{rem}

\subsection{Proof of Lemma \ref{youtube}}\label{subsec:conc:pf}
It suffices to prove the assertion for $\mathcal R_N\ge 1$. We first need some more notation: For any $x\in \mathcal O$, we denote the point evaluation map by
\begin{equation*}
\mathcal G^x:~~ \Theta\to \R, ~~~~~\theta\mapsto \mathcal G(\theta)(x).
\end{equation*}
For $Z=(Y,X)\sim P_{\theta_0}$, we will frequently use the following identities in the proofs below (where we recall that $\nabla$ and $\nabla^2$ act on the $\theta$-variable).
\begin{equation}\label{eq:llhderivs}
\begin{split}
-\ell(\theta,Z)&= \frac 12 \big[Y-\mathcal G^X(\theta)\big]^2=\frac 12 \big[\mathcal G^X(\theta_0)+\eps-\mathcal G^X(\theta) \big]^2,\\
-\nabla \ell (\theta,Z)&= \big[ \mathcal G^X(\theta)- \mathcal G(\theta_0) -\eps  \big] \nabla \mathcal G^X(\theta),\\
-\nabla^2 \ell (\theta,Z)&= \nabla \mathcal G^X(\theta)\nabla \mathcal G^X(\theta)^T + \big[ \mathcal G^X(\theta)- \mathcal G(\theta_0) -\eps \big]\nabla^2 \mathcal G^X(\theta),\\
-E_{\theta_0}\big[\ell(\theta,Z)\big]&= \frac 12 +\frac 12 E^X[\mathcal G^X(\theta_0)-\mathcal G^X(\theta)]^2,
\end{split}
\end{equation}
where we note that by Assumption \ref{ass:model}, the Hessian $\nabla^2 \ell (\theta,Z)$ is a symmetric $D\times D$ matrix field. When no confusion can arise, we will suppress the second argument $Z$ and write $\ell(\theta)$ for $\ell(\theta,Z)$. 
\par 
Throughout, $P_N:=N^{-1}\sum_{i=1}^N\delta_{Z_i}$ denotes the empirical measure induced by $Z^{(N)}$, which acts on measurable functions $h:\R\times \mathcal O\to \R$ via
\[ P_N(h)=\int_{\R\times \mathcal O}hdP_N=\frac 1N\sum_{i=1}^Nh(Z_i).  \] 

\subsubsection{Proof of (\ref{eq:emp:lb})}
Let us write $\bar \ell_N:=\ell_N/N$. Then, by a standard inequality due to Weyl as well as Assumption \ref{ass:geom}, we have for any $\theta\in \mathcal B$ that
\begin{equation}\label{eq:cmin}
\begin{split}
\lambda_{min}\big[-\nabla^2\bar \ell_N(\theta)\big]&\ge \lambda_{min}\big(E_{\theta_0}\big[-\nabla^2 \ell(\theta)\big]\big)-\big\|\nabla^2 \bar\ell_N(\theta)-E_{\theta_0}\big[\nabla^2 \ell(\theta)\big]\big\|_{op}\\
&\ge c_{min}-\big\|\nabla^2\bar\ell_N(\theta)-E_{\theta_0}\big[\nabla^2 \ell(\theta)\big]\big\|_{op}.
\end{split}
\end{equation}
Hence we deduce
\begin{equation}\label{eq:inf:est}
\begin{split}
&P_{\theta_0}^N\Big(\inf_{\theta\in \mathcal B}\lambda_{min}\big[\nabla^2\ell_N(\theta, Z)\big]<Nc_{min}/2 \Big)\\
&~~~~~~~~~ \le P_{\theta_0}^N\Big(\big\|\nabla^2\bar\ell_N(\theta)-E_{\theta_0}\big[\nabla^2 \ell(\theta)\big]\big\|_{op}\ge c_{min}/2 ~\text{for some}~\theta\in\mathcal B\Big)\\
&~~~~~~~~~ \le 	P_{\theta_0}^N\Big(\sup_{\theta\in \mathcal B}\sup_{v:\|v\|_{\R^D}\le 1}\Big|v^T \Big(\nabla^2\bar \ell_N(\theta) -E_{\theta_0}[\nabla^2\ell(\theta)]\Big) v \Big|\ge c_{min}/2 \Big)\\
&~~~~~~~~~ =P_{\theta_0}^N\Big(\sup_{\theta\in \mathcal B}\sup_{v:\|v\|_{\R^D}\le 1}\big|P_N(g_{v,\theta})\big|\ge c_{min}/2 \Big),
\end{split}
\end{equation}
where
\[g_{v,\theta}(\cdot):= v^T \Big(\nabla^2\ell(\theta,\cdot) -E_{\theta_0}[\nabla^2\ell(\theta)]\Big) v, ~~~~ v\in \R^D.\]
\par
The next step is to reduce the supremum over $\{v:\|v\|_{\R^D}\le1 \}$ to a suitable finite maximum over grid points $v_i$ by a contraction argument (commonly used in high-dimensional probability). For $\rho>0$, let $N(\rho)$ denote the minimal number of balls of $\|\cdot\|_{\R^D}-$radius $\rho$ required to cover $\{v:\|v\|_{\R^D}\le1 \}$, and let $v_i, \|v_i\|_{\R^D}\le 1$, be the centre points of a minimal covering. Thus for any $v\in \R^D$ there exists an index $i$ such that $\|v-v_i\|_{\R^D}\le \rho$. Hence, writing shorthand \[ M_\theta= \nabla^2\bar \ell_N(\theta)-E_{\theta_0} [\nabla^2 \ell(\theta)], ~~ \theta\in \mathcal B,\]
we have by the Cauchy-Schwarz inequality and the symmetry of the matrix $M_\theta$,
\begin{equation*}
\begin{split}
v^TM_\theta v&= v_i^TM_\theta v_i+ (v-v_i)^TM_\theta v+v_i^TM_\theta(v-v_i)\\
&\le v_i^TM_\theta v_i+ \|v-v_i\|_{\R^D} \|M_\theta v\|_{\R^D}+\|v-v_i\|_{\R^D}\|M_\theta v_i\|_{\R^D}\\
&\le v_i^TM_\theta v_i + 2\rho \sup_{v:\|v\|_{\R^D}\le 1 }v^TM_\theta v.
\end{split}
\end{equation*}
Choosing $\rho = \frac 14$ and taking suprema it follows that for any $\theta\in \mathcal B$,
\begin{equation}\label{eq:bilin:reduction}
\sup_{v:\|v\|_{\R^D}\le 1 }v^TM_\theta v\le 2\max_{i=1,...,N(1/4)}v_i^TM_\theta v_i. 
\end{equation}
Since the covering $(v_i)$ is independent of $\theta$, we can further estimate the right hand side of (\ref{eq:inf:est}) by a union bound to the effect that
\begin{equation}\label{eq:unionbd}
\begin{split}
&P_{\theta_0}^N\Big( \sup_{\theta\in \mathcal B}\sup_{v:\|v\|_{\R^D}\le 1}\Big|v^TM_\theta v\Big|\ge c_{min}/2 \Big)\\
&\le N(1/4)\cdot \sup_{v:\|v\|_{\R^D}\le 1}P_{\theta_0}^N\Big( \sup_{\theta\in \mathcal B}\Big|v^TM_\theta v\Big|\ge c_{min}/4 \Big)\\
%&~~= N(1/4)\cdot \sup_{v:\|v\|_{\R^D}\le 1}P_{\theta_0}^N\Big( \sup_{\theta\in \mathcal B}\Big|P_{N}(g_{v,\theta})\Big|\ge c_{min}/4 \Big)\\
&\le N(1/4)\cdot \sup_{v:\|v\|_{\R^D}\le 1}\Big[P_{\theta_0}^N\Big( \sup_{\theta\in \mathcal B}\Big|P_{N}(g_{v,\theta}-g_{v,\theta^*})\Big|\ge c_{min}/8 \Big)+P_{\theta_0}^N\big( \big|P_{N}(g_{v,\theta^*})\big|\ge c_{min}/8 \big)\Big],
\end{split}
\end{equation}
where we recall that $\theta^*$ is the centrepoint of the set $\mathcal B$ from (\ref{eq:B:ass}). For the rest of the proof, we fix any $v\in \R^D$ with $\|v\|_{\R^D}=1$. Next, we use (\ref{eq:llhderivs}) to decompose the `uncentred' part of $g_{v,\theta}$ as
\begin{equation*}
\begin{split}
-v^T \nabla^2 \ell(\theta,Z)v &=v^T\Big[ \nabla \mathcal G^{X}( \theta)\nabla \mathcal G^{X}( \theta)^T + \big[ \mathcal G^{X}( \theta)- \mathcal G^{X}( \theta_0) \big]\nabla^2 \mathcal G^{X}( \theta) \Big]v -\eps v^T \nabla^2\mathcal G^X(\theta) v\\
&=: \tilde g^I_{v,\theta}(X) + \eps g^{II}_{v,\theta}(X),
\end{split}
\end{equation*}
such that
\[ g_{v,\theta}(z)=  g^I_{v,\theta}(x) + \eps g^{II}_{v,\theta}(x),\]
where we have defined the centred version of $\tilde g^I_{v,\theta}$ as
\[ g^I_{v,\theta}(x)=\tilde g^I_{v,\theta}(x) -E_{\theta_0}[\tilde g^I_{v,\theta}(X)],~~x \in \mathcal O.\]
We can therefore bound the right hand side of (\ref{eq:unionbd}) by
\begin{equation*}
	\begin{split}
	N\Big(\frac{1}{4}\Big)\cdot &\sup_{v:\|v\|_{\R^D}\le 1}\bigg[P_{\theta_0}^N\Big( \sup_{\theta\in \mathcal B}\big|\frac 1N\sum_{i=1}^N(g_{v,\theta}^I-g_{v,\theta^*}^I)(X_i)\big|\ge \frac{c_{min}}{16} \Big)+P_{\theta_0}^N\Big(\big| \frac 1N\sum_{i=1}^Ng_{v,\theta^*}^I(X_i)\big|\ge \frac{c_{min}}{16} \Big)\\
	+&P_{\theta_0}^N\Big( \sup_{\theta\in \mathcal B}\big|\frac 1N\sum_{i=1}^N\eps_i(g_{v,\theta}^{II}-g_{v,\theta^*}^{II})(X_i)\big|\ge \frac{c_{min}}{16} \Big)+P_{\theta_0}^N\Big( \big|\frac 1N\sum_{i=1}^N\eps_ig_{v,\theta^*}^{II}(X_i)\big|\ge \frac{c_{min}}{16} \Big)\bigg]\\
	=:&N(1/4)\cdot (i+ii+iii+iv).
	\end{split}
\end{equation*}
We now use empirical process techniques (Lemma \ref{mixchain} and also Hoeffding's inequality) to bound the preceding probabilities.
\par
\textbf{Terms $i$ and $ii$.} In order to apply Lemma \ref{mixchain} to term $i$, we require some preparations. By the definition of $\tilde g_{v,\theta}^I$ and of the operator norm $\|\cdot\|_{op}$, using the elementary identity $v^T(aa^T-bb^T)v=v^T(a+b)(a-b)^Tv$ for any $v,a,b\in \R^D$ and
Assumption \ref{ass:model}, we have that for any $\theta,\bar\theta\in\mathcal{B}$,
\begin{equation}\label{ohlord}
\begin{split}
\| \tilde g^I_{v,\theta}-\tilde g^I_{v,\bar \theta}\|_{\infty} &\le \Big\| \big[\nabla \mathcal G( \theta)\nabla \mathcal G( \theta)^T + \big[ \mathcal G( \theta)- \mathcal G(\theta_0) \big]\nabla^2 \mathcal G( \theta) \big]\\
&~~~~~~~~~~~~~~~~~~-\big[\nabla \mathcal G(\bar \theta)\nabla \mathcal G(\bar \theta)^T + \big[ \mathcal G(\bar \theta)- \mathcal G(\theta_0) \big]\nabla^2 \mathcal G(\bar \theta) \big]\Big\|_{L^\infty(\mathcal O, \R^{D\times D})}\\
&\le \Big\|  \big[  \nabla \mathcal G(\theta)-\nabla \mathcal G(\bar\theta)   \big] \big[ \nabla \mathcal G(\theta)+\nabla \mathcal G(\bar\theta) \big]^T \Big\|_{L^\infty(\mathcal O, \R^{D\times D})}\\
&~~~~~~~~~~~~~+ \Big\|\big[ \mathcal G(\theta)-\mathcal G(\bar\theta) \big] \nabla ^2\mathcal G(\theta) \Big\|_{L^\infty(\mathcal O, \R^{D\times D})}\\
&~~~~~~~~~~~~~+ \Big\|\big[ \mathcal G(\bar\theta)-\mathcal G(\theta_0) \big] \big[ \nabla ^2\mathcal G(\theta) -\nabla ^2\mathcal G(\bar\theta) \big] \Big\|_{L^\infty(\mathcal O, \R^{D\times D})}\\
&\le 2m_1k_1|\theta-\bar\theta|_1 + m_0k_2|\theta-\bar\theta|_1 + m_2k_0|\theta-\bar\theta|_1\\
&\le 2C_{\mathcal G}|\theta-\bar\theta|_1.
\end{split}
\end{equation}
In particular, by (\ref{eq:B:ass}) we obtain the uniform bound
\begin{equation}\label{everlasting}
\begin{split}
\sup_{\theta\in \mathcal B}\|g^I_{v,\theta}-g^I_{v,\theta^*}\|_\infty &\le 2\sup_{\theta\in \mathcal B}\| \tilde g^I_{v,\theta}(X) -\tilde g^I_{v,\theta^*}\|_{\infty}\le 4C_{\mathcal G}|\theta-\theta^*|_1\le 4C_{\mathcal G}\eta.
\end{split}
\end{equation}
We introduce the rescaled function class
\[ h^I_\theta:=\frac{g^I_{v,\theta}-g^I_{v,\theta^*}}{16C_{\mathcal G}\eta},~~~~\mathcal H^I=\{h^I_\theta :\theta \in \mathcal B \}, \]
which has envelope and variance proxy bounded as
\begin{equation}\label{stamp}
	\sup_{\theta\in \mathcal B}\|h^I_{\theta}\|_\infty\le 1/4\equiv U, ~~~~\sup_{\theta\in \mathcal B}\big(E_{\theta_0} \big[h^I_{\theta}(X)^2\big]\big)^{\frac 12}\le 1/4\equiv \sigma.
\end{equation}
Next, if
\[d_2^2(\theta,\bar \theta)=  E_{\theta_0}\big[(h^I_{\theta}(X)-h^I_{\bar \theta}(X))^2\big],~~d_\infty(\theta,\bar \theta)=\|h^I_{\theta}-h^I_{\bar \theta}\|_{\infty}, ~~~~\theta,\bar\theta\in \mathcal B, \]
then using (\ref{ohlord}) we have that
\[ d_2(\theta,\bar\theta)\le d_\infty(\theta,\bar\theta)\le |\theta-\bar\theta|_1/\eta,~~~~\theta,\bar\theta\in \mathcal B. \]
Thus for any $\rho\in (0,1)$, using Proposition 4.3.34 in \cite{GN16}, we obtain that
\begin{equation}\label{eq:covernumber}
N\big( \mathcal H^I, d_2, \rho \big) \le N\big( \mathcal H^I, d_\infty, \rho \big) \le N\big( \mathcal B, |\cdot|_1/\eta,\rho \big) \le (3/\rho)^D.
\end{equation}
For any $A\ge 2$ we have
\begin{equation*}
\begin{split}
\int_0^1 \log (A/x) dx  = \log (A) + 1, ~~~~~~   \int_0^1 \sqrt{\log (A/x)} dx\le \frac{2\log A}{2\log A-1}\sqrt{\log (A)},
\end{split}
\end{equation*}
[see p.190 of \cite{GN16} for the latter inequality], and hence, using this for $A= 3$, we can respectively bound the $L^\infty$ and $L^2$ metric entropy integrals of $\mathcal H^I$ by
\begin{equation*}
\begin{split}
\mathcal J_\infty(\mathcal H^I)&=\int_0^{4U}\log N(\mathcal H^I, d_\infty, \rho)d\rho \lesssim D, \\
J_2(\mathcal H^I)&=  \int_0^{4\sigma}\sqrt{\log N(\mathcal H^I, d_2, \rho)}d\rho\lesssim \sqrt D.
\end{split}
\end{equation*}
Now, an application of Lemma \ref{mixchain} below implies that for any $x\ge 1$ and some universal constant $L'>0$, we have that
\begin{equation}\label{chainpower}
P_{\theta_0}^N\Big( \sup_{\theta\in\mathcal B} \frac{1}{\sqrt N}\Big| \sum_{i=1}^N h^I_\theta(X_i)\Big|\ge L'\Big[ \sqrt D +\sqrt x+(D+x)/\sqrt{N} \Big] \Big)\le 2e^{-x}.
\end{equation}
We also have by the definition of $g^I_{v,\theta^*}$ that \[\|g^I_{v,\theta^*}\|_\infty \le 2\|\tilde g^I_{v,\theta^*}\|_\infty \le 2(k_1^2+k_0k_2),\] and hence by Hoeffding's inequality (Theorem 3.1.2 in \cite{GN16}) that
\begin{equation}\label{eq:hoeff1}
ii \le 2\exp\Big(-\frac{2Nc_{min}^2}{256\cdot 4(k_1^2+k_0k_2)^2}\Big)\le2\exp\Big(-\frac{Nc_{min}^2}{512C_\mathcal G'^2}\Big).
\end{equation}

\par

Now if we define 
\begin{equation}\label{eq:Dcond}
\mathcal R_N^{2,I}:= CN\min\bigg\{  \frac{c_{min}^2}{C_{\mathcal G}^2\eta^2},\frac{c_{min}}{C_{\mathcal G}\eta}, \frac{c_{min}^2}{C_{\mathcal G}'^2}\bigg\},
\end{equation}
then for any $D\le \mathcal R_N^{2,I}$ and choosing $x=4 \mathcal R_N^{2,I}$ we have
\begin{equation*}
\begin{split}
L'\Big[\sqrt  D +\sqrt x+(D+x)/\sqrt{N} \Big]\le \frac{c_{min}\sqrt N}{256C_{\mathcal G}\eta}, ~~~~4 \mathcal R_N^{2,I}\le \frac{Nc_{min}^2}{512C_\mathcal G'^2},
\end{split}
\end{equation*}
whenever $C>0$ is small enough. Therefore, combining (\ref{chainpower}) and (\ref{eq:hoeff1}), and using the definitions of the term $i$ and of $h^I_{\theta}$, we obtain
\begin{equation}\label{eq:1stterm}
\begin{split}
ii+i\le 2e^{-4\mathcal R_N^{2,I}}+ P_{\theta_0}^N\Big( \sup_{\theta\in\mathcal B} \frac{1}{\sqrt N}\Big| \sum_{i=1}^N h^I_\theta(X_i)\Big|\ge \frac{c_{min}\sqrt N}{256C_{\mathcal G}\eta}\Big) \le 4e^{-4\mathcal R_N^{2,I}}.
\end{split}
\end{equation}
\par

\smallskip

\par 
\textbf{Terms $iii$ and $iv$.} Let us now treat the empirical process indexed by the functions $\{g_{v,\theta}^{II}:\theta\in\mathcal B \}$. Since $\|v\|_{\R^D}\le 1$, we have for any $\theta, \bar \theta \in \mathcal B$,
\[ \| g^{II}_{v,\theta}-g^{II}_{v,\bar \theta} \|_{\infty}\le \| \nabla^2\mathcal G(\theta)-\nabla^2\mathcal G(\bar\theta) \|_{L^\infty(\mathcal O,\R^{D\times D})} \le m_2|\theta-\bar\theta|_1,\]
which also yields the envelope bound
\[ \sup_{\theta\in\mathcal B} \big\|g^{II}_{v,\theta}-g^{II}_{v,\theta^*}\big\|_{\infty}\le m_2\sup_{\theta\in\mathcal B} |\theta-\theta^*|_{1} \le m_2\eta. \]
Now the rescaled function class 
\[ h^{II}_\theta:=\frac{g^{II}_{v,\theta}-g^{II}_{v,\theta^*}}{4m_2\eta},~~~~\mathcal H^{II}=\{h^{II}_\theta :\theta \in \mathcal B \}, \]
admits envelopes 
\[ \sup_{\theta\in \mathcal B}\|h^{II}_{v,\theta}\|_\infty \le 1/4\equiv U, ~~~~\sup_{\theta\in \mathcal B}\big(E_{\theta_0}\big[h^{II}_{v,\theta}(X)^2\big]\big)^{\frac 12}\le 1/4\equiv \sigma. \]
Thus defining 
\[d_2^2(\theta,\bar \theta):= E_{\theta_0} \big[(h^{II}_{v,\theta}(X)-h^{II}_{v,\bar \theta}(X))^2\big],~~~~d_\infty(\theta,\bar \theta)= \|h^{II}_{v,\theta}-h^{II}_{v,\bar \theta}\|_{\infty},~~~~\theta,\bar\theta\in\mathcal B\]
we have
\[d_2(\theta,\bar \theta)\le d_\infty(\theta,\bar \theta) \le |\theta-\bar\theta|_1/\eta, ~~~~\theta,\bar\theta\in\mathcal B.\]
Therefore, just as with the bounds obtained for term $i$, we have $N\big( \mathcal H^{II}, d_2, \rho \big)\le N\big( \mathcal H^{II}, d_\infty, \rho \big) \le(3/\rho)^D$ and thus, by Lemma \ref{mixchain} below, 
\begin{equation}\label{theforce}
P_{\theta_0}^N\Big( \sup_{\theta\in\mathcal B} \frac{1}{\sqrt N}\Big| \sum_{i=1}^N \eps_ih^{II}_\theta(X_i)\Big|\ge L'\Big[ \sqrt D +\sqrt x+(D+x)/\sqrt{N} \Big] \Big)\le 2e^{-x},~~~ x\ge 1.
\end{equation}
Moreover, by the hypotheses, $\|g^{II}_{v,\theta^*}\|_\infty  \le k_2$,
and hence, invoking the Bernstein inequality (\ref{bernstein}) with $U= \sigma \equiv k_2$, we obtain that
\begin{equation}\label{mineral}
	P_{\theta_0}^N\Big( \Big| \frac{1}{\sqrt N}\sum_{i=1}^N \eps_ig^{II}_{v,\theta^*}(X_i) \Big|\ge k_2\sqrt {2x}+\frac{k_2x}{3\sqrt N} \Big)\le 2e^{-x},~~ x>0.
\end{equation}
We can now set
\[ \mathcal R_N^{2,II}:=CN\min\bigg\{  \frac{c_{min}^2}{m_2^2\eta^2},\frac{c_{min}}{m_2\eta}, \frac{c_{min}^2}{k_2^2},\frac{c_{min}}{k_2} \bigg\}, \]
and choosing $x=4\mathcal R_N^{2,II}$ in the preceding displays, we obtain that for $C>0$ small enough and any $D\le \mathcal R_N^{2,II}$,
\begin{equation}\label{eq:2ndterm}
\begin{split}
iii+iv&\le P_{\theta_0}^N\Big( \sup_{\theta\in\mathcal B} \frac{1}{\sqrt N}\Big| \sum_{i=1}^N \eps_i h^{II}_\theta(X_i)\Big|\ge \frac{c_{min}\sqrt N}{64m_2\eta}\Big)+P_{\theta_0}^N\Big( \Big| \frac{1}{\sqrt N}\sum_{i=1}^N \eps_ig^{II}_{v,\theta^*}(X_i) \Big|\ge \frac{c_{min}\sqrt N}{16} \Big) \\
&\le 4e^{-4\mathcal R_N^{2,II}}.
\end{split}
\end{equation}

\par 
\textbf{Combining the terms.} By combining the bounds (\ref{eq:inf:est}), (\ref{eq:unionbd}), (\ref{eq:1stterm}), (\ref{eq:2ndterm}) and using that $N(1/4)\le 9^D\le e^{3D}$ (cf. Proposition 4.3.34 in \cite{GN16}) we obtain that since $D\le \mathcal R_N\le \min ( \mathcal R_N^{2,I},\mathcal R_N^{2,II})$ from (\ref{eq:itis what itis}),
\begin{equation*}
\begin{split}
P_{\theta_0}^N\Big(\inf_{\theta\in \mathcal B}\lambda_{min}\big(-\nabla^2\ell_N(\theta, Z)\big)<Nc_{min}/2 \Big)&\le N(1/4)\cdot (i+ii+iii+iv)\\
&\le 4 e^{3D-4\mathcal R_N^{2,I}}+4e^{3D-4\mathcal R_N^{2,II}}\le 8e^{-\mathcal R_N},
\end{split}
\end{equation*}
completing the proof of (\ref{eq:emp:lb}). \hfill $\square$

\subsubsection{Proof of (\ref{eq:emp:ub})}
We derive probability bounds for each of the three terms in (\ref{eq:emp:ub}) separately. The general scheme of proof for each of the three bounds is similar to the proof of (\ref{eq:emp:lb}), and we condense some of the steps to follow.

\smallskip

\textbf{Second order term.} Using that $c_{max}\ge c_{min}$, we can replace (\ref{eq:inf:est}) by
\begin{equation*}
\begin{split}
P_{\theta_0}^N\Big(\sup_{\theta\in \mathcal B}\lambda_{max}\big[-\nabla^2\ell_N(\theta, Z)\big]\ge 3Nc_{max}/2 \Big)\le P_{\theta_0}^N\Big(\sup_{\theta\in \mathcal B}\sup_{v:\|v\|_{\R^D}\le 1}\big|P_N(g_{v,\theta})\big|\ge c_{min}/2 \Big).
\end{split}
\end{equation*}
From here onwards, this term can be treated exactly as in the proof of (\ref{eq:emp:lb}) and thus, for $D\le \mathcal R_n$ from (\ref{eq:itis what itis}), we deduce
\begin{equation}\label{eq:ub1}
P_{\theta_0}^N\Big(\sup_{\theta\in \mathcal B}\lambda_{max}\big[-\nabla^2\ell_N(\theta, Z)\big]\ge 3Nc_{max}/2 \Big)\le 8e^{-\mathcal R_N}.
\end{equation}

\smallskip 

\textbf{First order term.} First, let us denote
\[ f_{v,\theta}(z):= v^T \Big(\nabla\ell(\theta,z)-E_{\theta_0}[\nabla \ell(\theta,Z)]\Big),~ \|v\|_{\mathbb R^D} \le 1, \theta \in \mathcal B,\]
and let $(v_i:i=1,...,N({1/2}))$ be the centre points of a $\|\cdot\|_{\R^D}$-covering with balls of radius $1/2$, of the unit ball $\{\theta:\|\theta\|_{\R^D}\le 1 \}$. Then for any $v$ there exists $v_i$ such that $\|v-v_i\|_{\R^D}\le 1/2$ so that by the Cauchy-Schwarz inequality,
\begin{equation*}
\begin{split}
|P_N(f_{v,\theta})|&\le |P_N(f_{v,\theta}-f_{v_i,\theta})|+|P_N(f_{v_i,\theta})|\\
&\le \|v-v_i\|_{\R^D}\big\|\nabla \bar{\ell}_N(\theta)- E_{\theta_0} \big[\nabla {\ell}(\theta) \big] \big\|_{\R^D} + |P_N(f_{v_i,\theta})|\\
&\le \frac 12\big\|\nabla \bar{\ell}_N(\theta)- E_{\theta_0} \big[\nabla \ell (\theta) \big] \big\|_{\R^D} + |P_N(f_{v_i,\theta})|.
\end{split}
\end{equation*}
Therefore, since $\|u\|_{\R^D}=\sup_{v:\|v\|_{\R^D}\le 1}|v^Tu|$ for any $u\in \R^D$, we deduce for any $\theta \in \mathcal B$,
\begin{equation}\label{eq:lin:reduction}
\sup_{v:\|v\|_{\R^D}\le 1} |P_N(f_{v,\theta})|\le 2 \max_{1\le i\le N(1/2)}|P_N(f_{v_i,\theta})|.
\end{equation}
We can hence estimate
\begin{equation}\label{union}
\begin{split}
P_{\theta_0}^N\Big( \sup_{\theta\in\mathcal B}\|\nabla \bar\ell_N(\theta)\|_{\R^D}\ge 3c_{max}/2 \Big)&\le P_{\theta_0}^N\Big( \sup_{\theta\in\mathcal B}\sup_{v:\|v\|_{\R^D}\le 1 } \big|v^T\big[\nabla \bar\ell_N(\theta)-E_{\theta_0}[\nabla \ell (\theta)]\big]\big|\ge c_{max}/2 \Big)\\
&\le N(1/2)\cdot \sup_{v:\|v\|_{\R^D}\le 1} P_{\theta_0}^N\big( \sup_{\theta\in\mathcal B} \big| P_N\big(f_{v,\theta} \big)\big|\ge c_{max}/4 \big).
\end{split}
\end{equation}
We fix $v\in \R^D$ with $\|v\|_{\R^D}\le 1$. Using (\ref{eq:llhderivs}), by decomposing the `uncentred' part of $f_{v,\theta}$ into
\begin{equation*}
\begin{split}
v^T\nabla \ell(\theta,Z)= v^T\nabla \mathcal G^X(\theta)\big[ \mathcal G^X(\theta)-\mathcal G(\theta_0)\big]-\eps v^T \nabla \mathcal G^X(\theta)=:\tilde f_{v,\theta}^I(X)-\eps f_{v,\theta}^{II}(X),
\end{split}
\end{equation*}
we can then write
\[ f_{v,\theta}(z)= f_{v,\theta}^I(x) +\eps f_{v,\theta}^{II}(x), \]
where we have further defined $f_{v,\theta}^I(x):=\tilde f_{v,\theta}^I(x)- E_{\theta_0}[\tilde f_{v,\theta}^I(X)]$ (and still write $P_N(f_{v,\theta}^I)=N^{-1}\sum_{i=1}^N f_{v,\theta}^I(X_i)$ to expedite notation). We then estimate the probability on the right hand side of (\ref{union}) as follows,
\begin{equation}\label{berlioz}
	\begin{split}
		P_{\theta_0}^N\big( \sup_{\theta\in\mathcal B} \big| P_N\big(f_{v,\theta} \big)\big|\ge &c_{max}/4 \big) \le P_{\theta_0}^N\big( \sup_{\theta\in\mathcal B} \big| P_N\big(f_{v,\theta}^I -f_{v,\theta^*}^I \big)\big|\ge c_{max}/16 \big)+P_{\theta_0}^N\big(\big| P_N\big(f_{v,\theta^*}^I \big)\big|\ge c_{max}/16 \big)\\
		&+P_{\theta_0}^N\big( \sup_{\theta\in\mathcal B} \big| P_N\big(f_{v,\theta}^{II} -f_{v,\theta^*}^{II} \big)\big|\ge c_{max}/16 \big)+P_{\theta_0}^N\big(\big| P_N\big(f_{v,\theta^*}^{II} \big)\big|\ge c_{max}/16 \big)\\
		&~~~~~~~~~~=: i+ii+iii+iv.
	\end{split}
\end{equation}

\smallskip

We first treat the terms $i$ and $ii$. By the definition of $\tilde f_{v,\theta}^I$ and Assumption \ref{ass:model}, we have that for any $\theta,\bar\theta\in\mathcal B$,
\begin{equation*}
\begin{split}
\big\|\tilde f^I_{v,\theta}-\tilde f^I_{v,\bar \theta}\big\|_{\infty}&\le \big\| \big[\nabla \mathcal G(\theta)-\nabla \mathcal G(\bar\theta)\big]\big[ \mathcal G(\theta)-\mathcal G(\theta_0) \big] +\nabla \mathcal G(\bar\theta)\big[ \mathcal G(\theta)-\mathcal G(\bar\theta)\big]\big\|_{L^\infty(\mathcal O,\R^D)}\\
&\le (k_0m_1+k_1m_0)|\theta-\bar\theta|_1.
\end{split}
\end{equation*}
Again using Assumption \ref{ass:model}, we also have
\[ \sup_{\theta\in\mathcal B}\big\|\tilde f^I_{v,\theta}-\tilde f^I_{v,\theta^*}\big\|_{\infty}\le (k_0m_1+k_1m_0)\eta .\]
Moreover, using that $\|f_{v,\theta^*}^I\|_\infty\le 2k_0k_1$, Hoeffding's inequality yields that 
\begin{equation*}
ii \le 2\exp\Big(-\frac{Nc_{max}^2}{512k_0^2k_1^2}\Big).
\end{equation*}
Therefore, by using Lemma \ref{mixchain} in the same manner as in (\ref{chainpower}), we obtain that the rescaled process 
\[ h_{v,\theta}^I:=\frac{\tilde f^I_{v,\theta}-\tilde f^I_{v,\theta^*}}{8(k_0m_1+k_1m_0)\eta }\]
satisfies
\begin{equation}\label{chain2}
P_{\theta_0}^N\Big( \sup_{\theta\in\mathcal B} \frac{1}{\sqrt N}\Big| \sum_{i=1}^N h^I_\theta(X_i)\Big|\ge L'\Big[ \sqrt D +\sqrt x+(D+x)/\sqrt{N} \Big] \Big)\le 2e^{-x}, ~~~ x\ge 1.
\end{equation}
Thus, setting
\[ \mathcal R_N^{1,I}=:CN\min\Big\{ \frac{c_{max}^2}{(k_0m_1+k_1m_0)^2\eta^2},\frac{c_{max}}{(k_0m_1+k_1m_0)\eta },\frac{c_{max}^2}{k_0^2k_1^2}\Big\},\]
and choosing $x=3\mathcal R_N^{1,I}$ in (\ref{chain2}), we obtain that for $C>0$ small enough and any $D\le\mathcal R_N^{1,I}$,
\begin{equation}\label{zumersten}
	ii+i\le 2e^{-3\mathcal R_N^{1,I}}+P_{\theta_0}^N\Big( \Big|\frac{1}{\sqrt N}\sum_{i=1}^N h_{v,\theta}^I(X_i)\Big|\ge \frac{c_{max}\sqrt N}{128(k_0m_1+k_1m_0)\eta} \Big)\le 4e^{-3\mathcal R_N^{1,I}}.
\end{equation}

\smallskip

We now treat the terms $iii$ and $iv$. As $\|v\|_{\R^D}\le 1$, we have that for any $\theta,\bar\theta\in\mathcal B$,
\[\|f_{v,\theta}^{II}- f_{v,\bar \theta}^{II} \|_\infty \le m_1|\theta-\bar\theta|_1,~~~\|f_{v,\theta}^{II}- f_{v,\theta^*}^{II} \|_\infty \le m_1\eta ,~~~ \|f_{v,\theta^*}^{II}\|_\infty\le k_1. \]
Therefore, by utilising the Lemma \ref{mixchain} below as well as Bernstein's inequality (\ref{bernstein}) in precisely the same manner as in the derivations of (\ref{theforce}) and (\ref{mineral}) respectively, we obtain the two inequalities
\begin{equation*}
P_{\theta_0}^N\Big( \sup_{\theta\in\mathcal B} \frac{1}{\sqrt N}\Big| \sum_{i=1}^N \eps_i   \frac{f_{v,\theta}^{II}(X_i)- f_{v,\theta^*}^{II}(X_i)}{4m_1\eta }\Big|\ge L'\Big[ \sqrt D +\sqrt x+(D+x)/\sqrt{N} \Big] \Big)\le 2e^{-x},~~~ x\ge 1,
\end{equation*}
and
\begin{equation*}
P_{\theta_0}^N\Big( \Big| \frac{1}{\sqrt N}\sum_{i=1}^N \eps_if^{II}_{v,\theta^*}(X_i) \Big|\ge k_1\sqrt {2x}+\frac{k_1x}{3\sqrt N} \Big)\le 2e^{-x},~~ x>0.
\end{equation*}
Thus, if we set
\[\mathcal R_N^{1,II}:=CN\min\Big\{ \frac{c_{max}^2}{m_1^2\eta^2},\frac{c_{max}}{m_1\eta },\frac{c_{max}^2}{k_1^2},\frac{c_{max}}{k_1}\Big\},\]
then for $C>0$ small enough, for any $D\le 3\mathcal R_N^{1,II}$ and choosing $x=3\mathcal R_N^{1,II}$ in the preceding displays, we obtain
\begin{equation}\label{zumzweiten}
	iii+iv\le 4e^{-3\mathcal R_N^{1,II}}.
\end{equation}

\smallskip

By combining (\ref{union}), (\ref{berlioz}), (\ref{zumersten}), (\ref{zumzweiten}), using that $N(1/2)\le e^{2D}$ (cf. Proposition 4.3.34 in \cite{GN16}) and since $D\le \mathcal R_N\le \min ( \mathcal R_N^{1,I},\mathcal R_N^{1,II})$, we conclude that

\begin{equation}\label{eq:ub2}
\begin{split}
P_{\theta_0}^N\Big( \sup_{\theta\in\mathcal B}\|\nabla \bar\ell_N(\theta)\|_{\R^D}\ge 3c_{max}/2 \Big)&\le N(1/2)\cdot (i+ii+iii+iv)\\
&\le 4e^{2D-3\mathcal R_N^{1,I}}+4e^{2D-3\mathcal R_N^{1,II}} \le 8e^{-\mathcal R_N}.
\end{split}
\end{equation}

\smallskip

\textbf{Order zero term.} As with the previous terms, we introduce a decomposition
\begin{equation*}
\begin{split}
-\ell (\theta,Z)&=\frac 12\big[\mathcal G^X(\theta_0)-\mathcal G^X(\theta)\big]^2-\eps \big[\mathcal G^X(\theta_0)-\mathcal G^X(\theta)\big]+\frac {\eps^2}2\\
&=:\tilde l^I_{\theta}(X)+\eps l^{II}_{\theta}(X)+\frac {\eps^2}2,
\end{split}
\end{equation*}
and therefore, defining
\begin{equation*}
l^I_{\theta}(x)=:\tilde l^I_{\theta}(x)-E_{\theta_0}[\tilde l^I_{\theta}(X)],~~~ x\in\mathcal O,
\end{equation*}
we have that
\[ -\ell (\theta,Z)+E_{\theta_0}[\ell (\theta)]=l^I_{\theta}(X)+\eps l^{II}_{\theta}(X)+\frac {\eps^2}2. \]
Then, using Assumption \ref{ass:geom}, we can estimate
\begin{equation*}
\begin{split}
	P_{\theta_0}^N&\big(\sup_{\theta\in\mathcal B} \big|\bar \ell_N(\theta,Z)\big|\ge 2c_{max}+1\big)\\
	&\le P_{\theta_0}^N\big(\sup_{\theta\in\mathcal B} \big|\bar \ell_N(\theta,Z)-E_{\theta_0}[\ell (\theta,Z)]\big|\ge c_{max}+1\big)\\
	&\le P_{\theta_0}^N\big(\sup_{\theta\in\mathcal B} \big|P_N(l^I_{\theta}-l^I_{\theta^*})\big|\ge \frac{c_{max}}{4}\big) +P_{\theta_0}^N\big(\sup_{\theta\in\mathcal B} \big|P_N(l^I_{\theta^*})\big|\ge \frac{c_{max}}{4}\big) \\
	&~~+P_{\theta_0}^N\big(\sup_{\theta\in\mathcal B} \big|P_N(l^{II}_{\theta}-l^{II}_{\theta^*})\big|\ge \frac{c_{max}}{4}\big)+ P_{\theta_0}^N\big(\sup_{\theta\in\mathcal B} \big|P_N(l^{II}_{\theta^*})\big|\ge \frac{c_{max}}{4}\big) \\
	&~~+P_{\theta_0}^N\Big(\frac 1{2N}\sum_{i=1}^N \eps_i^2\ge 1\Big) =:i+ii+iii+iv+v.
\end{split}
\end{equation*}
To bound the preceding terms, we use Assumption \ref{ass:model} to deduce that for all $\theta,\bar\theta\in\mathcal B$,
\begin{equation*}
\begin{split}
\| l_\theta^I- l_{\bar \theta}^I\|_{\infty}\le 2\|\tilde l_\theta^I-\tilde l_{\bar \theta}^I\|_{\infty}&=\big\|-2\mathcal G(\theta_0)\big[\mathcal G (\theta) -\mathcal G (\bar \theta)\big]+\mathcal G (\theta)^2-\mathcal G (\bar \theta)^2 \big\|_\infty \\
&=\big\|\big[ (\mathcal G (\theta)-\mathcal G(\theta_0)) +(\mathcal G (\bar \theta)-\mathcal G(\theta_0))\big] \big[\mathcal G (\theta)-\mathcal G (\bar \theta)\big]\big\|_\infty \\
&\le 2k_0m_0|\theta-\bar\theta|_1,
\end{split}
\end{equation*}
as well as
\begin{equation*}
\begin{split}
\sup_{\theta\in\mathcal B} \|l_\theta^I-l_{\theta^*}^I\|_{\infty} \le 2k_0m_0\eta ,~~~ \|l_{\theta^*}^I\|_{\infty} \le k_0^2.
\end{split}
\end{equation*}
Moreover, again by Assumption \ref{ass:model} we have that for all $\theta,\bar\theta\in\mathcal B$,
\[ \| l_\theta^{II}- l_{\bar \theta}^{II}\|_{\infty}\le 2m_0|\theta-\bar\theta|_1,~~~~ \sup_{\theta\in\mathcal B}\|l_\theta^{II}-l_{ \theta^*}^{II}\|_{\infty} \le 2m_0\eta,~~~~ \|l_{ \theta^*}^{II}\|_{\infty}\le2 k_0. \]
Next, similarly as for the second and first order terms, in order to control the terms $i$ and $iii$ we now apply Lemma \ref{mixchain} to the rescaled empirical processes
\[ h_\theta^I:=\frac{l_\theta^I-l_{ \theta^*}^I}{8k_0m_0\eta},~~~~h_\theta^{II}:=\frac{l_\theta^{II}-l_{ \theta^*}^{II}}{8m_0\eta},\]
and in order to control the terms $ii$ and $iv$, we respectively apply Hoeffding's inequality and Bernstein's inequality (\ref{bernstein}) in the same manner as before. Overall, if we set 
\begin{equation}
\begin{split}
	\mathcal R_N^{0,I} &:= CN\min\Big\{ \frac{c_{max}^2}{k_0^2m_0^2\eta^2},\frac{c_{max}}{k_0m_0\eta },\frac{c_{max}^2}{k_0^4}\Big\},\\
	\mathcal R_N^{0,II} &:= CN\min\Big\{ \frac{c_{max}^2}{m_0^2\eta^2},\frac{c_{max}}{m_0\eta },\frac{c_{max}^2}{k_0^2},\frac{c_{max}}{k_0}\Big\},
\end{split}
\end{equation}
then for $C>0$ small enough, we obtain that for any $D\le \mathcal R_N \le \min (\mathcal R_N^{0,I},\mathcal R_N^{0,II})$,

\begin{equation*}
	\begin{split}
		i+ii+iii+iv&\le P_{\theta_0}^N\big(\sup_{\theta\in\mathcal B}\frac{1}{\sqrt N} \big|\sum_{i=1}^{N}h_\theta^I(X_i)  \big|\ge \frac{c_{max}\sqrt N}{32k_0m_0\eta }\big)+2\exp\Big(-\frac{Nc_{max}^2}{8 k_0^4}\Big)\\
		&~~~+P_{\theta_0}^N\big(\sup_{\theta\in\mathcal B}\frac{1}{\sqrt N} \big|\sum_{i=1}^{N}h_\theta^{II}(X_i)  \big|\ge \frac{c_{max}\sqrt N}{32m_0\eta }\big)+2e^{-\mathcal R_N^{0,II}}\\
		&\le 4e^{-\mathcal R_N^{0,I}}+4e^{-\mathcal R_N^{0,II}}\le 8e^{-\mathcal R_N}.
	\end{split}
\end{equation*}
Finally, we estimate the term $v$ by a standard tail inequality (see Theorem 3.1.9 in \cite{GN16}),
\[ v=P_{\theta_0}^N\Big(\sum_{i=1}^N( \eps_i^2-1)\ge N\Big)\le e^{-N/8}, \]
and thus obtain
\begin{equation}\label{eq:ub3}
	P_{\theta_0}^N\big(\sup_{\theta\in\mathcal B} \big|\bar \ell_N(\theta,Z)\big|\ge 2c_{max}+1\big)\le i+ii+iii+iv+v\le 8e^{-\mathcal R_N}+e^{-N/8}.
\end{equation}

\par
\textbf{Conclusion.} By combining (\ref{eq:ub1}), (\ref{eq:ub2}) and (\ref{eq:ub3}), the proof of (\ref{eq:emp:ub}) is completed. \hfill $\square$

\subsection{A chaining lemma for empirical processes}

The following key technical lemma is based on a chaining argument for stochastic processes with a mixed tail (cf.~Theorem 2.2.28 in Talagrand \cite{T14} and Theorem 3.5 in Dirksen \cite{D15}). For us it will be sufficient to control the `generic chaining' functionals employed in these references by suitable metric entropy integrals. For any (semi-)metric $d$ on a metric space $T$, we denote by $N=N(T,d,\rho)$ the minimal cardinality of a covering of $T$ by balls with centres $(t_i:i=1, \dots, N) \subset T$ such that for all $t \in T$ there exists $i$ such that $d(t,t_i)<\rho$. Below we require the index set $\Theta$ to be countable (to avoid measurability issues). Whenever we apply Lemma \ref{mixchain} in this article with an uncountable set $\Theta$, one can show that the supremum can be realised as one over a countable subset of it.

\begin{lem}\label{mixchain}
	Let $\Theta$ be a countable set. Suppose a class of real-valued measurable functions $$\mathcal H=\{h_\theta: \mathcal X \to \mathbb R, \theta \in \Theta\}$$ defined on a probability space $(\mathcal X, \mathcal A, P^X)$ is uniformly bounded by $U\ge \sup_\theta \|h_\theta\|_\infty$ and has variance envelope $\sigma^2 \ge \sup_\theta E^Xh_\theta^2(X)$ where $X \sim P^X$. Define metric entropy integrals $$J_2(\mathcal H) = \int_0^{4\sigma} \sqrt{\log N(\mathcal H, d_2,\rho)}d\rho,~~d_2(\theta,\theta'):=\sqrt{E^X[h_\theta(X)-h_{\theta'}(X)]^2},$$
	$$J_\infty(\mathcal H) = \int_0^{4U} \log N(\mathcal H, d_\infty,\rho) d\rho,~~d_\infty(\theta,\theta'):=\|h_\theta-h_{\theta'}\|_{\infty}.$$
	For $X_1, \dots, X_N$ drawn i.i.d.~from $P^X$ and $\varepsilon_i \sim^{iid} N(0,1)$ independent of all the $X_i$'s, consider empirical processes arising either as $$Z_{N}(\theta)=\frac{1}{\sqrt N}\sum_{i=1}^N h_\theta(X_i)\varepsilon_i,~~\theta \in \Theta,$$ or as $$Z_{N}(\theta)=\frac{1}{\sqrt N}\sum_{i=1}^N (h_\theta(X_i)-Eh_\theta(X)), ~~\theta \in \Theta.$$ We then have for some universal constant $L>0$ and all $x\ge 1$,
	$$\Pr\left(\sup_{\theta \in \Theta}|Z_{N}(\theta)| \ge L \Big[J_2(\mathcal H) + \sigma \sqrt{x} + (J_\infty(\mathcal H) +Ux)/\sqrt N   \Big] \right) \le 2e^{-x}.$$
	\end{lem}
\begin{proof}
 We only prove the case where $Z_N(\theta)= \sum_{i}h_\theta(X_i)\eps_i/\sqrt N$, the simpler case without Gaussian multipliers is proved in the same way. We will apply Theorem 3.5 in \cite{D15}, whose condition (3.8) we need to verify. First notice that for $|\lambda|<1/\|h_\theta-h_{\theta'}\|_\infty$, and $E^\varepsilon$ denoting the expectation with respect to $\varepsilon$,
	\begin{align}\label{zack}
	E\exp\big\{\lambda \varepsilon (h_\theta-h_{\theta'})(X) \big\} &\le 1 + \sum_{k=2}^\infty \frac{|\lambda|^k E^\eps | \varepsilon|^k E^X|h_\theta-h_{\theta'}|^k(X)}{k!} \notag\\
	&\le 1+ \lambda^2 E^X[h_\theta(X)-h_{\theta'}(X)]^2 \sum_{k=2}^\infty \frac{E^\varepsilon|\varepsilon|^k}{k!} \big(|\lambda| \|h_\theta-h_{\theta'}\|_\infty\big)^{k-2} \notag \\
	& \le \exp\Big\{\frac{\lambda^2d^2_2(\theta, \theta')}{1- |\lambda|d_\infty(\theta, \theta')}\Big\}
	\end{align}
	where we have used the basic fact $E^\varepsilon|\varepsilon|^k/k!\le 1$. By the i.i.d.~hypothesis we then also 
have	$$E\exp\Big\{\lambda (Z_{N}(\theta)-Z_{N}(\theta')) \Big\} \le  \exp\left\{\frac{\lambda^2 d^2_2(\theta,\theta')}{1- |\lambda| d_\infty(\theta, \theta')/\sqrt N}\right\}.$$ An application of the exponential Chebyshev inequality (and optimisation in $\lambda$, as in the proof of Proposition 3.1.8 in \cite{GN16}) then implies that condition (3.8) in \cite{D15} holds for the stochastic process $Z_N(\theta)$ with metrics $\bar d_2=2d_2$ and $\bar d_1=d_\infty/\sqrt N.$ In particular, the $\bar d_2$-diameter $\Delta_2(\mathcal H)$ of $\mathcal H$ is at most $4\sigma$ and the $\bar d_1$-diameter $\Delta_1(\mathcal H)$ of $\mathcal H$ is bounded by $4U/\sqrt N$. [These bounds are chosen so that they remain valid for the process without Gaussian multipliers as well.] Theorem 3.5 in \cite{D15} now gives, for some universal constant $M$, and any $\theta_\dagger \in \Theta$ that
	$$\Pr\left(\sup_{\theta \in \Theta}|Z_{N}(\theta)-Z_N(\theta_\dagger)| \ge M \big(\gamma_2(\mathcal H) + \gamma_1(\mathcal H) + \sigma \sqrt{x} + (U/\sqrt N)x   \big) \right) \le e^{-x}$$ where the `generic chaining' functionals $\gamma_1, \gamma_2$ are upper bounded by the respective metric entropy integrals of the metric spaces $(\mathcal H, \bar d_i), i=1,2$, up to universal constants (see (2.3) in \cite{D15}). For $\gamma_1$ also notice that a simple substitution $\rho'=\rho \sqrt N$ implies that 
	$$\int_0^{4U/\sqrt N} \log N(\mathcal H, \bar d_1, \rho)d\rho = \frac{1}{\sqrt N} \int_0^{4U} \log N(\mathcal H, d_\infty, \rho')d\rho',$$ and we hence deduce that
\begin{equation}\label{kitzloch}
\Pr\left(\sup_{\theta \in \Theta}|Z_{N}(\theta)-Z_N(\theta_\dagger)| \ge  \bar L \Big[ J_2(\mathcal H) + \sigma \sqrt{x} + (J_\infty(\mathcal H) +Ux)/\sqrt N \Big] \right) \le e^{-x}
\end{equation}
for some universal constant $\bar L$. 
	
	The preceding argument also implies the classical Bernstein inequality
	\begin{equation}\label{bernstein}
	\Pr \Big(|Z_{N}(\theta)| \ge \sigma \sqrt{2x} + \frac{Ux}{3\sqrt N}   \Big) \le 2e^{-x},~x>0,
	\end{equation}
	for any fixed $\theta \in \Theta, U \ge \|h_\theta\|_\infty$ and $\sigma^2 \ge E^Xh_\theta^2(X)$, proved as (3.24) in \cite{GN16}, using (\ref{zack}). Applying this with $\theta_\dagger$ and using (\ref{kitzloch}), the final result follows now from 
	\begin{align*}
	& \Pr\big(\sup_{\theta \in \Theta}|Z_{N}(\theta)| > 2\tau(x)\big) \le \Pr\big(\sup_{\theta \in \Theta}|Z_{N}(\theta)-Z_N(\theta_\dagger)| > \tau(x)\big) + \Pr\big(|Z_{N}(\theta_\dagger)| > \tau(x))\big)  \le 2 e^{-x},
	\end{align*}
	for any $x\ge 1$, where $\tau(x)= \bar L \big[ J_2(\mathcal H) + \sigma \sqrt{x} + (J_\infty(\mathcal H) +Ux)/\sqrt N \big]$ and $L\ge 2\bar L>0$ is large enough.
\end{proof}

\subsection{Proofs for Section \ref{subsec:genMAIN}}\label{subsec:genpfs} We apply the results from Appendix \ref{app:ULA} to $\mu = \tilde \Pi(\cdot|Z^{(N)})$.

\smallskip

\textbf{Proof of Theorem \ref{thm:gen:wass}.}
For any $\theta,\bar\theta\in \R^D$, we have for the log-prior density that
\begin{equation*}
	\begin{split}
		\|\nabla \log \pi(\theta)-\nabla \log \pi(\bar \theta)\|_{\R^D}= \|\Sigma^{-1}(\theta-\bar\theta)\|_{\R^D} &\le \lambda_{max}(\Sigma^{-1})\|\theta-\bar \theta\|_{\mathbb R^D}, \\
		\lambda_{min}(-\nabla^2 \log \pi(\theta))&\ge \lambda_{min}(\Sigma^{-1}),
	\end{split}
\end{equation*}
and for the likelihood surrogate $\tilde\ell_N$, by Proposition \ref{prop:log:conv} and on the event $\mathcal E$ from (\ref{nevergonnahappen}), that
\begin{equation*}
\begin{split}
\|\nabla \tilde\ell_N(\theta)-\nabla \tilde\ell_N(\bar \theta)\|_{\R^D}&\le 7K\lambda_{max}(M)\|\theta-\bar\theta\|_{\R^D},\\
\lambda_{min}(-\nabla^2 \tilde\ell_N(\theta))&\ge Nc_{min}/2.
\end{split}
\end{equation*}
Combining the last two displays, and on the event $\mathcal E$, we can verify Assumption \ref{ass:ULA} below for $-\log d\tilde \Pi(\cdot|Z^{(N)})$ with constants
\[ m= Nc_{min}/2+\lambda_{min}(\Sigma^{-1}),~~~~~~~\Lambda =7 K\lambda_{max}(M)+\lambda_{max}(\Sigma^{-1}).\] 
We may thus apply Proposition \ref{prop:ULA:wass} below to obtain,
\begin{equation*}
	\begin{split}
		W_2^2(\mathcal L(\vartheta_k),\Pi(\cdot|Z^{(N)}) )&\le 2W_2^2(\Pi(\cdot|Z^{(N)}) ,\tilde \Pi(\cdot|Z^{(N)}) )+ 2W_2^2(\mathcal L(\vartheta_k),\tilde \Pi(\cdot|Z^{(N)}))\\
		&\le \rho + b(\gamma)+ 4(1- m\gamma /2)^k \Big[\|\theta_{init}- \theta_{max} \|^2_{\R^D} +\frac{D}{m}\Big],
	\end{split}
\end{equation*}
where $\theta_{max}$ denotes the unique maximiser of $\log d\tilde \Pi(\cdot|Z^{(N)})$ over $\mathbb R^D$ (which exists on the event $\mathcal E_{conv}$, by virtue of strong concavity).
\par 

We conclude by an estimate for $\|\theta_{init}-  \theta_{max} \|_{\R^D}$. To start, notice that for any $\theta\in\R^D$ we have
\begin{equation}\label{Mbound}
	\begin{split}
		|\theta-\theta_{init}|_1^2=(\theta-\theta_{init})^T M (\theta-\theta_{init})\ge \lambda_{min}(M) \|\theta-\theta_{init} \|_{\R^D}^2.
	\end{split}
\end{equation}
Thus, for any $\theta\in \R^D$ with $\|\theta-\theta_{init} \|_{\R^D}^2\ge 4\eta^2/\lambda_{min}(M)$, we have that $|\theta-\theta_{init}|_1\ge 2\eta$, and therefore also that $g_\eta(\theta)\ge \big(|\theta-\theta_{init}|_1-\eta\big)^2\ge \frac 14|\theta-\theta_{init}|_1^2$. Thus, for $C$ from (\ref{eq:K:cond}) and any $\theta\in\R^D$ satisfying
\[\|\theta-\theta_{init} \|_{\R^D}^2\ge \frac{20}{C}+\frac{4\eta^2}{\lambda_{min}(M)}, \]
using (\ref{Mbound}), (\ref{eq:K:cond}) as well as the upper bound for $|\ell_N(\theta)|$ in the definition of $\mathcal E_{conv}$, we obtain
\begin{equation*}
	\begin{split}
		-\tilde \ell_N(\theta) = Kg_\eta(\theta)&\ge CN(c_{max}+1)\frac{1+\lambda_{max}(M)/\eta^2}{\lambda_{min}(M)}\cdot\frac{|\theta-\theta_{init}|_1^2}4\\
		& \ge \frac{C}{4}N(c_{max}+1)\|\theta-\theta_{init} \|_{\R^D}^2 \\
		&\ge 5N(c_{max}+1)\ge -\tilde \ell_N(\theta_{init}).
	\end{split}
\end{equation*}
This implies that necessarily the unique maximiser $\theta_{\tilde \ell}$ of the (on $\mathcal E_{conv}$) strongly concave map $\tilde \ell_N$ over $\mathbb R^D$ satisfies $\|\theta_{\tilde \ell}- \theta_{init}\|_{\mathbb R^D}^2 \le 20/C+4\eta^2/\lambda_{min}(M).$ Moreover, in view of the definition of $\mathcal B$ and the hypotheses on $\theta^*$ we have that 
\[ \|\theta_{init}\|_{\R^D}\le \|\theta_{init}-\theta^*\|_{\R^D}+\|\theta^*\|_{\R^D} \le \frac{|\theta_{init}-\theta^*|_1}{\sqrt{\lambda_{min}(M)}} +R\le \frac{\eta}{\sqrt{\lambda_{min}(M)}} +R, \]
which also allows us to deduce
\begin{equation*}
\begin{split}
	\|\theta_{\tilde \ell}\|_{\R^D}&\le \|\theta_{\tilde \ell}-\theta_{init}\|_{\R^D}+\|\theta_{init}\|_{\R^D} \le \sqrt{20/C}+ \frac{3\eta}{\sqrt{\lambda_{min}(M)}} +R.
\end{split}
\end{equation*}
 We further have that $\theta_{max}^T\Sigma^{-1}\theta_{max}\le\theta_{\tilde \ell}^T\Sigma^{-1}\theta_{\tilde \ell}$ (otherwise $\theta_{max}$ would not maximise $\log d\tilde \Pi(\cdot|Z^{(N)})$) and thus, for $\kappa(\Sigma)$ the condition number of $\Sigma$,
\[ \|\theta_{max}\|_{\R^D}^2\le \frac{1}{\lambda_{min}(\Sigma^{-1})}\theta_{max}^T\Sigma^{-1}\theta_{max}\le \frac{1}{\lambda_{min}(\Sigma^{-1})}\theta_{\tilde \ell}^T\Sigma^{-1}\theta_{\tilde \ell}\le \kappa (\Sigma)\|\theta_{\tilde \ell}\|_{\R^D}^2. \]
Combining the preceding displays, the proof is now completed as follows:
\begin{equation*}
	\begin{split}
		\|\theta_{max}-\theta_{init}\|_{\R^D}^2&\lesssim \|\theta_{max}\|_{\R^D}^2+\|\theta_{init}\|_{\R^D}^2\\
		&\lesssim \kappa (\Sigma)\|\theta_{\tilde \ell}\|_{\R^D}^2 + \frac{\eta^2}{\lambda_{min}(M)} +R^2\\
		&\lesssim \kappa (\Sigma)\Big[1+ \frac{\eta^2}{\lambda_{min}(M)} +R^2\Big].
	\end{split}
\end{equation*}

\smallskip

\textbf{Proof of Theorem \ref{thm:gen:func}.}
For any $t\ge 0$ and any Lipschitz function $H:\R^D\to \R$ we have
\begin{equation}\label{Hlip}
\begin{split}
&\mathbf P_{\theta_{init}}\Big( \big| \hat\pi_{J_{in}}^J(H)-E^\Pi[ H|Z^{(N)}]\big|\ge t \Big)\\
&~~~~~~~~~~\le \mathbf P_{\theta_{init}}\Big( \big| \hat\pi_{J_{in}}^J(H)-\mathbf E_{\theta_{init}}[\hat\pi_{J_{in}}^J(H)]\big|\ge t -\big| \mathbf E_{\theta_{init}}[\hat\pi_{J_{in}}^J(H)]-E^\Pi[ H|Z^{(N)}]\big| \Big).
\end{split}
\end{equation}  
To further estimate the right side, note that for $c_3$ large enough and any $k\ge J_{in}$, by (\ref{eq:burn}) and Theorem \ref{thm:gen:wass}, we have 
\[ W_2^2(\mathcal L(\vartheta_k), \Pi(\cdot|Z^{(N)}))\le 2 (\rho + b(\gamma)). \]
Noting that (\ref{eq:ULA:bias}) below in fact holds for any probability measure $\mu$ and thus in particular for $\mu=\Pi(\cdot|Z^{(N)})$, it follows that for any Lipschitz function $H:\R^D\to \R$,
\begin{equation*}
	\big( \mathbf E_{\theta_{init}}[\hat\pi_{J_{in}}^J(H)]-E^\Pi[ H|Z^{(N)}]\big)^2\le 2\|H\|^2_{Lip}(\rho + b(\gamma)).
\end{equation*}
Thus if $t\ge 0 $ satisfies (\ref{eq:tea}), then applying Proposition \ref{prop:ULA:conc} to both $H$ and $-H$ yields that the r.h.s.~in (\ref{Hlip}) is further bounded by
\begin{equation*}
\mathbf P_{\theta_{init}}\Big( \big| \hat\pi_{J_{in}}^J(H)-\mathbf E_{\theta_{init}}[\hat\pi_{J_{in}}^J(H)]\big|\ge t/2\Big)\le 2\exp\Big(-c\frac{t^2m^2J\gamma}{\|H\|_{Lip}^2(1+1/(mJ\gamma))}\Big).
\end{equation*}  

\smallskip

\textbf{Proof of Corollary \ref{generallymean}.}
	We first estimate the probability to be bounded by
	\begin{equation*}
	\begin{split}
 \mathbf P_{\theta_{init}}\Big(\big\|\bar\theta_{J_{in}}^J-\mathbf E_{\theta_{init}}\big[\bar\theta_{J_{in}}^J\big]\big\|_{\R^D} \ge t- \big\|\mathbf E_{\theta_{init}}\big[\bar\theta_{J_{in}}^J\big]-E^\Pi[\theta|Z^{(N)}]  \big\|_{\R^D} \Big).
	\end{split}
	\end{equation*}
	Next, for any $k \ge 1$, let $\nu_k$ denote an optimal coupling between $\mathcal L(\vartheta_k)$ and $\Pi[\cdot|Z^{(N)}]$ (cf.~Theorem 4.1 in \cite{V09}). Then by Jensen's inequality and the definition of $W_2$ from (\ref{verymuchso}),
	\begin{equation*}
	\begin{split}
	\big\|\mathbf E_{\theta_{init}}\big[\bar\theta_{J_{in}}^J\big]-E^\Pi[\theta|Z^{(N)}] \big\|_{\R^D}^2&=\bigg\|\frac 1J\sum_{k=J_{in}+1}^{J_{in}+J}\int_{\R^D\times \R^D} (\theta-\theta')d\nu_k(\theta,\theta') \bigg\|_{\R^D}^2\\
	&= \sum_{j=1}^D\bigg(\frac 1J\sum_{k=J_{in}+1}^{J_{in}+J}\int_{\R^D\times \R^D} (\theta_j-\theta_j')d\nu_k(\theta,\theta') \bigg)^2\\
	&\le \frac 1J\sum_{k=J_{in}+1}^{J_{in}+J}\int_{\R^D\times \R^D} \sum_{j=1}^D(\theta_j-\theta_j')^2d\nu_k(\theta,\theta')\\
	&=\frac 1J\sum_{k=J_{in}+1}^{J_{in}+J}W_2^2(\mathcal L(\vartheta_k), \Pi[\cdot|Z^{(N)}]).
	\end{split}
	\end{equation*}
	Thus we obtain from (\ref{eq:gen:wass}), (\ref{eq:burn}) (as after (\ref{Hlip})) that
	\[\big\|E_{\theta_{init}}\big[\bar\theta_{J_{in}}^J\big]-E^\Pi[\theta|Z^{(N)}] \big\|_{\R^D}\le \sqrt{2}\sqrt{\rho+b(\gamma)}.\]
	Now for any $j=1,...,d$, let us write $H_j:\R^D\to \R,~ \theta\mapsto \theta_j,$
	for the $j$-the coordinate projection map, of Lipschitz constant $1$. Then in the notation (\ref{ergopop}) we can write
	\[ [\bar\theta_{J_{in}}^J]_j=\hat\pi_{J_{in}}^J(H_j),~~~~ j=1,...,D. \]
	For $t\ge \sqrt{8(\rho+b(\gamma))}$ and applying Proposition \ref{prop:ULA:conc} as in the proof of Theorem \ref{thm:gen:func} as well as a union bound gives
	\begin{equation*}
	\begin{split}
	\mathbf P_{\theta_{init}}\Big(\big\|\bar\theta_{J_{in}}^J-E^\Pi[\theta|Z^{(N)}]  \big\|_{\R^D} \ge t \Big)&\le \mathbf P_{\theta_{init}}\Big(\big\|\bar\theta_{J_{in}}^J-\mathbf E_{\theta_{init}}\big[\bar\theta_{J_{in}}^J\big]\big \|_{\R^D} \ge t/2 \Big)\\
	&= \mathbf P_{\theta_{init}}\bigg(\sum_{j=1}^D  \Big[\hat\pi_{J_{in}}^J(H_j)-\mathbf E_{\theta_{init}}\big[\hat\pi_{J_{in}}^J(H_j)]\Big]^2 \ge \frac{t^2}{4} \bigg)\\
	&\le \sum_{j=1}^D \mathbf P_{\theta_{init}}\bigg( \Big[\hat\pi_{J_{in}}^J(H_j)-\mathbf E_{\theta_{init}}\big[\hat\pi_{J_{in}}^J(H_j)]\Big]^2 \ge \frac{t^2}{4D} \bigg)\\
	&\le 2D\exp\Big(-c\frac{t^2m^2J\gamma}{D\big[ 1 +1/(mJ\gamma)\big]}\Big).
	\end{split}
	\end{equation*}

\section{Proofs for the Schr\"odinger model}\label{sec:main:pfs}
In this section, we will show how the results from Section \ref{barocco} can be applied to the nonlinear problem for the Schr\"odinger equation (\ref{fwdG}). Recalling the notation of Sections \ref{sec:schrres} and \ref{barocco}, we will set $\theta^*=\theta_{0,D}$, the norm $|\cdot|_{1}:=\|\cdot\|_{\R^D}$
as well as $\eta:=\epsilon D^{-4/d}$ (for $\epsilon$ to be chosen), such that the region $\mathcal B$ from (\ref{eq:B:ass}) equals the Euclidean ball
\begin{equation}\label{eq:B:def}
\mathcal B_\epsilon:=\Big\{ \theta \in \R^D: \|\theta- \theta_{0,D}\|_{\R^D}< \epsilon D^{-4/d} \Big\}.
\end{equation}

The first key observation is the following result on the local log-concavity of the likelihood function on $\mathcal B_\epsilon$, which will be proved by a combination of the concentration result Lemma \ref{youtube} with the PDE estimates below, notably the `average curvature' bound from Lemma \ref{wundervoncordoba}.

\begin{prop}\label{schroe-emperor-smalldim}
	Let $\theta_0\in h^2$ satisfy $\|\theta_0\|_{h^{2}}\le S$ for some $S>0$ and consider $\ell_N$ from (\ref{loglik}) with forward map $\mathcal G: \R^D \to \R$ from (\ref{fwdG}). Then there exist constants $0<\epsilon_S=\epsilon_S(\mathcal O, g, \Phi)\le 1$ and $c_1,c_2,c_3,c_4>0$ such that for any $\epsilon\le \epsilon_S$ and all $D,N$ satisfying $D\le c_2N^\frac{d}{d+12}$ as well as $\|\mathcal G(\theta_0)-\mathcal G(\theta_{0,D})\|_{L^2(\mathcal O)}\le c_1D^{-4/d}$, the event
	\begin{equation*}
	\mathcal E_{conv}(\epsilon)=\Big\{\inf_{\theta\in \mathcal B_\epsilon}\lambda_{min}\big(-\nabla^2\ell_N(\theta)\big)>c_3ND^{-4/d},~\sup_{\theta\in \mathcal B_\epsilon}\Big[|\ell_N(\theta)|+\|\nabla \ell_N(\theta)\|_{\R^D}+\|\nabla^2\ell_N(\theta)\|_{op}\Big]< c_4N\Big\}
	\end{equation*}
	satisfies
	\begin{equation}\label{eq:Econv}
	P_{\theta_0}^N\big(\mathcal E_{conv}(\epsilon)\big)\ge 1- 33 e^{-c_2N^{\frac{d}{d+12}}}.
	\end{equation}
\end{prop}

\begin{proof}
	For any $\theta\in\R^D$, $F_\theta$ as in (\ref{Ft}), by a Sobolev embedding and (\ref{weyl}), we have $ \|F_\theta\|_{\infty} \lesssim \|\theta\|_{h^2}\lesssim D^{2/d}\|\theta\|_{\R^D}$. This and the Lemmas \ref{stamitz}, \ref{lem:schr:one}, \ref{lem:schr:two} verify Assumption \ref{ass:model} in the present setting, with constants
	\begin{equation*}
	k_0\simeq k_1\simeq \textnormal{const.}, ~~~~k_2\simeq m_0\simeq m_1 \simeq D^{2/d}, ~~~~  m_2 \simeq D^{4/d},
	\end{equation*}
	whence the constants from (\ref{eq:CG}) satisfy
	\[ C_{\mathcal G}\simeq D^{4/d},~~~~C_{\mathcal G}'\simeq D^{2/d},~~~~C_{\mathcal G}''\simeq D^{2/d},~~~~C_{\mathcal G}'''\simeq \textnormal{const.}.\]
	Moreover, using (\ref{dimbias}) and (\ref{hammer}), Lemmas \ref{wundervoncordoba} and \ref{lem:schr:ub} verify Assumption \ref{ass:geom} for our choice of $\eta$ with
	\begin{equation}\label{AA39}
		c_{min}\simeq D^{-4/d},~~~~ c_{max}\simeq \textnormal{const.}
	\end{equation}
	Then the minimum (\ref{eq:itis what itis}) is dominated by the third term, yielding that \[ \mathcal R_N=\mathcal R_{N,D}\simeq c_{min}^2/C_\mathcal G'^2\simeq ND^{-12/d}. \] Therefore, we can choose $c>0$ small enough such that for any $D,N\in\N$ satisfying $D\le cN^{d/(d+12)}$, we also have $D\le \mathcal R_{N,D}$. Lemma \ref{youtube} then implies that for all such $D,N$, we have
	\begin{equation}
	P_{\theta_0}^N\big(\mathcal E_{conv}^c\big)\le 32 e^{-\mathcal R_N}+e^{-N/8}\le 33 e^{-cN^{\frac{d}{d+12}}}.
	\end{equation}
\end{proof}
Next, if $\theta_{init}$ is the estimator from Theorem \ref{triebelei}, then in the present setting with $\epsilon=1/\log N$, the event (\ref{initevent}) equals
\[ \mathcal E_{init}=\Big\{ \|\theta_{init}-\theta_{0,D}\|_{\R^D} \le \frac{1}{8(\log N)D^{4/d}}\Big\}.\]
\begin{prop}
	Assuming Condition \ref{FAQ}, there exist constants $c_5,c_6>0$ such that for all $N\in \N$,
	\begin{equation*}
	P_{\theta_0}^N\big(\mathcal E_{init}\big) \ge 1- c_5e^{-c_6N^{d/(2\alpha+d)}}.
	\end{equation*}
\end{prop}

\begin{proof}
	Using Theorem \ref{triebelei} and  $\alpha>6$, we obtain that with sufficiently high probability,
	\[\|\theta_{init}-\theta_{0,D}\|_{\R^D}\lesssim N^{-(\alpha-2)/(2\alpha+d)} = o\big((\log N)^{-1}D^{-4/d}\big). \] 
\end{proof}

Next, denoting by $\tilde \Pi(\cdot|Z^{(N)})$ the `surrogate' posterior measure with density (\ref{surrod}), and if 
	\[ \mathcal E_{wass}=\Big\{ W^2_2(\tilde \Pi(\cdot|Z^{(N)}), \Pi(\cdot|Z^{(N)})) \le \exp(-N^{d/(2\alpha+d)}) \Big\},\]
is given by (\ref{wasser}) with $\rho = 2\exp(-N^{d/(2\alpha+d)})$, then Theorem \ref{waterstone} implies the following approximation result in Wasserstein distance.

\begin{prop}
	Assume Conditions \ref{asymptopia} and \ref{FAQ}. Then there exist constants $c_7,c_8>0$ such that for all $N\in \N$,
	\begin{equation*}
	P_{\theta_0}^N\big(\mathcal E_{wass}\big) \ge 1- c_7e^{-c_8N^{d/(2\alpha+d)}}.
	\end{equation*}
\end{prop}

The preceding propositions imply that the events
\begin{equation}\label{blah}
	\mathcal E_N :=\mathcal E_{conv} \cap \mathcal E_{init}\cap \mathcal E_{wass}
\end{equation}
satisfy the probability bound $P_{\theta_0}^N(\mathcal E_N)\ge 1-c'e^{-c''N^{d/(2\alpha+d)}}$. In what follows, the events $\mathcal E_N$ will be tacitly further intersected with events which have probability $1$ for all $N$ large enough, ensuring that the non-asymptotic conditions required in the results of Section \ref{barocco} are eventually  verified.

\smallskip

\textbf{Proof of Theorem \ref{thm:schr:wass}.} We will prove Theorem \ref{thm:schr:wass} by applying Theorem \ref{thm:gen:wass} with the choices $\mathcal B=\mathcal B_\epsilon$ from (\ref{eq:B:def}), $\epsilon=1/\log N$ and $K$ from Condition \ref{asymptopia}, $\rho=2\exp(-N^{d/(2\alpha+d)})$ and $M=I_{D\times D}$ generating the ellipsoidal norm $\|\cdot\|_{\R^D}$. Using (\ref{weyl}), the prior covariance $\Sigma$ from (\ref{schrottprior}) satisfies
	\[  \lambda_{min}(\Sigma^{-1})\simeq N^{\frac{d}{2\alpha+d}},~~~~~ \lambda_{max}(\Sigma^{-1})\simeq N^{\frac{d}{2\alpha+d}}D^{2\alpha/d}. \]
	Then using Condition \ref{asymptopia}, we first have that 
	\[ K \gtrsim ND^{8/d}(\log N)^2 \simeq Nc_{max}\cdot \big(1+\eta^{-2}\big), \]
	verifying the lower bound (\ref{eq:K:cond}), and then also that $m,\Lambda>0$ from Theorem \ref{thm:gen:wass} satisfy
	\begin{equation*}
	\begin{split}
	m \simeq ND^{-4/d}+N^{\frac{d}{2\alpha+d}}, ~~~\Lambda \simeq ND^{8/d}(\log N)^3+ N^{\frac{d}{2\alpha+d}}D^{\frac{2\alpha}{d}}.
	\end{split}
	\end{equation*}
	The dimension condition (\ref{dimbias}) and the condition on $\alpha$ further imply
	\[ ND^{-4/d}\gtrsim N^{\frac{d}{2\alpha+d}},~~~~~ N^{\frac{d}{2\alpha+d}}D^{\frac{2\alpha}{d}}\lesssim N, \]
	whence we further obtain
	\begin{equation} \label{lamborghini}
	 m \simeq ND^{-4/d},~~~~~ \Lambda \simeq ND^{8/d}(\log N)^3. 
	 \end{equation}
	Noting that also $\gamma =o(\Lambda^{-1})$ with our choices, Theorem \ref{thm:gen:wass} yields that on the event $\mathcal E_N$ from (\ref{blah}), the Markov chain $(\vartheta_k)$ satisfies the Wasserstein bound (\ref{eq:gen:wass}) with
	\begin{equation}
	\begin{split}
		b(\gamma)&\lesssim  \frac{\gamma D\Lambda^2}{m^2}+\frac{\gamma^2D\Lambda^4}{m^3} \lesssim \gamma D^{(d+24)/d}(\log N)^6 + \gamma^2ND^{(d+44)/d}(\log N)^{12},
	\end{split}
	\end{equation}
	as well as
	\[	\tau(\Sigma,M,\|\theta_{0,D}\|_{\R^D})\lesssim \kappa(\Sigma)\simeq D^{2\alpha/d}. \]
	Using also that $D/m\lesssim \textnormal{const.}$, the first part of Theorem \ref{thm:schr:wass} follows.
	\par 
	For the choice of $\gamma=\gamma_\eps$ from (\ref{gammel}), straightforward calculation yields that (for $N$ large enough)
	\begin{equation}\label{zadok}
		B(\gamma_\eps) = o(\eps^2 + N^{-2P}),
	\end{equation}
	which proves the second part of Theorem \ref{thm:schr:wass}.
	
	\smallskip
	
	\textbf{Proof of Proposition \ref{thm:schr:func} and of Theorems \ref{thm:post:mean}, \ref{brainfreeze}.} The proof of Proposition \ref{thm:schr:func} now follows directly from Theorem \ref{thm:gen:func} and the preceding computations. Noting that for all $N$ large enough we have $B(\gamma)\le N^{-P}$, Theorem \ref{thm:post:mean} follows from Corollary \ref{generallymean}, (\ref{zadok}) as well as (\ref{eq:gen:concmean}), for $J_{in}\ge  (\log N)^3/(\gamma_\varepsilon ND^{-4/d})$. Finally, intersecting further with the event
	\[ \mathcal E_{mean}:= \big\{ \|E^{\Pi}[\theta|Z^{(N)}]-\theta_0 \|_{\ell^2}\le LN^{-\frac{\alpha}{2\alpha+d}\frac{\alpha}{\alpha+2}} \big\},~~L>0, \]
	Theorem \ref{brainfreeze} now follows from the triangle inequality and (\ref{postmeanrat}).
	
	\smallskip
	
	\textbf{Proof of Theorem \ref{MAP}.} In the proof we intersect $\mathcal E_N$ from (\ref{blah}) further with the event on which the conclusion of Theorem \ref{maprate} holds. Part iii) then follows from part ii) and straightforward calculations. Part i) follows from the arguments following (\ref{score}) below, where it is proved in particular that $\hat \theta_{MAP}$ is the unique maximiser of the proxy posterior density $\tilde \pi(\cdot|Z^{(N)})$ over $\R^D$. We can now apply Proposition \ref{reiskorn} with $m, \Lambda$ from (\ref{lamborghini}), using also that
	\begin{align*}
	 & |\log \tilde \pi (\theta_{init}|Z^{(N)})-\log \tilde \pi (\hat \theta_{MAP}|Z^{(N)})| \\
	 &\lesssim  \sup_{\theta\in \mathcal B_{1/8\log N}} \big|\ell_N(\theta) \big| + N^{d/(2\alpha+d)}\|\hat \theta_{MAP}\|_{h^\alpha}^2 +N^{d/(2\alpha+d)}\|\theta_{init}\|_{h^\alpha}^2\\
	 & \lesssim N + N^{d/(2\alpha+d)} (1+ D^{2\alpha/d}) \lesssim N,
	 \end{align*}
	 in view of $\ell_N =\tilde \ell_N$ on $\mathcal B_{1/8 \log N}$, the definition of $\mathcal E_{init}$,  (\ref{weyl}) and since $\theta_0\in h^\alpha$.

\subsection{Analytical properties of the Schr\"odinger forward map}\label{sec:schr:llh}

This section is devoted to proving the four auxiliary Lemmas \ref{lem:schr:one}-\ref{lem:schr:ub} used in the proof of Proposition \ref{schroe-emperor-smalldim}. Throughout we consider forward map $\mathcal G: \R^D \to L^2(\mathcal O)$, $\mathcal G=G\circ \Phi^* \circ \Psi$ given by (\ref{fwdG}) and assume the hypotheses of Proposition \ref{schroe-emperor-smalldim}, where the set $\mathcal B_\epsilon$ was defined in (\ref{eq:B:def}).

For any $f\in C(\mathcal O)$ with $f\ge 0$, by standard theory for elliptic PDEs (see e.g. Chapter 6.3 of \cite{E10}) there exists a linear, continuous operator $V_f:L^2(\mathcal O) \to H^2_0(\mathcal O)$ describing (weak) solutions $V_f[\psi]=w\in H^2_0$ of the (inhomogeneous) Schr\"odinger equation
\begin{equation}\label{Vf}
\begin{cases}
\frac{\Delta}{2} w - fw =\psi~~~ \text{on}~ \mathcal O, \\
w=0 ~~~\text{on}~ \partial \mathcal O.
\end{cases}
\end{equation}
\begin{lem}\label{stamitz}
For any $x\in\mathcal O$, the map $\theta\mapsto \mathcal G(\theta)(x)$ is twice continuously differentiable on $\R^D$. The vector field $\nabla \mathcal G_\theta:\mathcal O\to \R^D$ is given by
\[ v^T\nabla \mathcal G_\theta (x) =V_{f_\theta}\big[ u_{f_\theta}(\Phi'\circ F_\theta)\Psi(v) \big](x), ~~~x\in\mathcal O, ~v\in \R^D.\]
Moreover, for any $v_1,v_2\in \R^D$ and $x\in\mathcal O$, the matrix field $\nabla^2 \mathcal G_\theta:\mathcal O\to \R^{D\times D}$ is given by
\begin{equation*}
\begin{split}
	v_1^T\nabla^2\mathcal G_\theta (x)v_2=&V_{f_\theta}\big[u_{f_\theta}\Psi(v_1)\Psi(v_2)(\Phi''\circ F_\theta)\big](x)\\
	&~~~~~~+V_{f_\theta}\big[ (\Phi'\circ F_\theta)\Psi(v_1)V_{f_\theta}\big[ u_{f_\theta}(\Phi'\circ F_\theta)\Psi(v_2)\big]\big](x)\\
	&~~~~~~+V_{f_\theta}\big[ (\Phi'\circ F_\theta)\Psi(v_2)V_{f_\theta}\big[ u_{f_\theta}(\Phi'\circ F_\theta)\Psi(v_1)\big]\big](x).
\end{split}
\end{equation*}
\end{lem}

\begin{proof}
	In the notation from (\ref{fwdG}), the map $\theta\mapsto \mathcal G(\theta)(x)$ can be represented as the composition $\delta_x \circ G\circ \Phi^*\circ \Psi$, where $\delta_x: w\mapsto w(x)$ denotes point evaluation. We first show that each of these four operators is twice differentiable. The continuous linear maps $\Psi:\R^D\to C(\mathcal O)$ and $\delta_x: C(\mathcal O) \to \mathbb R$ are infinitely differentiable (in the Frech\'et sense). Moreover, the maps $G:C(\mathcal O) \cap \{f>0\} \to C(\mathcal O)$ and $\Phi^*:C(\mathcal O)\to C(\mathcal O) \cap \{f>0\}$ are twice Fre\'chet differentiable with derivatives $DG,~DG^2$ and $D\Phi^*,D^2\Phi^*$ given by Lemma \ref{lem:schr:deriv} and (\ref{Dphi}) respectively. We deduce overall by the chain rule for Fr\'echet derivatives (cf. Lemma \ref{lem:chain}), that $x\mapsto \mathcal G(\theta)(x)$ is twice differentiable, with the desired expressions for the vector and matrix fields. The continuity of the second partial derivatives follows from inspection of the expression for the matrix field, and by applying the regularity results for $V_f,G$ and $\Phi^*$ from Appendix \ref{sec:aux}.
\end{proof}

Now since $\|\theta_0\|_{h^2} \le S$ and by the definition (\ref{eq:B:def}) of the set $\mathcal B_1$, we have from (\ref{weyl}) that
\begin{equation*}
\begin{split}
\sup_{\theta\in\mathcal B_1}\|\theta\|_{h^{2}}\le\|\theta_{0,D}\|_{h^{2}}+\sup_{\theta\in\mathcal B_1}\|\theta-\theta_{0,D}\|_{h^{2}}\lesssim S + D^{\frac{2}{d}}\sup_{\theta\in\mathcal B_1}\|\theta-\theta_{0,D}\|_{\R^D}\lesssim S+1.
\end{split}
\end{equation*}
It follows further from the Sobolev embedding and regularity of the link function $\Phi$ (Appendix \ref{ssec-link}) that there exists a constant $B=B(S,\Phi,\mathcal O)<\infty$, such that
\begin{equation}\label{hammer}
\sup_{\theta\in\mathcal B_1}\Big[\|F_\theta\|_\infty +\|F_\theta\|_{H^2} + \|f_\theta\|_{H^2}+ \|f_\theta\|_\infty \Big]  \le  B.
\end{equation}
In particular, this estimate implies that the constants appearing in the inequalities from Lemma \ref{cpebach} can be chosen independently of $\theta\in\mathcal B$, which we use frequently below.

\smallskip

For notational convenience we also introduce spaces
	\begin{equation}\label{subba}
		E_D:=\spn(e_1,...,e_D)\subseteq L^2(\mathcal O),~~~ D\in \N,
	\end{equation}
	spanned by the first $D$ eigenfunctions of $\Delta$ on $\mathcal O$ (cf.~Section \ref{ONB}).

\smallskip

We first verify the boundedness property required in Assumption \ref{ass:model} ii).

\begin{lem}\label{lem:schr:one}
	There exists a constant $C>0$ such that
	\begin{equation*}
	\begin{split}
	\sup_{\theta\in\mathcal B_1} \|\mathcal G(\theta)\|_{L^\infty}\le C,~~~~\sup_{\theta\in\mathcal B_1}\|\nabla \mathcal G(\theta)\|_{L^\infty(\mathcal O,\R^D)}\le C,~~~~\sup_{\theta\in\mathcal B_1}\|\nabla^2 \mathcal G(\theta)\|_{L^\infty(\mathcal O,\R^{D\times D})}\le CD^{2/d}.
	\end{split}
	\end{equation*}
\end{lem}
\begin{proof}
	The estimate for $\|\mathcal G(\theta)\|_{\infty}$ follows immediately from (\ref{ubd}). To estimate $\|\nabla \mathcal G(\theta)\|_{L^\infty(\mathcal O,\R^D)}$, we first note that by Lemma \ref{stamitz},
\begin{equation*}
	\begin{split}
		\|\nabla \mathcal G(\theta)\|_{L^\infty(\mathcal O,\R^D)} =\sup_{v:\|v\|_{\R^D}\le 1}
		\|v^T \nabla \mathcal G(\theta)\|_{L^\infty} \le \sup_{H\in E_D :\|H\|_{L^2}\le 1}\big\|V_{f_\theta}\big[ u_{f_\theta}(\Phi'\circ F_\theta)H \big]\big\|_{\infty}.
	\end{split}
\end{equation*}
		Thus by the Sobolev embedding $\|\cdot\|_\infty\lesssim \|\cdot\|_{H^2}$, Lemma \ref{cpebach} and boundedness of $\Phi'$, we have that for any $\theta\in \mathcal B_1$ and any $H\in E_D$,
\begin{equation*}
	\begin{split}
		\big\|V_{f_{\theta}}[u_{f_{\theta}} (\Phi'\circ F_\theta) H]\big\|_{\infty}& \lesssim \big\|V_{f_\theta}[u_{f_\theta} (\Phi'\circ F_\theta) H]\big\|_{H^2}\\
		&\lesssim \big\|u_{f_\theta} (\Phi'\circ F_\theta) H\big\|_{L^2}\\
		&\lesssim \big\|u_{f_\theta}\|_\infty\|\Phi'\circ F_\theta\|_\infty \| H\|_{L^2}\lesssim \| H\|_{L^2}.
	\end{split}
\end{equation*}
Again using Lemma \ref{stamitz}, we can similarly estimate $\|\nabla^2\mathcal G(\theta)\|_{L^\infty(\mathcal O,\R^D)}$ by 
\begin{equation}\label{bilinear}
	\begin{split}
	&\|\nabla^2\mathcal G(\theta)\|_{L^\infty(\mathcal O,\R^D)}\le \sup_{v:\|v\|_{\R^D}\le 1}
	\|v^T\nabla^2\mathcal G(\theta)v\|_{L^\infty}\\
	&~~~~~\le \sup_{H\in E_D :\|H\|_{L^2}\le 1} 2\big\| V_{f_\theta}\big[ H(\Phi'\circ F_\theta)V_{f_\theta}\big[ H(\Phi'\circ F_\theta)u_{f_\theta}\big]\big]\big\|_{\infty}+ \big\|V_{f_\theta}\big[H^2(\Phi''\circ F_\theta)u_{f_\theta}\big]\big\|_\infty\\
	&~~~~~=: \sup_{H\in E_D :\|H\|_{L^2}\le 1} I+II.
	\end{split}
\end{equation}
Arguing as in the estimate for $\|\nabla \mathcal G(\theta)\|_{L^\infty(\mathcal O,\R^D)}$, we have that for any $\theta \in \mathcal B_1$ and $H\in E_D$,
\begin{equation*}
\begin{split}
I &\lesssim \|H(\Phi'\circ F_\theta)V_{f_\theta}\big[ H(\Phi'\circ F_\theta)u_{f_\theta}\big]\|_{L^2}\\
&\lesssim \|H\|_{L^2}\|\Phi'\circ F\|_\infty \|V_f[H(\Phi'\circ F)u_f]\|_{\infty}\\
&\lesssim \|H\|_{L^2} \|H(\Phi'\circ F)u_f\|_{L^2}\lesssim  \|H\|_{L^2}^2,
\end{split}
\end{equation*}
as well as
\begin{equation*}
\begin{split}
II\lesssim \|H^2(\Phi''\circ F_\theta)u_{f_\theta}\|_{L^2}\lesssim \|u_{f_\theta}\|_\infty \|\Phi''\circ F_\theta\|_\infty \|H\|_{L^2}\|H\|_{\infty}\lesssim \|H\|_{L^2}\|H\|_{H^2}\lesssim D^{2/d}\|H\|_{L^2}^2,
\end{split}
\end{equation*}
where we used the basic norm estimate on $E_D\subseteq L^2(\mathcal O)$ from Lemma \ref{normal}. By combining the last three displays, the proof is completed.
\end{proof}

Next, we verify the increment bound needed in Assumption \ref{ass:model} iii).

\begin{lem}\label{lem:schr:two}
	There exists a constant $C>0$ such that for any $D\in \N$ and any $\theta,\theta' \in \R^D$,
	\begin{align}\label{eq:zero-increm}
			\|\mathcal G(\theta)-\mathcal G(\bar \theta)\|_\infty \le C\|F_\theta-F_{\bar\theta}\|_{\infty},~~~\|\mathcal G(\theta)-\mathcal G(\bar \theta)\|_{L^2} \le C\|F_\theta-F_{\bar\theta}\|_{L^2},
	\end{align}
	as well as, for any $\theta,\theta' \in \mathcal B_1$,
	\begin{align}
			\|\nabla \mathcal G(\theta)-\nabla \mathcal G(\bar \theta)\|_{L^\infty(\mathcal O,\R^D)}&\le C\|F_\theta-F_{\bar\theta}\|_{\infty},\label{eq:one-increm}\\
			\|\nabla^2 \mathcal G(\theta)-\nabla^2 \mathcal G(\bar \theta)\|_{L^\infty(\mathcal O,\R^{D\times D})}&\le CD^{2/d}\|F_\theta-F_{\bar\theta}\|_{\infty}. \label{eq:two-increm}
	\end{align}
\end{lem}

\begin{proof} The estimate (\ref{eq:zero-increm}) follows immediately from (\ref{hummel}) and (\ref{linklip}). Now fix any $\theta,\bar\theta \in \mathcal B_1$. To ease notation, in what follows we write $F=\Psi(\theta), \bar F=\Psi(\bar \theta)$, $f=\Phi\circ F$ and $\bar f=\Phi\circ \bar F$. For (\ref{eq:one-increm}), arguing as in the proof of Lemma \ref{lem:schr:one}, we first have
\begin{equation*}
	\begin{split}
		\big\|\nabla \mathcal G(\theta)&-\nabla \mathcal G(\bar\theta)\big\|_{L^\infty (\mathcal O,\R^D)}\\
		&\le \sup_{v:\|v\|_{\R^D}\le 1}\big\|v^T(\nabla \mathcal G(\theta)-\nabla \mathcal G(\bar\theta))\big\|_{\infty}\\
		&\le \sup_{H\in E_D:\|H\|_{L^2}\le 1} \big\|V_f[H(\Phi'\circ F)u_f]-V_{\bar f}[H(\Phi'\circ \bar F)u_{\bar f}]\big\|_{\infty}\\
		&= \sup_{H\in E_D:\|H\|_{L^2}\le 1} \big\|(V_f-V_{\bar f})[H(\Phi'\circ F)u_f]\big\|_\infty +\big\|V_{\bar f}[H(\Phi'\circ F-\Phi'\circ \bar F)u_{\bar f}]\big\|_\infty\\
		&~~~~~~~~~~~~~~~~~~~~~~~~~~~~~~~~~~~~~~~+  \big\|V_{\bar f}[H(\Phi'\circ F)(u_f-u_{\bar f})]\big\|_\infty\\
		&=:\sup_{H\in E_D:\|H\|_{L^2}\le 1} I_a+I_b+I_c.
	\end{split}
\end{equation*}
Now, we fix $H\in E_D$ for the rest of the proof. The term $I_a$ can further be estimated by repeatedly using the Sobolev embedding $\|\cdot\|_\infty \lesssim \|\cdot\|_{H^2}$, Lemma \ref{cpebach} as well as (\ref{hammer}) and (\ref{linklip}):
\begin{equation}\label{eq:Ia:est}
\begin{split}
I_a &= \|V_f[ (f-\bar f) V_{\bar f} [u_{\bar f}(\Phi'\circ  F)H] ] \|_\infty \\
&\lesssim \|V_f[ (f-\bar f) V_{\bar f} [u_{\bar f}(\Phi'\circ  F)H] ] \|_{H^2}\\
&\lesssim \|(f-\bar f) V_{\bar f} [u_{\bar f}(\Phi'\circ F)H]\|_{L^2}\\
&\lesssim \|f-\bar f\|_\infty \| u_{\bar f}(\Phi'\circ \bar F)H\|_{L^2}\\
&\lesssim \|F-\bar F\|_\infty \|H\|_{L^2}.
\end{split}
\end{equation}
Similarly, $I_b$ is estimated as follows:
\begin{equation*}
\begin{split}
I_b\lesssim \|H(\Phi'\circ F-\Phi'\circ \bar F)u_{\bar f}\|_{L^2}\lesssim \|\Phi'\circ F-\Phi'\circ \bar F\|_\infty  \|u_{\bar f}\|_\infty\|H\|_{L^2}\lesssim \|F-\bar F\|_\infty \|H\|_{L^2}.
\end{split}
\end{equation*}
Finally, we can similarly estimate 
\begin{equation*}
\begin{split}
I_c \lesssim \|(u_f-u_{\bar f})(\Phi'\circ F)H\|_{L^2}\lesssim \|u_f-u_{\bar f}\|_\infty \|\Phi'\circ F\|_\infty \|H\|_{L^2}\lesssim \|F-\bar F\|_\infty\|H\|_{L^2},
\end{split}
\end{equation*}
where we have also used (\ref{eq:zero-increm}). By combining the estimates for $I_a,I_b$ and $I_c$, we have completed the proof of (\ref{eq:one-increm}).

It remains to prove (\ref{eq:two-increm}). In analogy to (\ref{bilinear}), we may fix any $v\in \R^D$, and it suffices to derive a bound for $v^T(\nabla^2\mathcal G(\theta)-\nabla^2\mathcal G(\bar \theta))v$. To ease notation, let us write $H=\Psi v\in E_D\cong \R^D$, as well as $h=H(\Phi'\circ F)$ and $\bar h=H(\Phi'\circ \bar F)$. Then by Lemma \ref{stamitz}, we have the following decomposition into eight terms:
\begin{equation}\label{B39}
\begin{split}
&v^T(\nabla^2\mathcal G(\theta)-\nabla^2\mathcal G(\bar \theta))v\\
&~=2V_{\bar f}\big[\bar hV_{\bar f}[\bar h u_{\bar f}] \big]- 2V_{ f}\big[ hV_{f}[h u_{f}] \big]+V_{\bar f}[u_{\bar f}H^2(\Phi''\circ \bar F) ]- V_{f}[u_{f}H^2(\Phi''\circ F)]\\
&~=2(V_{\bar f}-V_f)\big[\bar hV_{\bar f}[\bar h u_{\bar f}] \big] +2V_f\big[(\bar h-h)V_{\bar f}[\bar h u_{\bar f}] \big]\\
&~~~~~+2V_f\big[h(V_{\bar f}-V_f)[\bar h u_{\bar f}] \big]+2V_f\big[hV_f[(\bar h-h) u_{\bar f}] \big]+2V_f\big[hV_f[h (u_{\bar f}-u_f)] \big]\\
&~~~~~+(V_{\bar f}-V_f)[u_{\bar f}H^2(\Phi''\circ \bar F) ]+V_f[(u_{\bar f}-u_f)H^2(\Phi''\circ \bar F) ]+V_f[u_f H^2(\Phi''\circ \bar F-\Phi''\circ F)]\\
&~=: II_a+II_b+II_c+II_d+II_e+II_f+II_g+II_h.
\end{split}
\end{equation}

To estimate these terms, we will again repeatedly use (\ref{hammer}), the regularity estimates from Lemmas \ref{cpebach}- \ref{lem:schr:deriv} below, the estimates $\|h\|_{L^2},\|\bar h\|_{L^2} \lesssim \|H\|_{L^2}$ as well as $\|f-\bar f\|_\infty\lesssim \|F-\bar F\|_\infty$, which all hold uniformly in $\theta\in\mathcal B_1$.
\par 
Using Lemma \ref{cpebach}, including the estimate (\ref{eq:Vfdiff}) with $\psi= \bar h V_{\bar f}[\bar hu_{\bar f}]$, we obtain
\begin{equation*}
\begin{split}
\|II_a\|_\infty &\lesssim \|f-\bar f\|_\infty\|\bar h V_{\bar f}[\bar hu_{\bar f}]\|_{L^2}\lesssim \|f-\bar f\|_\infty\|\bar h\|_{L^2} \|V_{\bar f}[\bar hu_{\bar f}]\|_{\infty}\\
&\lesssim \|f-\bar f\|_\infty\|H\|_{L^2} \|\bar hu_{\bar f}\|_{L^2}\lesssim \|f-\bar f\|_\infty\|H\|_{L^2}^2 \|u_{\bar f}\|_{\infty}\\
&\lesssim \|F-\bar F\|_\infty\|H\|_{L^2}^2.
\end{split}
\end{equation*}
Similarly, we have
\begin{equation*}
\begin{split}
\|II_b\|_\infty &\lesssim \|(\bar h-h)V_{\bar f}[\bar hu_{\bar f}]\|_{L^2}\lesssim \|H(\Phi'\circ \bar F-\Phi'\circ F)\|_{L^2}\|V_{\bar f}[\bar hu_{\bar f}]\|_\infty\\
&\lesssim \|u_f\|_\infty \|H\|_{L^2}\|\bar F-F\|_\infty \|\bar h u_{\bar f}\|_{L^2}\\
&\lesssim \|H\|_{L^2}^2\|\bar F-F\|_\infty,
\end{split}
\end{equation*}
and, again using (\ref{eq:Vfdiff}),
\begin{equation*}
\begin{split}
\|II_c\|_\infty &\lesssim \|h(V_{\bar f}-V_f)[\bar hu_{\bar f}]\|_{L^2}\lesssim \|h\|_{L^2}\|(V_{\bar f}-V_f)[\bar hu_{\bar f}]\|_{\infty}\lesssim \|H\|_{L^2} \|\bar f-f\|_\infty \| \bar hu_{\bar f} \|_{L^2}\\
&\lesssim \|H\|_{L^2}^2 \|\bar F-F\|_\infty.
\end{split}
\end{equation*}
For $II_d$, by following similar steps as for $II_b$, we see that
\begin{equation*}
\begin{split}
\|II_d\|_\infty \lesssim \|H\|_{L^2}\|V_f[(\bar h-h)u_{\bar f}]\|_\infty \lesssim \|H\|_{L^2}^2\|\bar F-F\|_\infty,
\end{split}
\end{equation*}
and similarly, using also (\ref{eq:zero-increm}), we obtain
\begin{equation*}
\begin{split}
\|II_e\|_\infty &\lesssim \|H\|_{L^2}\|V_f[h(u_{\bar f}-u_f)]\|_{\infty}\lesssim \|H\|_{L^2}^2\|u_{\bar f}-u_f\|_{\infty}\lesssim \|H\|_{L^2}^2\|\bar F-F\|_\infty.
\end{split}
\end{equation*}
For the term $II_f$, we note that by the Sobolev embedding,
\[ \|w\|_{(H^2_0)^*}\le \sup_{\psi:\|\psi\|_{H^2}\le 1}\big|\int_{\mathcal O}w\psi \big|\lesssim \|w\|_{L^1}\sup_{\psi:\|\psi\|_{H^2}\le 1}\|\psi\|_\infty\lesssim \|w\|_{L^1},~~ w\in L^1(\mathcal O),\]
and consequently by Lemma \ref{cpebach},
\begin{equation*}
\begin{split}
\|II_f\|_\infty & = \|V_f[(\bar f-f)V_{\bar f}[u_{\bar f}H^2(\Phi''\circ F)]]\|_{\infty}\\
&\lesssim \|\bar f-f\|_\infty \|V_{\bar f}[u_{\bar f}H^2(\Phi''\circ F)]\|_{L^2}\\
&\lesssim \|\bar f-f\|_\infty \|u_{\bar f}H^2(\Phi''\circ F)\|_{(H^2_0)^*}\\
&\lesssim \|\bar f-f\|_\infty \|u_{\bar f}H^2(\Phi''\circ F)\|_{L^1}\\
&\lesssim \|\bar F-F\|_\infty \|H\|_{L^2}^2.
\end{split}
\end{equation*}
For terms $II_g$ and $II_h$, by similar steps and additionally using that by Lemma \ref{normal}, $\|H\|_{\infty}\lesssim \|H\|_{H^2}\lesssim D^{2/d}\|H\|_{L^2}$ for any $H\in E_D$, we obtain
\begin{equation*}
\begin{split}
\|II_g\|_\infty &\lesssim \|u_{\bar f}-u_f\|_\infty \|H^2\|_{L^2}\|\Phi''\circ \bar F\|_\infty \lesssim \|\bar f-f\|_\infty \|H\|_{L^2}\|H\|_{\infty}\lesssim D^{2/d}\|\bar F-F\|_\infty \|H\|_{L^2}^2,
\end{split}
\end{equation*}
as well as
\begin{equation*}
\begin{split}
\|II_h\|_\infty \le \|u_fH^2(\Phi''\circ \bar F-\Phi''\circ F)\|_{L^2}\lesssim \|H\|_{L^2}\|H\|_{\infty}\| \bar F-F \|_\infty\lesssim D^{2/d}\|\bar F-F \|_\infty\|H\|_{L^2}^2.
\end{split}
\end{equation*}
By combining (\ref{B39}) with the estimates for the terms $II_a-II_h$, the proof of (\ref{eq:two-increm}) is complete.
\end{proof}

We now turn to the key `geometric' bound from the first part of Assumption \ref{ass:geom}, which quantifies the average curvature of the likelihood function $\ell_N$ near $\theta_{0,D}$ in a high-dimensional setting (when $P^X$ is uniform on $\mathcal O$). The curvature deteriorates with rate $D^{-4/d}$ as $D\to\infty$, which is in line with the (local) ill-posedness of the Schr\"odinger model, and the related fact that the associated `information operator' is of the form $I^2$, with $I$ being the inverse of a second order (elliptic Schr\"odinger-type) operator (cf.~also Section 4 in \cite{N17}).

\begin{lem}\label{wundervoncordoba}
Let $\ell(\theta)$ be as in (\ref{eq:gen:llh}) with $\mathcal G: \R^D \to \R$ from (\ref{fwdG}), and let $\mathcal B_\epsilon$ be as in (\ref{eq:B:def}). Let $\theta_0\in h^2$ satisfy $\|\theta_0\|_{h^{2}}\le S$ for some $S>0$. Then there exist constants $0<\epsilon_S\le 1, c_1,c_2>0$ such that if also $\|\mathcal G(\theta_0)-\mathcal G(\theta_{0,D})\|_{L^2(\mathcal O)}\le c_1D^{-4/d}$, then for all $D\in\N$ and all $\epsilon\le \epsilon_S$,
	\begin{equation}\label{important}
	\inf_{\theta\in\mathcal B_\epsilon}\lambda_{min}\big( E_{\theta_0}\big[-\nabla^2\ell(\theta)\big]\big)\ge c_2D^{-4/d}.
	\end{equation}
\end{lem}

\begin{proof}
	We begin by noting that for any $Z=(Y,X)\in \R\times \mathcal O$, we have
	\begin{equation*}
	\begin{split}
		-\nabla^2 \ell (\theta, Z)= \nabla \mathcal G^X(\theta)\nabla\mathcal G^X(\theta)^T-(Y-\mathcal G^X(\theta))\nabla^2 \mathcal G^X(\theta).
	\end{split}
	\end{equation*}
	Using this and Lemma \ref{stamitz}, we obtain that for any $v\in \R^D$, with the previous notation $H=\Psi (v)$ and $h=(\Phi'\circ F_\theta)H$,
	\begin{equation}\label{eq:the-expected-one}
	\begin{split}
	 v^TE_{\theta_0}[-\nabla^2\ell (\theta,Z)]v&=  \|V_{f_\theta}[u_{f_\theta}(\Phi'\circ F_\theta)H ]\|_{L^2(\mathcal O)}^2-\langle u_{f_{\theta_0}}-u_{f_\theta}, 2V_{f_\theta}[hV_{f_\theta}[hu_{f_\theta}] ]\rangle_{L^2(\mathcal O)}\\
	 &~~~~~~~~~~~~~~~- \langle u_{f_{\theta_0}}-u_{f_\theta},V_{f_\theta}[u_{f_\theta}H^2 (\Phi''\circ F_\theta)] \rangle_{L^2(\mathcal O)}\\
	&=:I+II+III.
	\end{split}
	\end{equation}
	We next derive a lower bound on the term $I$ and upper bounds for the terms $II$ and $III$, for any fixed $v\in\R^D$. 
	
\smallskip

\textit{Lower bound for $I$.} Writing $a_\theta:=u_{f_\theta}(\Phi'\circ F_\theta)$, using the elliptic $L^2$-$(H^2_0)^*$ coercivity estimate (\ref{eq:schr:lb}) from Lemma \ref{cpebach} below as well as (\ref{hammer}), we have
\begin{equation}\label{lb-init}
\sqrt I=\|V_{f_\theta}[a_\theta H]\|_{L^2(\mathcal O)} \gtrsim \frac{\| a_\theta H\|_{(H^2_0)^*}}{1+\|f_\theta\|_\infty}\gtrsim  \| a_\theta H\|_{(H^2_0)^*}, ~~~ \theta\in \mathcal B_1.
\end{equation}
The next step is to lower bound $a_\theta$. By Theorem 1.17 in \cite{CZ95}, the expected exit time $\tau_\mathcal O$ featuring in the Feynman-Kac formula (\ref{fkac}) satisfies the uniform estimate $\sup_{x \in \mathcal O}E^x \tau_\mathcal O \le K(vol(\mathcal O), d)<\infty$. Therefore, using also Jensen's inequality and $g\ge g_{min}>0$, we have that, with $B$ from (\ref{hammer}),
\begin{equation} \label{leivand} \inf_{\theta\in\mathcal B_1} \inf_{x\in\mathcal O} u_{f_\theta}(x)\ge g_{min}e^{-BK(vol(\mathcal O),d)}=:u_{min}>0.
\end{equation}
Also, since $\Phi$ is a regular link function, for some $k=k(B)>0$ we have
\[ \inf_{\theta\in\mathcal B_1}\inf_{x\in\mathcal O}[\Phi'\circ F_\theta](x)\ge \inf_{t\in[-k,k]}\Phi'(t)>0, \]
and therefore for some $a_{min}=a_{min}(\Phi,B,\mathcal O,g_{min})>0$,
\begin{equation}\label{eq:aF:amin}
\inf_{\theta\in\mathcal B_1}\inf_{x\in\mathcal O} a_\theta(x)\ge  a_{min}>0.
\end{equation}
We thus obtain, by definition of $(H^2_0)^*$ and the multiplication inequality (\ref{eq:h-mult}) that for some $c=c(a_{min})>0$,
\begin{equation}\label{multi}
\begin{split}
 \|H\|_{(H^2_0)^*}=\| a_\theta a_\theta^{-1}H\|_{(H^2_0)^*}\le \|a_\theta^{-1}\|_{H^2} \| a_\theta H\|_{(H^2_0)^*}\le c(1+ \|a_\theta\|_{H^2}^2) \| a_\theta H\|_{(H^2_0)^*},
\end{split}
\end{equation}
where in the last inequality we used (\ref{links}) for the function $x\mapsto 1/x$. Using again (\ref{hammer}), regularity of $\Phi'$, the chain rule as well as the elliptic regularity estimate (\ref{2bound}), we obtain that
\begin{equation}\label{super}
\sup_{\theta\in\mathcal B_1} \|a_\theta\|_{H^2}\le \sup_{\theta\in\mathcal B_1} \|u_{f_\theta}\|_{H^2}\sup_{\theta\in\mathcal B_1} \|\Phi'\circ F_\theta\|_{H^2}\le C(g,S,\mathcal O,\Phi)<\infty.
\end{equation}
Therefore, combining the displays (\ref{lb-init}), (\ref{multi}), (\ref{super}), we have proved that, uniformly in $\theta\in\mathcal B_1$,
\begin{equation}\label{diekrux}
I\gtrsim \|a_\theta H\|_{(H^2_0)^*}^2 \gtrsim \frac{\| H\|_{(H^2_0)^*}^2}{c^2\sup_{\theta\in\mathcal B_1}(1+ \|a_\theta\|_{H^2}^2)^2}\gtrsim D^{-4/d} \| H\|_{L^2}^2,
\end{equation}
where we have used Lemma \ref{normal} below in the last inequality.

\smallskip

	\textit{Upper bound for $II$ and $III$.}
	Using the self-adjointness of $V_{f_\theta}$ on $L^2(\mathcal O)$, a Sobolev embedding, Lemma \ref{cpebach}, (\ref{hammer}), the Lipschitz estimate (\ref{hummel}) as well as (\ref{ubd}), we have uniformly in $\theta\in\mathcal B_1$,
	\begin{equation}\label{eq:I-upper}
	\begin{split}
	|II|&\lesssim \Big| \int_{ \mathcal O} (u_{f_{\theta_0}}-u_{f_\theta})V_{f_\theta}[hV_{f_\theta}[h u_{f_\theta}] ]\Big|=\Big|\int_{ \mathcal O} V_{f_\theta}[u_{f_{\theta_0}}-u_{f_\theta}][hV_{f_\theta}[h u_{f_\theta}] ]\Big| \\
	&\lesssim \|V_{f_\theta}[u_{f_{\theta_0}}-u_{f_\theta}]\|_\infty \|hV_{f_\theta}[h u_{f_\theta}] \|_{L^1}\\
	&\lesssim \|u_{f_{\theta_0}}-u_{f_\theta}\|_{L^2} \|h\|_{L^2}\|V_{f_\theta}[h u_{f_\theta}] \|_{L^2}\\
	&\lesssim \|u_{f_{\theta_0}}-u_{f_\theta}\|_{L^2} \|H\|_{L^2}^2.
	\end{split}
	\end{equation}
	Similarly, for the term $III$, using also $\|\Phi''\|_\infty<\infty$, we estimate
	\begin{equation}\label{eq:II-upper}
	\begin{split}
	|III|&=\big| \langle u_{f_{\theta_0}}-u_{f_\theta}, V_{f_\theta}[u_{f_\theta}H^2(\Phi''\circ F_\theta) ]  \rangle_{L^2(\mathcal O)} 
	\big|\\
	&= \big| \langle V_{f_\theta}[u_{f_{\theta_0}}-u_{f_\theta}], u_{f_\theta}H^2(\Phi''\circ F_\theta)  \rangle_{L^2(\mathcal O)} \big|\\
	&\le \|V_{f_\theta}[u_{f_{\theta_0}}-u_{f_\theta}]\|_\infty \| u_{f_\theta}\|_\infty  \|\Phi''\circ F_\theta\|_\infty\|H^2\|_{L^1}\\
	&\lesssim \|u_{f_{\theta_0}}-u_{f_\theta}\|_{L^2} \|H\|_{L^2}^2
	\end{split}
	\end{equation}
	
	\smallskip
	
	Combining the displays (\ref{eq:the-expected-one}), (\ref{diekrux}), (\ref{eq:I-upper}) and (\ref{eq:II-upper}), we have proved that for any $\theta \in \mathcal B_1$, any $v\in \R^D$ and some constants $c',c''>0$,
	\[  v^TE_{\theta_0}[-\nabla^2 \ell (\theta,Z)]v\ge  \big(c'D^{-4/d}-c''\|u_{f_{\theta_0}}-u_{f_\theta}\|_{L^2}\big) \|H\|_{L^2}^2 . \]
	Using (\ref{eq:zero-increm}) and the hypotheses, we obtain that for some $c_g>0$,
	\[  \|u_{f_{\theta_0}}-u_{f_\theta}\|_{L^2}\le  \|\mathcal G(\theta_0)-\mathcal G(\theta_{0,D})\|_{L^2} + c_g\|\theta_{0,D}-\theta\|_{\R^D}\le (c_1+c_g\epsilon_S)D^{-4/d}.\]
	Thus for all $c_1,\eps_S>0$ small enough and taking the infimum over $v\in \R^D$ with $\|v\|_{\R^D}=\|\Psi (v)\|_{L^2}=\|H\|_{L^2}=1$, we obtain that for any $\theta\in\mathcal B_{\epsilon_S}$ and some $c'''>0$,
	\[ \lambda_{min} \big(E_{\theta_0}[-\nabla^2 \ell(\theta,Z)] \big)\ge c'''D^{-4/d},\]
	which completes the proof.
\end{proof}

Finally, we prove the upper bound required for Assumption \ref{ass:geom} ii).

\begin{lem}[Upper bound]\label{lem:schr:ub}
	For every $S>0$, there exists a constant $c>0$ such that for $\|\theta_0\|_{h^{2}}\le S$ and all $D\in\N$, we have
	\[ \sup_{\theta\in\mathcal B_1} \Big[ |E_{\theta_0}[\ell(\theta,Z) ]|+ \| E_{\theta_0}[\nabla \ell(\theta,Z) ]\|_{\R^D} + \| E_{\theta_0}[\nabla^2 \ell(\theta,Z) ]\|_{op} \Big]\le c. \]
\end{lem}
\begin{proof}
	For the first term, using Lemma \ref{lem:schr:one}, we have that for some $K_0>0$ and any $\theta\in\mathcal B_1$,
	\[	|E_{\theta_0}[\ell(\theta)]|=1/2 +1/2 \|\mathcal G(\theta)-\mathcal G(\theta_0)\|_{L^2}^2 \lesssim 1+\|\mathcal G(\theta)\|_\infty^2+ \|u_{f_0}\|_\infty ^2\le K_0.\]
	For the first derivative, similarly by Lemma \ref{lem:schr:one} there exists some $K_1>0$ such that for any $\theta\in\mathcal B_1$,
	\begin{equation*}
	\begin{split}
	\big\|E_{\theta_0}\big[-\nabla \ell (\theta)\big]\big\|_{\R^D} \lesssim  \big\|\langle \mathcal G(\theta_0)-\mathcal G(\theta),\nabla \mathcal G(\theta)\rangle_{L^2(\mathcal O)}\big\|_{\R^D}\lesssim \big\| G(\theta_0)-\mathcal G(\theta)\big\|_{\infty} \big\|\nabla \mathcal G(\theta)\big\|_{L^\infty(\mathcal O,\R^D)}\le K_1.
	\end{split}
	\end{equation*}
	For the second derivative, we recall the decomposition
	\[	\lambda_{max}\big(E_{\theta_0}\big[-\nabla^2 \ell (\theta)\big]\big)=\sup_{v:\|v\|_{\R^D}\le 1}v^TE_{\theta_0}\big[-\nabla^2 \ell (\theta)\big]v= \sup_{v:\|v\|_{\R^D}\le1} \big[I+II+III\big],\]
	where the terms $I-III$ were defined in (\ref{eq:the-expected-one}). Suitable uniform upper bounds for the terms $II$ and $III$ have already been shown in (\ref{eq:I-upper}) and (\ref{eq:II-upper}) respectively, whence it suffices to upper bound the term $I$. We do this by using (\ref{hammer}) and Lemma \ref{cpebach}: for any $\theta\in\mathcal B_1$ and any $H=\Psi(v)$, $v\in\R^D$,
	\[ \sqrt I =\|V_{f_\theta}[u_{f_\theta}(\Phi'\circ F_\theta)H]\|_{L^2}\lesssim  \|u_{f_\theta}(\Phi'\circ F_\theta)H\|_{L^2}\lesssim \|u_{f_\theta}\|_\infty\|\Phi'\circ F_\theta\|_\infty\|H\|_{L^2}\lesssim \|v\|_{\R^D}. \]
\end{proof}

We conclude with the following basic comparison lemma for Sobolev norms on the subspaces $E_D\subseteq  L^2(\mathcal O)$ from (\ref{subba}).

\begin{lem}\label{normal} 
	 There exists $C>0$ such that for any $D\in \N$ and any $H \in E_D$, 
	\begin{equation}\label{eq:norm:equiv}
	\|H\|_{H^2}\leq CD^{2/d}\|H\|_{L^2}, ~~~~ \|H\|_{L^2}\leq CD^{2/d}\|H\|_{(H^2_0)^*}.
	\end{equation}
\end{lem}
\begin{proof}
	Fix $D\in \N$. By the isomorphism property of $\Delta$ between the spaces $H^2_0$ and $L^2$ (see e.g. Theorem II.5.4 in \cite{LM72}), we first have the norm equivalence
	\[ \|\Delta H\|_{L^2}\lesssim \|H\|_{H^2_0}\lesssim \|\Delta H\|_{L^2},~~~ H\in E_D. \]
	It follows by Weyl's law (\ref{weyl}) that
	\begin{equation*}
	\begin{split}
	\|H\|_{H^2_0}^2\lesssim \sum_{k=1}^D \big|\langle H,e_k\rangle_{L^2}\big|^2 \lambda_k^2\lesssim D^{4/d}\|H\|_{L^2}^2.
	\end{split}
	\end{equation*}
	Thus, combining the above display with the following duality argument completes the proof:
	\begin{equation*}
	\begin{split}
	\|H\|_{L^2}=\sup_{\psi\in E_D: \|\psi\|_{L^2}\le 1 }\big| \langle H, \psi\rangle_{L^2}\big|\lesssim D^{2/d}\sup_{\psi \in E_D:\|\psi\|_{H^2_0}\le 1 }\big| \langle H, \psi\rangle_{L^2}\big|\leq D^{2/d} \|H\|_{(H^2_0)^*}.
	\end{split}
	\end{equation*}
\end{proof}

\subsection{Wasserstein approximation of the posterior measure}\label{sec:freqpfs}

The main purpose of this section is to prove Theorem \ref{waterstone}, which provides a bound on the Wasserstein distance between the posterior measure $\Pi(\cdot|Z^{(N)})$ from (\ref{postbus}) and the surrogate posterior $\tilde \Pi(\cdot|Z^{(N)})$ from (\ref{surrod}) in the Schr\"odinger model. The idea behind the proof of this theorem is to show that both $\Pi(\cdot|Z^{(N)})$ and $\tilde \Pi(\cdot|Z^{(N)})$ concentrate most of their mass on the region (\ref{eq:B:def}) where the log-likelihood function $\ell_N$ is strongly concave (with high $P_{\theta_0}^N$-probability, cf.~Proposition \ref{schroe-emperor-smalldim}). We achieve this by a careful study of the mode (maximiser) of the posterior density, given in Theorem \ref{maprate}. Our derivations reveal that both $\pi(\cdot|Z^{(N)})$ and $\tilde \pi(\cdot|Z^{(N)})$ possess a \textit{unique} mode, with high frequentist probability (see after (\ref{score})).

\subsubsection{Convergence rate of MAP estimates}\label{maprat}
For $(Y_i, X_i)_{i=1}^N$ arising from (\ref{model}) with $\mathcal G:\R^D \to \R$ from (\ref{fwdG}), we now study maximisers
 \begin{equation}\label{maptheta}
\hat \theta_{MAP} \in \arg \max_{\theta \in \mathbb R^{D}} \left[-\frac{1}{2N}\sum_{i=1}^N\big(Y_i - \mathcal G(\theta)(X_i)\big)^2 - \frac{\delta_N^2}{2} \|\theta\|^2_{h^\alpha} \right],~~~\delta_N = N^{-\frac{\alpha}{2\alpha+d}},
\end{equation}
of the posterior density (\ref{postbus}). For $\Lambda_\alpha$ from (\ref{schrottprior}) we will write $I(\theta):=\frac{1}{2}\|\theta\|_{h^\alpha}^2 = \frac{1}{2} \theta^T \Lambda_\alpha \theta$ for $\theta \in \mathbb R^D$. We denote the empirical measure on $\mathbb R \times \mathcal O$ induced by the $Z_i=(Y_i,X_i)$'s as
\begin{equation} \label{empm}
P_N= \frac{1}{N}\sum_{i=1}^N \delta_{(Y_i, X_i)}, ~~\text{so that } \int h dP_N = \frac{1}{N} \sum_{i=1}^N h(Y_i, X_i)
\end{equation}
for any measurable map $h: \mathbb R \times \mathcal O \to \mathbb R$. Recall also that $p_\theta: \R \times \mathcal O \to [0,\infty)$ denotes the marginal probability densities of $P_\theta^N$ defined in (\ref{lik}).

\begin{lem} \label{basiclemma}
Let $\hat \theta_{MAP}$ be any maximiser in (\ref{maptheta}), and denote by $\theta_{0,D}$ the projection of $\theta_0$ onto $\mathbb R^D$. We have ($P^N_{\theta_0}$-a.s.)
$$\frac{1}{2}\|\mathcal G(\hat \theta_{MAP})- \mathcal G(\theta_{0})\|_{L^2}^2 +   \delta_N^2 I(\hat\theta_{MAP}) \le  \int \log \frac{p_{\hat \theta_{MAP}}}{p_{\theta_{0,D}}} d(P_N-P_{\theta_0}) +  \delta_N^2 I(\theta_{0,D}) + \frac{1}{2}\|\mathcal G(\theta_{0,D}) -\mathcal G(\theta_0)\|_{L^2}^2.$$
\end{lem}
\begin{proof}
By the definitions 
$$ \ell_N(\hat \theta_{MAP})- \ell_N(\theta_{0,D}) - N \delta_N^2 I(\hat \theta_{MAP}) \ge - N \delta_N^2 I(\theta_{0,D}) $$
which is the same as
\begin{equation}
N\int \log \frac{p_{\hat \theta_{MAP}}}{p_{\theta_{0,D}}} d(P_N-P_{\theta_0}) +  N \delta_N^2 I(\theta_{0,D}) \ge N \delta_N^2 I(\hat \theta_{MAP}) - N\int \log \frac{p_{\hat \theta_{MAP}}}{p_{\theta_{0,D}}} dP_{\theta_0}.
\end{equation}
The last term can be decomposed as
\begin{align*}
-\int \log \frac{p_{\hat \theta_{MAP}}}{p_{\theta_{0,D}}} dP_{\theta_0} &= -\int \log \frac{p_{\hat \theta_{MAP}}}{p_{\theta_{0}}} dP_{\theta_0} + \int \log \frac{p_{\theta_{0,D}}}{p_{\theta_0}} dP_{\theta_0} \\
&= \frac{1}{2}\|\mathcal G(\hat \theta_{MAP}) - \mathcal G(\theta_0)\|_{L^2(\mathcal O)}^2 - \frac{1}{2}\|\mathcal G(\theta_{0,D}) - \mathcal G(\theta_{0})\|_{L^2(\mathcal O)}^2
\end{align*}
where we have used a standard computation of likelihood ratios (see also Lemma 23 in \cite{GN19}). The result follows from the last two displays after dividing by $N$.
\end{proof}

\smallskip

The following result can be proved by adapting techniques from $M$-estimation \cite{V00} (see also \cite{V01}, \cite{NVW18}) to the present situation. We will make crucial use of the concentration Lemma \ref{mixchain}.

\begin{prop}\label{fwdrate}
Let $\alpha>d$. Suppose $\|\theta_0\|_{h^\alpha} \le c_0$ and that $D$ is such that $\|\mathcal G(\theta_0)-\mathcal G(\theta_{0,D})\|_{L^2} \le c_1 \delta_N$ for some $c_0, c_1>0$. Then, for any $c \ge 1$ we can choose $C=C(c, c_0, c_1)$ large enough so that every $\hat \theta_{MAP}$ maximising (\ref{maptheta}) satisfies,
\begin{equation}
P^N_{\theta_0}\left(\frac{1}{2}\|\mathcal G(\hat \theta_{MAP})- \mathcal G(\theta_{0})\|_{L^2}^2 +   \delta_N^2 I(\hat\theta_{MAP}) > C\delta_N^2  \right) \lesssim  e^{-c^2N\delta_N^2}.
\end{equation}
\end{prop}
\begin{proof}
We define functionals $$\tau(\theta, \theta') = \frac{1}{2}\|\mathcal G(\theta)-\mathcal G(\theta')\|_{L^2}^2 + \delta_N^2 I(\theta),~~\theta \in \mathbb R^D, \theta' \in h^\alpha,$$ and empirical processes $$W_N(\theta) =  \int \log \frac{p_{\theta}}{p_{\theta_{0,D}}} d(P_N-P_{\theta_0}), ~W_{N,0}(\theta) =  \int \log \frac{p_{\theta}}{p_{\theta_0}} d(P_N-P_{\theta_0}), ~\theta \in \mathbb R^D,$$ so that $$W_N(\theta)= W_{N,0}(\theta) - W_{N,0}(\theta_{0,D}),~~~\theta \in \mathbb R^D.$$
Using the previous lemma it suffices to bound
\begin{equation*}
P^N_{\theta_0}\left(\tau(\hat \theta_{MAP}, \theta_0) > C\delta_N^2, W_N(\hat \theta_{MAP}) \ge \tau(\hat \theta_{MAP}, \theta_0) -  \delta_N^2 I(\theta_{0,D}) - \|\mathcal G(\theta_{0,D}) -\mathcal G(\theta_0)\|_{L^2}^2/2 \right)
\end{equation*} 
Since $$I(\theta_{0,D}) =  \|\theta_{0,D}\|_{h^\alpha}^2/2 \le \|\theta_0\|_{h^\alpha}^2/2 \le c^2_0/2 \text{ and }\|\mathcal G(\theta_{0,D}) -\mathcal G(\theta_0)\|_{L^2}^2 \le c_1^2\delta_N^2$$ by hypothesis, we can choose $C$ large enough so that the last probability is bounded by 
\begin{align}
&P^N_{\theta_0}\left(\tau(\hat \theta_{MAP}, \theta_0) > C\delta_N^2, |W_N(\hat \theta_{MAP})| \ge \tau(\hat \theta_{MAP}, \theta_0)/2 \right) \notag \\
& \le \sum_{s =1}^\infty P^N_{\theta_0}\left(\sup_{\theta \in \mathbb R^D: 2^{s-1}C \delta_N^2 \le \tau(\theta, \theta_0) \le 2^sC \delta_N^2}|W_{N,0}(\theta)| \ge  2^{s}C \delta_N^2/8 \right) +P^N_{\theta_0}\big(|W_{N,0}(\theta_{0,D})| \ge  C \delta_N^2/4 \big) \notag \\
&  \le 2\sum_{s=1}^\infty P^N_{\theta_0}\left(\sup_{\theta \in\Theta_s}|W_{N,0}(\theta)| \ge  2^{s}C \delta_N^2/8 \right),  \label{peel}
\end{align} 
where, for $s \in \mathbb N$,
\begin{equation}
\Theta_s := \big\{\theta \in \mathbb R^D: \tau(\theta, \theta_0) \le 2^sC \delta_N^2 \big\} = \big\{\theta \in \mathbb R^D: \|\mathcal G(\theta) - \mathcal G(\theta_0)\|^2_{L^2} + \delta_N^2 \|\theta\|_{h^\alpha}^2 \le 2^{s+1}C \delta_N^2\big\},
\end{equation}
and where we have used that $\theta_{0,D} \in \Theta_1$ for $C$ large enough by the hypotheses. To proceed, notice that $$NW_{N,0}(\theta)=\ell_N(\theta) - \ell_N(\theta_{0}) - E_{\theta_0}[\ell_N(\theta) -  \ell_N(\theta_{0})]$$ and that,
for $(Y_i, X_i) \sim^{i.i.d.} P_{\theta_0}$,
\begin{align}\label{twoproc}
\ell_N(\theta) - \ell_N(\theta_{0}) &= -\frac{1}{2}\sum_{i=1}^N\big[(\mathcal G(\theta_0)(X_i)-\mathcal G(\theta)(X_i) + \varepsilon_i)^2 - \varepsilon_i^2 \big] \notag \\
&=- \sum_{i=1}^N(\mathcal G(\theta_0)(X_i)-\mathcal G(\theta)(X_i))\varepsilon_i -\frac{1}{2}\sum_{i=1}^N(\mathcal G(\theta_0)(X_i)-\mathcal G(\theta)(X_i))^2,
\end{align}
so that we have to deal with two empirical processes separately. We first bound
\begin{equation}\label{sum1}
\sum_{s =1}^\infty P^N_{\theta_0}\left(\sup_{\theta \in \Theta_s}|Z_N(\theta)| \ge  \sqrt N 2^{s}C \delta_N^2/16 \right)
\end{equation}
where $$Z_N=\frac{1}{\sqrt N}\sum_{i=1}^N h_\theta(X_i)\varepsilon_i,~~ h_\theta = \mathcal G(\theta_0) - \mathcal G(\theta),~~\theta \in \Theta = \Theta_s, s \in \mathbb N,$$ is as in Lemma \ref{mixchain}. We will apply that lemma with bounds (recalling $vol(\mathcal O)=1$)
\begin{equation}\label{envelope}
E^Xh^2_\theta(X) = \|\mathcal G(\theta)-\mathcal G(\theta_0)\|_{L^2}^2 \le 2^{s+1}C \delta_N^2 =: \sigma_s^2,~~\|h_\theta\|_\infty \le 2\sup_\theta \|\mathcal G(\theta)\|_\infty \le U<\infty 
\end{equation}
uniformly in all $\theta \in \Theta_s$, for some fixed constant $U=U(g,\mathcal O)$ (cf.~(\ref{ubd})).
For the entropy bounds, we use that on each slice $\sup_{\theta \in \Theta_s}\|F_\theta\|_{H^{\alpha}} \le \sqrt{2C}2^{s/2}$, which for $\alpha>d$ implies (using (4.184) in \cite{GN16} and standard extension properties of Sobolev norms) 
\begin{equation*}
\log N\big(\{F_\theta: \theta \in \Theta_s\}, \|\cdot\|_\infty, \rho \big) \le K\Big(\frac{\sqrt C2^{s/2}}{\rho}\Big)^{d/\alpha},~~\rho>0,
\end{equation*}
for some constant $K=K(\alpha,d)$. Since the map $F_\theta \mapsto \mathcal G(\theta)$ is Lipschitz for the $\|\cdot\|_\infty$-norm (Lemma \ref{lem:schr:two}) we deduce that also
\begin{equation}
\log N\big(\{h_\theta=\mathcal G(\theta)-\mathcal G(\theta_0): \theta \in \Theta_s\}, \|\cdot\|_\infty, \rho \big) \le K'\Big(\frac{\sqrt C2^{s/2}}{\rho}\Big)^{d/\alpha},~~\rho>0,
\end{equation}
and as a consequence, for $\alpha>d$ and $J_2(\mathcal H), J_\infty(\mathcal H)$ defined in Lemma \ref{mixchain},
\begin{equation}\label{entbd}
\begin{split}
J_2(\mathcal H) &\lesssim \int_0^{4\sigma_s}\Big(\frac{\sqrt C2^{s/2}}{\rho}\Big)^{d/2\alpha}d\rho \lesssim C^{d/4\alpha}2^{sd/4\alpha}\sigma_s^{1-\frac{d}{2\alpha}},\\
J_\infty(\mathcal H) &\lesssim \int_0^{4U}\Big(\frac{\sqrt C2^{s/2}}{\rho}\Big)^{d/\alpha}d\rho \lesssim C^{d/2\alpha}2^{sd/2\alpha} U^{1-\frac{d}{\alpha}}.
\end{split}
\end{equation}
The sum in (\ref{sum1}) can now be bounded by Lemma \ref{mixchain} with $x=c^2N2^s \delta_N^2$ and the choices of $\sigma_s, U$ in (\ref{envelope}) for $C>0$ large enough,
\begin{equation}\label{sum11}
\sum_{s=1}^\infty P^N_{\theta_0}\left(\sup_{\theta \in \Theta_s}|Z_N(\theta)| \ge  \sqrt N \sigma_s^2/32 \right) \le 2 \sum_{s \in \mathbb N} e^{-c^22^sN\delta_N^2} \lesssim e^{-c^2 N\delta_N^2}
\end{equation}
since then, by definition of $\delta_N$, for $\alpha>d$ and $C$ large enough, the quantities 
\begin{equation}\label{gaussiantail}
\mathcal J_2(\mathcal H) \lesssim  C^{d/4\alpha}2^{sd/4\alpha} (2^{s/2} \sqrt C\delta_N)^{1-\frac{d}{2\alpha}} \lesssim \frac{1}{\sqrt C}\sqrt N \sigma^2_s,~~\sigma_s \sqrt x  \le \frac{c}{\sqrt {2C}} \sqrt N \sigma_s^2,
\end{equation}
 and
 \begin{equation}\label{exponentialtail}
 \frac{1}{\sqrt N}\mathcal J_\infty(\mathcal H) \lesssim \frac{C^{d/2\alpha}2^{sd/2\alpha}}{\sqrt N} \lesssim \frac{1}{C^{d/2\alpha-1}} \sqrt N \sigma^2_s,~~\frac{x}{\sqrt N} = \frac{c^2}{2C} \sqrt N \sigma_s^2
 \end{equation}
are all of the correct order of magnitude compared to $\sqrt N \sigma_s^2$.

We now turn to the process corresponding to the second term in (\ref{twoproc}), which is bounded by
\begin{equation}\label{sum2}
\sum_{s \in \mathbb N} P^N_{\theta_0}\left(\sup_{\theta \in \Theta_s}|Z'_N(\theta)| \ge  \sqrt N 2^{s}C \delta_N^2/16 \right)
\end{equation}
where $Z'_N$ is now the centred empirical process 
$$Z'_N(\theta) = \frac{1}{\sqrt N} \sum_{i=1}^N(h_\theta-E^X h_\theta(X)\big),~~\text{ with }~  \mathcal H = \{ h_\theta=(\mathcal G(\theta)-\mathcal G(\theta_0))^2 : \theta \in \Theta_s\}$$
to which we will again apply Lemma \ref{mixchain}. Just as in (\ref{envelope}) the envelopes of this process are uniformly bounded by a fixed constant, again denoted by $U$, which implies in particular that the bounds (\ref{entbd}) also apply to $\mathcal H$ as then, for some constant $c_U>0$, $$ \|h_\theta- h_{\theta'}\|_\infty \leq c_U \|\mathcal G(\theta)-\mathcal G(\theta')\|_\infty.$$ Moreover on each slice $\Theta_s$ the weak variances are bounded by $$E^Xh_\theta^2(X) \le c'_U \|\mathcal G(\theta)-\mathcal G(\theta_0)\|_{L^2}^2 \le c'_U\sigma^2_s$$ with $\sigma_s$ as in (\ref{envelope}) and some $c'_U>0$. We see that all bounds required to obtain (\ref{sum1}) apply to the process $Z_N'$ as well, and hence the series in (\ref{peel}) is indeed bounded as required in the proposition, completing the proof.
\end{proof}

From a stability estimate for $\theta \mapsto \mathcal G(\theta)$ we now obtain the following convergence rate for $\|\hat \theta_{MAP} - \theta_0\|_{\ell^2}$ which in turn also bounds $\|\hat \theta_{MAP}-\theta_{0,D}\|_{\R^D}$. 

\begin{thm}\label{maprate}
Let $Z^{(N)} \sim P_{\theta_0}^N$ be as in (\ref{ZN}) where $\theta_0 \in h^\alpha, \alpha>d, d \le 3$. Define $$\bar \delta_N := N^{-r(\alpha)}~\text{ where } r(\alpha) = \frac{\alpha}{2\alpha +d} \frac{\alpha}{\alpha+2}.$$ Suppose $\|\theta_0\|_{h^\alpha} \le c_0$ and that $D$ is such that $\|\mathcal G(\theta_0)-\mathcal G(\theta_{0,D})\|_{L^2} \le c_1 \delta_N,$ for some constants $c_0, c_1>0$. Then given $c\ge 1$ we can choose $\bar C, \bar c$ large enough (depending on $c, c_0, c_1, \alpha, \mathcal O$) so that for all $N$ and any maximiser $\hat \theta_{MAP}$ satisfying (\ref{maptheta}), one has
\begin{equation}\label{betarate}
P^N_{\theta_0}\left(\|{\hat \theta_{MAP}}-\theta_0\|_{\ell^2}  \le \bar C\bar \delta_{N},~~\|\hat\theta_{MAP}\|_{h^\alpha} \le \bar C \right) \ge 1- \bar ce^{-c^2N\delta_N^2}.
\end{equation}
\end{thm}
\begin{proof}
By Proposition \ref{fwdrate} we can restrict to events 
\begin{equation} \label{TN}
T_N := \big\{\|\mathcal G(\hat \theta_{MAP})- \mathcal G(\theta_{0})\|_{L^2}^2 \le 2C \delta_N^2, \|F_{\hat \theta_{MAP}}\|_{H^\alpha} = \|\hat \theta_{MAP}\|_{h^\alpha} \le \sqrt{2C}\big\}
\end{equation}
 of sufficiently high $P^N_{\theta_0}$-probability. If we write $\hat f = \Phi \circ F_{\hat \theta_{MAP}}$ for $\Phi$ from (\ref{fwdG}) then by (\ref{links}), on the events $T_N$ we also have $~\|\hat f\|_{H^\alpha} \le C' $ and $\|\hat f\|_\infty \le C',$ for some $C'>0$. We write $ u_{\hat f} = \mathcal G(\hat \theta_{MAP})$ for the unique solution of the Schr\"odinger equation (\ref{eq:schr}) corresponding to $\hat f$. We then necessarily have $f = \Delta u_f/(2 u_f)$ both for $f= \hat f$ and $f=f_0$, where we also use that denominator $u_f$ is bounded away from zero by a constant $C''>0$ depending only on $\|f\|_\infty, \mathcal O, g$, see (\ref{leivand}). Then using the multiplication and interpolation inequalities (\ref{eq:h-mult}), (\ref{krankl}), the regularity estimate from (\ref{alphabound}) and (\ref{links}), we have for $t=\alpha/(\alpha+2)$,
\begin{align}
\|\hat f-f_0\|_{L^2} & \lesssim \|u_{\hat f} - u_{f_0}\|_{H^{2}} \notag \\
& \lesssim \|\mathcal G(\hat \theta_{MAP}) - \mathcal G(\theta_0)\|_{L^2}^t \|u_{\hat f} -u_{f_0}\|_{H^{\alpha+2}}^{1-t} \notag \\
& \lesssim \delta_N^t (\|\hat f\|_{H^\alpha}+\|f_0\|_{H^\alpha}) \lesssim \delta_N^t \label{interpol}
\end{align}
on the event $T_N$. From a Sobolev imbedding (for some $\kappa>0$) and applying (\ref{krankl}) again we further deduce $\|\hat f - f_0\|_\infty \lesssim \delta_N^{(\alpha-d/2-\kappa)/(\alpha+2)} \to 0$ as $N \to \infty$, hence using $\inf_{x}f_0(x)>K_{min}$ we also have $\inf_{x}\hat f(x) \ge K_{min}+k$ for some $k>0$ (on $T_N$, for all $N$ large enough). We deduce 
\begin{align*}
\|\hat \theta_{MAP}-\theta_{0}\|_{\ell^2} &\le \|F_{\hat \theta_{MAP}}- F_{\theta_0}\|_{L^2} = \|\Phi^{-1} \circ \hat f- \Phi^{-1} \circ f_0\|_{L^2} \lesssim \|\hat f-f_0\|_{L^2} \lesssim \delta_N^{t}
\end{align*} on the events $T_N$, where in the last inequality we have used regularity of the inverse link function $\Phi^{-1}: [K_{min}+k, \infty)$ and (\ref{linklip}). This completes the proof. 
\end{proof}

\subsubsection{Posterior contraction rates}

We now study the full posterior distribution (\ref{postbus}) arising from the Gaussian prior $\Pi$ for $\theta$ from (\ref{schrottprior}). The result we shall prove parallels Theorem \ref{maprate} but holds for most of the `mass' of the posterior measure instead of just for its `mode' $\hat \theta_{MAP}$. This requires very different techniques and we rely on ideas from Bayesian nonparametrics \cite{vdVvZ08, GV17}, specifically recent progress \cite{MNP21} that allows one to deal with non-linear settings (see also \cite{GN19}).

In the proof of Theorem \ref{waterstone} to follow we will require control of the posterior `normalising factors', expressed via sets
\begin{equation}\label{CN}
\mathcal C_N = \mathcal C_{N, K} = \left\{\int_{\mathbb R^D} e^{\ell_N(\theta)-\ell_N(\theta_0)} d\Pi(\theta) \ge \Pi(B(\delta_N))  \exp\{-(1+K)N\delta_N^2\}\right\},
\end{equation} 
for some $K>0$, where $\delta_N = N^{-\alpha/(2\alpha+d)}$ and $$B(\delta_N) = \big\{\theta \in \mathbb R^D: \|\mathcal G(\theta)-\mathcal G(\theta_0)\|_{L^2(\mathcal O)} \le \delta_N\big\}.$$ This is achieved in the course of the proof of our next result. We denote by $c_g$ the global Lipschitz constant of the map $\theta \mapsto \mathcal G(\theta)$ from $\ell^2(\mathbb N) \to L^2(\mathcal O)$, see (\ref{eq:zero-increm}).
\begin{thm}\label{contractionrate}
Let $Z^{(N)}, \theta_0, \alpha, d, \bar\delta_N$ be as in Theorem \ref{maprate} and let $\Pi(\cdot|Z^{(N)})$ denote the posterior distribution from (\ref{postbus}). Suppose $\|\theta_0\|_{h^\alpha}\le c_0$ and that $D \le c_2 N \delta_N^2$ is such that
\begin{equation}\label{sievecond}
\|\mathcal G(\theta_0) - \mathcal G(\theta_{0,D})\|_{L^2(\mathcal O)}  \le  c_1 \delta_N
\end{equation}
for some finite constants $c_0,c_2>0, 0<c_1<1/2$. Then for any $a>0$ there exist $c', c''$ such that for $K, L=L(a,c_0, c_2, c_g, \alpha, \mathcal O)$ large enough,
\begin{equation} \label{postrat}
P^N_{\theta_0}\big(\big\{\Pi(\theta: \|\theta-\theta_{0,D}\|_{\R^D} \le L \bar \delta_N, ~\|\theta\|_{h^\alpha} \le L |Z^{(N)}) \ge 1- e^{-aN\delta_N^2}\big\}, \mathcal C_{N,K}\big) \ge 1-c'e^{-c''N\delta_N^2}.
\end{equation}
\end{thm}
\begin{proof}
We initially establish some auxiliary results that will allow us to apply a standard contraction theorem from Bayesian non-parametrics, specifically in a form given in Theorem 13 in \cite{GN19}. By Lemma 23 in \cite{GN19} and (\ref{ubd})  we can lower bound $\Pi_N(\mathcal B_N)$ in (A5)~in \cite{GN19} by our $\Pi_N(B(\delta_N))$ (after adjusting the choice of $\delta_N$ in \cite{GN19} by a multiplicative constant). Then using (\ref{sievecond}), Corollary 2.6.18 in \cite{GN16}, and ultimately Theorem 1.2 in \cite{LL99} combined with (4.184) in \cite{GN16}, we have for $\theta' \sim N(0, \Lambda_\alpha^{-1})$,
\begin{align}
\Pi_N(\|\mathcal G(\theta)-\mathcal G(\theta_0)\|_{L^2(\mathcal O)} < \delta_N) & \ge  \Pi_N(\|\mathcal G(\theta)-\mathcal G(\theta_{0,D})\|_{L^2(\mathcal O)} <\delta_N/2) \notag \\
& \ge \Pi_N(\|\theta-\theta_{0,D}\|_{\mathbb R^D} <\delta_N/2c_g) \notag \\
& \ge e^{-N\delta_N^2 \|\theta_{0,D}\|^2_{h^\alpha}/2}\Pr(\|\theta'\|_{\mathbb R^D} < \sqrt N \delta_N^2/2c_g) \ge e^{-\bar d N\delta_N^2} \label{smallball}
\end{align}
for some $\bar d>0$. From this we deduce further from Borell's Gaussian iso-perimetric inequality \cite{B75} (in the form of Theorem 2.6.12 in \cite{GN16}), arguing just as in Lemma 17 in \cite{GN19} (and invoking the remark after that lemma with $\kappa=0$ there), that given $B>0$ we can find $M$ large enough (depending on $\bar d, B$) such that 
$$\Pi_N\big(\theta=\theta_1 + \theta_2 \in \R^D: \|\theta_1\|_{\R^D} \le M \delta_N, \|\theta_2\|_{h^\alpha} \le M\big) \ge 1- 2e^{-BN\delta_N^2}.$$ Next the eigenvalue growth $\lambda_k^\alpha \lesssim k^{2\alpha/d}$ from  (\ref{weyl}) and the hypothesis on $D$ imply that for $\bar L$ large enough we have
\begin{equation} \label{rkhstrick}
\|\theta_1\|_{h^\alpha} \lesssim D^{\alpha/d} \|\theta_1\|_{\mathbb R^D} \leq (c_2N \delta_N^2)^{\alpha/d}M \delta_N \le \bar L/2
\end{equation}
and then also
\begin{equation}
\Pi_N(\mathcal A_N^c) \le 2e^{-BN\delta_N^2} \text{ where }\mathcal A_N =\{\theta \in \mathbb R^D: \|\theta\|_{h^\alpha} \le \bar L\}.
\end{equation}
The $\|\cdot\|_\infty$-covering numbers of the implied set of regression functions $\mathcal G(\theta)$ satisfy the bounds 
\begin{align*}
\log N(\{\mathcal G(\theta): \theta \in \mathcal A_N\}, \|\cdot\|_{\infty}, \delta_N) &\lesssim \log N(\{F_\theta: \theta \in \mathcal A_N\}, \|\cdot\|_{\infty}, c\delta_N) \\
& \lesssim \log N(\{F: \|F\|_{H^\alpha(\mathcal O)} \le \bar L\}, \|\cdot\|_\infty, c\delta_N) \lesssim N \delta_N^2,
\end{align*}
for some $c>0$, using that the map $F_\theta \mapsto \mathcal G(\theta)$ is globally Lipschitz for the $\|\cdot\|_\infty$-norm (Lemma \ref{lem:schr:two}) and also the bound (4.184) in \cite{GN16}. By (\ref{ubd}) and Lemma 22 in \cite{GN19} the previous metric entropy inequality also holds for the Hellinger distance replacing $\|\cdot\|_\infty$-distance on the l.h.s.~in the last display. Theorem 13 and again Lemma 22 in \cite{GN19} now imply that for any $a>0$ there exists $L$ large enough,
\begin{equation}\label{rate1}
P^N_{\theta_0}\Big(\Pi(\{\theta: \|\mathcal G(\theta)-\mathcal G(\theta_0)\|_{L^2} > L \delta_N\} \cup \mathcal A^c_N|Z^{(N)}) \le e^{-aN\delta_N^2}\Big) \to 0
\end{equation}
as $N \to \infty$. The convergence in probability to zero obtained in the proof of Theorem 13 in \cite{GN19} is in fact exponentially fast, as required in (\ref{postrat}): This is true by virtue of the bound to follow in the next display (which forms part of the proof in \cite{GN19} as well), and since  the type-one testing errors in (39) in \cite{GN19} are controlled at the required exponential rate (via Theorem 7.1.4 in \cite{GN16}). The inequality
\begin{align*}
 P_{\theta_0}^N\Big(\int_{B(\delta_N)} e^{\ell_N(\theta)-\ell_N(\theta_0)} d\Pi(\theta) \ge \Pi(B(\delta_N))  \exp\{-(1+K)N\delta_N^2\}\Big)\le c'e^{-c''N\delta_N^2},
\end{align*}
bounding $P^N_{\theta_0}(\mathcal C_{N,K}^c)$ as required in the theorem follows from Lemma \ref{expsmall} below for large enough $K$ and $\bar C=1/2$. 

\medskip

Now to conclude, we can define subsets of $\mathbb R^D$ as $$\Theta_N := \{\theta: \|\mathcal G(\theta)-\mathcal G(\theta_0)\|_{L^2} \le L \delta_N\} \cap \mathcal A_N = \{\theta: \|\mathcal G(\theta)-\mathcal G(\theta_0)\|_{L^2} \le  L \delta_N, \|F_\theta\|_{H^\alpha} = \|\theta\|_{h^\alpha} \le \bar L \}$$ paralleling the events $T_N$ from (\ref{TN}) above. Then arguing as in and after (\ref{interpol}), one shows that
\begin{equation*}
\Theta_N \subset \tilde \Theta_N = \{\theta: \|\theta-\theta_{0}\|_{\R^D} \le L N^{-r(\alpha)}, \|\theta\|_{h^\alpha} \le L  \},
\end{equation*}
increasing also the constant $L$ if necessary, and hence the posterior probability of this event is also lower bounded by $\Pi(\tilde \Theta_N|Z^{(N)}) \ge 1 - e^{-aN\delta_N^2},$ with the desired $P^N_{\theta_0}$-probability, proving the theorem, since $\|\theta-\theta_{0}\|_{\ell^2} \le \|\theta-\theta_{0,D}\|_{\R^D}$.
\end{proof}

\medskip

Moreover, a quantitative uniform integrability argument from Section 5.4.5 in \cite{MNP21} (see the proof of Theorem \ref{waterstone}, term III, below) then also gives a convergence rate for the posterior mean $E^{\Pi}[\theta|Z^{(N)}]$ towards $\theta_0$, namely that for $L$ large enough there exist $\bar c', \bar c''>0$ such that
\begin{equation}\label{postmeanrat}
P^N_{\theta_0}\big(\|E^\Pi[\theta|Z^{(N)}]-\theta_0\|_{\ell^2} > L\bar \delta_N\big) \le \bar c' e^{-\bar c''N\delta_N^2}.
\end{equation}

\subsubsection{Globally log-concave approximation of the posterior in Wasserstein distance}\label{subsec:schr:wass}

Recall the surrogate posterior measure $\tilde \Pi(\cdot|Z^{(N)})$ from (\ref{surrod}) with log-density 
\begin{equation}
\log \tilde \pi_N(\theta) = const + \tilde \ell_N (\theta) - \frac{N \delta_N^2}{2}\|\theta\|_{h^\alpha}^2,~~\theta \in \mathbb R^D
\end{equation}
with $\theta_{init}$ and parameters $\epsilon, K$ chosen as in Condition \ref{asymptopia}, and with $\delta_N=N^{-\alpha/(2\alpha+d)}$. We now prove the main result of this section.
\begin{thm} \label{waterstone}
Assume Condition \ref{FAQ} and let $\tilde\Pi(\cdot|Z^{(N)})$ be the probability measure of density given in (\ref{surrod}) with $K,\eps>0$ chosen as in Condition \ref{asymptopia}. Then for some $a_1, a_2>0$ and all $N\in \N$,
$$P_{\theta_0}^N \big(W^2_2(\tilde \Pi(\cdot|Z^{(N)}), \Pi(\cdot|Z^{(N)}))> e^{-N \delta_N^2} \big) \leq a_1 e^{-a_2 N\delta_N^2}.$$ 
\end{thm}
\begin{proof}
In the proof we will require a new sequence 
\begin{equation}\label{tildedeltan}
\tilde \delta_N = N^{(-\alpha+2)/(2\alpha+d)} \sqrt{\log N}
\end{equation}
describing the `rate of contraction' of the surrogate posterior obtained below. 
We first notice that the definitions of $\bar \delta_N$ (from Theorem \ref{maprate}) and of $\delta_N$ imply by straightforward calculations and using $D \lesssim N \delta_N^2, \alpha>6$, the asymptotic relations as $N \to \infty$,
\begin{equation}\label{not-cosi}
\delta_N D^{2/d} \sqrt{\log N} =O(\tilde \delta_N), ~~\delta_N \ll \bar \delta_N \ll \tilde \delta_N \ll \frac{1}{\log N}D^{-\frac{4}{d}},
\end{equation}
which we shall use in the proof.  We will prove the bound for all $N$ large enough, which is sufficient to prove the desired inequality after adjusting the constant in $\lesssim$ (since probabilities are always bounded by one). 

\smallskip

\textbf{Geometry of the surrogate posterior.} To set things up, consider MAP estimates $\hat \theta_{MAP}$ from (\ref{maptheta}). In view of (\ref{ubd}), the function $q_N$ to be maximised over $\mathbb R^D$ in (\ref{maptheta}) satisfies $q_N(\theta)<q_N(0)$ for all $\theta$ such that $\|\theta\|_{h^\alpha}$ exceeds some positive constant $k$. Then on the compact set $M=\{\theta \in \R^D: \|\theta\|_{h^\alpha} \le k\}$ the function $q_N$ is continuous (as $\mathcal G$ is continuous from $\mathbb R^D \to L^\infty(\mathcal O)$, Lemma \ref{lem:schr:two}), and hence attains its maximum at some $\hat \theta_M \in M$, which must be a global maximiser of $q_N$ since $q_N(\hat \theta_M) \ge q_N(0) > \inf_{\theta \in M^c} q_N(\theta)$. Conclude that a maximiser $\hat \theta_{MAP}$ exists (one shows that it can be taken to be measurable, Exercise 7.2.3 in \cite{GN16}). 

\smallskip

In view of Proposition \ref{schroe-emperor-smalldim}, Theorem \ref{maprate}, Theorem \ref{triebelei} (and the remark before it) and $\alpha>6$, we may restrict ourselves in the rest of the proof to the following event
\begin{equation*}
	\begin{split}
		\mathcal S_N := \Big\{&\|\theta_{init} - \theta_{0,D}\|_{\R^D}  \le \frac{1}{8\log N D^{4/d}} \Big\} \cap \Big\{\inf_{\theta \in \mathcal B_{1/\log N}} \lambda_{min}(-\nabla^2 \ell_N(\theta)) \ge \underline cND^{-4/d}\Big\}\\
		&~~~\cap \Big\{\sup_{\theta \in \mathcal B_{1/\log N}}\Big[|\ell_N(\theta)|+\|\nabla \ell_N(\theta)\|_{\R^D}+\|\nabla^2\ell_N(\theta)\|_{op}\Big]< \underline c'N \Big\}\\
		&~~~ \cap \Big\{\text{any } \hat \theta_{MAP} \text{ satisfies }~\|\hat \theta_{MAP} - \theta_{0,D}\|_{\R^D} \le \min\big\{\frac{1}{8\log N D^{4/d}}, \bar C\bar\delta_N \big\}\Big\},
	\end{split}
\end{equation*}
where $\mathcal B_\epsilon$ was defined in (\ref{eq:B:def}), where $\bar C$ is from (\ref{betarate}) and where $\underline c=c_3$, $\underline c' =c_4$ from Proposition $\ref{schroe-emperor-smalldim}$. On $\mathcal S_N$ we have the following properties of $\tilde \ell_N$. First, from (\ref{eq:localization}),
\begin{equation}\label{38}
\tilde \ell_N(\theta)=\ell_N(\theta) \text{ for any } \theta ~s.t.~\|\theta-\theta_{0,D}\|_{\R^D}\le \frac{3}{8D^{4/d}\log N}.
\end{equation}
Moreover, by Proposition \ref{prop:log:conv}, $\log \tilde \pi(\cdot|Z^{(N)})$ is strongly concave in view of
\begin{equation}\label{hesse}
\sup_{\theta\in \mathcal B_{1/\log N}, \vartheta \in \mathbb R^D, \|\vartheta\|_{\mathbb R^D} = 1} \vartheta^T [\nabla^2 \log \tilde \pi_N(\theta)] \vartheta \le \sup_{\theta\in \mathcal B_{1/\log N}, \vartheta \in \mathbb R^D, \|\vartheta\|_{\mathbb R^D} = 1} \vartheta^T [\nabla^2 \tilde \ell_N(\theta)] \vartheta \le - \underline c ND^{-4/d}.
\end{equation}
Finally, any $\hat \theta_{MAP}$ necessarily satisfies
\begin{equation} \label{score}
0=\nabla \log \pi(\hat \theta_{MAP}|Z^{(N)}) = \nabla \log \tilde \pi(\hat \theta_{MAP}),
\end{equation}
from which we conclude that $\hat \theta_{MAP}$ necessarily equals the \textit{unique} global maximiser of the strongly concave function $\log \tilde \pi(\cdot|Z^{(N)})$ over $\mathbb R^D$.

%\begin{equation}\label{caltn}\mathcal T_N=\Big\{\|{\hat \theta_{MAP}}-\theta_{0,D}\|_{h^\beta}  \le \bar C \bar \delta_N(\beta), \beta \in \{0,d/2+\eta\},~~\|\hat\theta_{MAP}\|_{h^\alpha} \le \bar C \Big\},~\bar \delta_N = \bar \delta_N(0),\end{equation}

\medskip

\textbf{Decomposition of the Wasserstein distance.} Now let us write $$\hat {\mathcal B}(r) = \{\theta \in \mathbb R^D: \|\theta-\hat \theta_{MAP}\|_{\mathbb R^D} \le r \},~~ $$ for the Euclidean ball of radius $r>0$ centred at $\hat \theta_{MAP}$. Then using Theorem 6.15 in \cite{V09} with $x_0=\hat \theta_{MAP}$, we obtain for any $m>0$ that
\begin{align*}
W_2^2(\tilde \Pi(\cdot|Z^{(N)}), \Pi(\cdot|Z^{(N)})) &\le 2 \int_{\mathbb R^D} \|\theta-\hat \theta_{MAP}\|_{\mathbb R^D}^2 d|\tilde \Pi(\cdot|Z^{(N)})-\Pi(\cdot|Z^{(N)})|(\theta) \\
& \le 2  \int_{\hat {\mathcal B}(m\tilde \delta_N)} \|\theta-\hat \theta_{MAP}\|_{\mathbb R^D}^2 d|\tilde \Pi(\cdot|Z^{(N)})-\Pi(\cdot|Z^{(N)})|(\theta) \\
&~~~~+  2\int_{\mathbb R^D \setminus \hat {\mathcal B}(m\tilde \delta_N)} \|\theta-\hat \theta_{MAP}\|_{\mathbb R^D}^2 d|\tilde \Pi(\cdot|Z^{(N)})-\Pi(\cdot|Z^{(N)})|(\theta) \\
& \le  2m^2\tilde \delta_N^2~ \int_{\hat {\mathcal B}(m\tilde \delta_N)} d|\Pi(\cdot|Z^{(N)})-\tilde \Pi(\cdot|Z^{(N)})|(\theta)\\
&~~~~~+ 2\int_{\|\theta-\hat \theta_{MAP}\|_{\R^D} >  m\tilde \delta_N}\|\theta-\hat \theta_{MAP}\|_{\mathbb R^D}^2 d\tilde \Pi(\theta|Z^{(N)}) \\
&~~~~~+2\int_{\|\theta-\hat \theta_{MAP}\|_{\R^D} >  m\tilde \delta_N}\|\theta-\hat \theta_{MAP}\|_{\mathbb R^D}^2 d \Pi(\theta|Z^{(N)}) \\
& \equiv  I +II + III,
\end{align*}
and we now bound $I,II, III$ in separate steps. 

\smallskip

\textbf{Term II.} We can write the surrogate posterior density as $$\tilde \pi(\theta|Z^{(N)}) = \frac{e^{\tilde \ell_N(\theta)- \tilde \ell_N(\hat \theta_{MAP})}\pi(\theta)}{\int_{\mathbb R^D}e^{\tilde \ell_N(\theta)- \tilde \ell_N(\hat \theta_{MAP})}\pi(\theta)d\theta},~~\theta \in \mathbb R^D,$$ and will first lower bound the normalising factor. From  (\ref{not-cosi}) we have for any $c>0$ the set inclusion $$B_N \equiv \{\|\theta - \theta_{0,D}\|_{\mathbb R^D} \le c \delta_N\}\subset \Big\{\|\theta- \theta_{0,D}\|_{\R^D} \le \frac{3}{8D^{4/d}\log N} \Big\}$$ whenever $N$ is large enough. Since $\ell_N(\theta)=\tilde \ell_N(\theta)$ on the last set we have on an event of large enough $P^N_{\theta_0}$-probability,
\begin{align*}
\int_{\mathbb R^D} e^{\tilde \ell_N(\theta)-\tilde \ell_N(\hat \theta_{MAP})}d\Pi(\theta) &\ge \int_{B_N} e^{\tilde \ell_N(\theta)-\tilde \ell_N(\hat \theta_{MAP})} d\Pi(\theta) \\
& = \int_{B_N} e^{\ell_N(\theta)- \ell_N(\hat \theta_{MAP})} d\nu(\theta) \times \Pi(B_N)  \ge e^{-\bar c N \delta_N^2}
\end{align*}
for some $\bar c = \bar c(\bar d, c)$, where we have used Lemma \ref{expsmall} for our choice of $B_N$ (permitted for appropriate choice of $c>0$ by (\ref{dimbias}) and since $\mathcal G: \mathbb R^D \to L^2$ is Lipschitz, see Appendix \ref{sec:aux}) with $\nu = \Pi(\cdot)/\Pi(B_N), \bar C=1/2;$ as well as the small ball estimate for $\Pi$ in (\ref{smallball}).

Now recall the prior (\ref{schrottprior}) and define scaling constants
$$V_N = (2\pi)^{-D/2}\sqrt{\det(N \delta_N^2 \Lambda_\alpha)} \times e^{\bar cN\delta_N^2}.$$
Then on the preceding events the term II can be bounded, using a second order Taylor expansion of $\log \tilde \pi(\cdot|Z^{(N)})$ around its maximum $\hat \theta_{MAP}$ combined with (\ref{hesse}), (\ref{score}), as
\begin{align*}
 &\int_{\|\theta-\hat \theta_{MAP}\|_{\mathbb R^D} >  m\tilde \delta_N}\|\theta-\hat \theta_{MAP}\|_{\mathbb R^D}^2 \tilde \pi(\theta|Z^{(N)})d\theta \\
 &\le e^{\bar cN\delta_N^2} \int_{\|\theta-\hat \theta_{MAP}\|_{\mathbb R^D} >  m\tilde \delta_N}\|\theta-\hat \theta_{MAP}\|_{\mathbb R^D}^2 e^{\tilde \ell_N(\theta)- \tilde \ell_N(\hat \theta_{MAP})}\pi(\theta)d\theta \\
 &\le V_N \times \int_{\|\theta-\hat \theta_{MAP}\|_{\mathbb R^D} >  m\tilde \delta_N}\|\theta-\hat \theta_{MAP}\|_{\mathbb R^D}^2 e^{\tilde \ell_N(\theta)-\frac{N\delta_N^2}{2}\|\theta\|^2_{h^{\alpha}} - \tilde \ell_N(\hat \theta_{MAP}) + \frac{N\delta_N^2}{2}\|\hat \theta_{MAP}\|^2_{h^{\alpha}}}d\theta \\
 &= V_N \times \int_{\|\theta-\hat \theta_{MAP}\|_{\mathbb R^D} >  m\tilde \delta_N}\|\theta-\hat \theta_{MAP}\|_{\mathbb R^D}^2 e^{\log \tilde \pi_N(\theta)-\log \tilde \pi_N(\hat \theta_{MAP})}d\theta \\
 & \le V_N \times \int_{\|\theta-\hat \theta_{MAP}\|_{\mathbb R^D} >  m\tilde \delta_N}\|\theta-\hat \theta_{MAP}\|_{\mathbb R^D}^2 e^{-\underline cND^{-4/d} \|\theta- \hat \theta_{MAP}\|^2_{\mathbb R^D}/2}d\theta \\
 &\leq 2 V_N \times \big(\frac{4\pi}{\underline c ND^{-4/d}}\big)^{D/2}  \Pr\big(\|Z\|_{\mathbb R^D}>m \tilde \delta_N \big)
\end{align*}
where we have used $x^2 e^{-cx^2} \le  2e^{-cx^2/2}$ for all $x \in \R$, $c \ge 1$ (and $N$ such that $\underline c ND^{-4/d} \ge 1$) and where $$Z \sim N \Big(0, \frac{2}{\underline c D^{-4/d} N} I_{D \times D}\Big).$$ Now by $D \le c_0 N \delta_N^2$ and (\ref{not-cosi}), $$E\|Z\|_{\mathbb R^D} \le \sqrt{E\|Z\|_{\mathbb R^D}^2} \le \sqrt{2D/(\underline c D^{-4/d}N)} \leq (2 c_0/\underline c)^{1/2} \delta_ND^{2/d} \le (m/2) \tilde \delta_N$$ for $m$ large enough, so that 
$$\Pr\big(\|Z\|_{\mathbb R^D}>m \tilde \delta_N \big)\le \Pr\big(\|Z\|_{\mathbb R^D}-E\|Z\|_{\mathbb R^D}>(m/2) \tilde \delta_N \big) \le e^{-m^2 \underline c ND^{-4/d} \tilde \delta_N^2/16}$$ by a concentration inequality for Lipschitz-functionals of $D$-dimensional Gaussian random vectors (e.g., Theorem 2.5.7 in \cite{GN16} applied to $(\underline c N D^{-4/d}/2)^{1/2}Z \sim N(0,I_{D \times D})$ and $F=\|\cdot\|_{\mathbb R^D}$). By (\ref{weyl}) and since $D \lesssim N \delta_N^2$ we have for some $c'>0$
$$V_N \leq e^{c' N\delta_N^2 \log N}$$ so that for $m$ large enough and using (\ref{not-cosi}), the last term in the  displayed array above, and hence $II/2$ is bounded by
$$2 V_N \times \big(\frac{4\pi}{\underline c ND^{-4/d}}\big)^{D/2}  \times e^{-m^2 \underline c D^{-4/d} N \tilde \delta_N^2/16} \leq e^{-m^2 D^{-4/d} N \tilde \delta_N^2/32} \le \frac{1}{8}e^{-N\delta_N^2}.$$

\medskip

\textbf{Term III:} We first note that Theorem \ref{contractionrate} and (\ref{not-cosi}) imply that for every $a>0$ we can find $m$ large enough such that
\begin{align*}
\Pi(\|\theta-\hat \theta_{MAP}\|_{\mathbb R^D} >  m\tilde \delta_N |Z^{(N)}) &\le \Pi(\|\theta-\theta_{0,D}\|_{\mathbb R^D} >  m\bar \delta_N - \|\hat \theta_{MAP} - \theta_{0,D}\|_{\mathbb R^D} |Z^{(N)}) \\
&\le \Pi(\|\theta- \theta_{0,D}\|_{\mathbb R^D} >  m\bar \delta_N/2 |Z^{(N)}) \le e^{-aN \delta_N^2} 
\end{align*}
on events $\mathcal S_N' \subset \mathcal S_N$ of sufficiently high probability. Moreover, again by Theorem \ref{contractionrate}, we can further restrict the argument that follows to the event $\mathcal C_{N, K}$ from (\ref{CN}) for some $K>0$. Now using the Cauchy-Schwarz and Markov inequalities as well as $E^N_{\theta_0}e^{\ell_N(\theta) - \ell_N(\theta_0)}=1$ and the small ball estimate for $\Pi$ in (\ref{smallball}), we have
\begin{align*}
&P^N_{\theta_0}\Big(\mathcal C_{N, K} \cap \mathcal S'_N, \int_{\|\theta-\hat \theta_{MAP}\|_{\mathbb R^D} >  m\tilde \delta_N}\|\theta-\hat \theta_{MAP}\|_{\mathbb R^D}^2 d \Pi(\theta|Z^{(N)})>e^{-N\delta_N^2}/8\Big)  \\
&\le P^N_{\theta_0}\Big(\mathcal C_{N, K} \cap \mathcal S'_N, \Pi(\|\theta-\hat \theta_{MAP}\|_{\mathbb R^D} >  m\tilde \delta_N |Z^{(N)}) E^\Pi[\|\theta-\hat \theta_{MAP}\|_{\mathbb R^D}^4 |Z^{(N)}]>e^{-2N\delta_N^2}/64\Big) \\
& \le P_{\theta_0}^N\Big(\mathcal S_N', e^{(1+K+\bar d+2-a)N\delta_N^2} \int_{\mathbb R^D} \|\theta-\hat \theta_{MAP}\|_{\mathbb R^D}^4 e^{\ell_N(\theta) - \ell_N(\theta_0)} d\Pi(\theta) >1/64\Big) \\
&\lesssim e^{(1+K+\bar d+2-a)N \delta_N^2} \int_{\mathbb R^D} (1+\|\theta\|_{\mathbb R^D}^4 )d\Pi(\theta) \le e^{-a_2N \delta_N^2}
\end{align*}
whenever $m$ and then $a$ are large enough, since $\Pi$ has uniformly bounded fourth moments and since $\|\hat \theta_{MAP}\|_{\mathbb R^D}$ is uniformly bounded by a constant depending only on $\|\theta_0\|_{\ell^2}$ on the events $\mathcal S_N$.

\smallskip

\textbf{Term I:} On the events $\mathcal S_N$ we have from (\ref{not-cosi}) that for fixed $m>0$ and all $N$ large enough $$\hat {\mathcal B}(m\tilde \delta_N)\subseteq \{\theta:\|\theta-\theta_{0,D}\|_{\R^D}\le 3/(8 D^{4/d} \log N) \}.$$ On the latter set, by (\ref{38}), the probability measures $\tilde \Pi(\cdot|Z^{(N)})$ and $\Pi(\cdot|Z^{(N)})$ coincide up to a normalising factor, and thus we can represent their Lebesgue densities as $$\tilde \pi(\theta|Z^{(N)})=p_N \pi(\theta|Z^{(N)}),~~\theta \in \hat {\mathcal B}(m\tilde \delta_N),$$ for some $0<p_N<\infty$. Moreover, by the preceding estimates for terms II and III (which hold just as well without the integrating factors $\|\theta-\hat \theta_{MAP}\|_{\mathbb R^D}^2$), we have both
$$ p_N \Pi(\hat {\mathcal B}(m\tilde \delta_N)|Z^{(N)}) = \tilde \Pi(\hat {\mathcal B}(m\tilde \delta_N)|Z^{(N)})  \ge 1 - e^{-N\delta_N^2}/8~\Rightarrow~ 1 - e^{-N\delta_N^2}/8 \le p_N,$$
$$  p^{-1}_N \tilde \Pi(\hat {\mathcal B}(m\tilde \delta_N)|Z^{(N)}) = \Pi(\hat {\mathcal B}(m\tilde \delta_N)|Z^{(N)}) \ge 1 - e^{-N\delta_N^2}/8~\Rightarrow~ 1 - e^{-N\delta_N^2}/8 \le \frac{1}{p_N}$$ on events of sufficiently high $P_{\theta_0}^N$-probability. On these events necessarily $$p_N \in \Big[1- \frac{e^{-N\delta_N^2}}{8}, \frac{1}{1-\frac{e^{- N\delta_N^2}}{8}}\Big]$$
and so for $N$ large enough
$$\int_{\hat {\mathcal B}(m\tilde \delta_N)} d|\Pi(\cdot|Z^{(N)})-\tilde \Pi(\cdot|Z^{(N)})|(\theta) =|1-p_N|\int_{\hat {\mathcal B}(m\tilde \delta_N)}  \pi(\theta|Z^{(N)}) d\theta \le |1-p_N| \le e^{-N \delta_N^2}/4,$$ which is obvious for $p_N \le 1$ and follows from the mean value theorem applied to $f(x)=(1-x)^{-1}$ near $x=0$ also for $p_N>1$. Collecting the  bounds for $I, II, III$ completes the proof.
\end{proof}

\subsubsection{An `exponential' small ball lemma}

\begin{lem}\label{expsmall}
Let $\mathcal G$ be as in (\ref{fwdG}) and let $\nu$ be a probability measure on some ($\ell^2(\mathbb N)$-measurable) set 
\begin{equation} \label{ischgl}
B_N \subseteq \big\{\theta \in h^\alpha: \|\mathcal G(\theta)-\mathcal G(\theta_{0})\|_{L^2}^2 \le 2\bar C\delta_N^2 \big\}, \text{ for some }\bar C>0.
\end{equation}
Then for $\ell_N$ from (\ref{loglik}) there exists $b>0$ such that for every $K>0$ large enough,
\begin{equation}
P^N_{\theta_0}\left(\int_{B_N} e^{\ell_N(\theta)-\ell_N(\hat \theta_{MAP})}d\nu(\theta) \le e^{-(1+K)\bar C^2N\delta_N^2} \right) \lesssim e^{-b N \delta_N^2}.
\end{equation} The same conclusion holds true with $\ell_N(\hat \theta_{MAP})$ replaced by $\ell_N(\theta_0)$.
\end{lem}
\begin{proof}
We proceed as in Lemma 7.3.2 in \cite{GN16} to deduce from Jensen's inequality (applied to $\log$ and $\int (\cdot)d\nu$) that, for $P_N$ the empirical measure from (\ref{empm}), the probability in question is bounded by
$$P^N_{\theta_0}\left(\int \int_{B_N} \log \frac{p_\theta}{p_{\hat \theta_{MAP}}}d\nu(\theta)d(P_N-P_{\theta_0}) \le -(1+K)\bar C^2 \delta_N^2 - \int \int_{B_N} \log  \frac{p_\theta}{p_{\hat \theta_{MAP}}}d\nu(\theta)dP_{\theta_0} \right).$$  Now just as in the proof of Lemma \ref{basiclemma} we see that for all $\theta\in B_N$,
\begin{align*}
-\int \log  \frac{p_\theta}{p_{\hat \theta_{MAP}}}dP_{\theta_0} &= -\int \log  \frac{p_\theta}{p_{\theta_0}}dP_{\theta_0} - \int \log  \frac{p_{\theta_0}}{p_{\hat \theta_{MAP}}}dP_{\theta_0} \\
& = \frac{1}{2}\|\mathcal G(\theta)-\mathcal G(\theta_0)\|_{L^2}^2 - \frac{1}{2} \|\mathcal G(\hat \theta_{MAP}) - \mathcal G(\theta_0)\|_{L_2}^2 \le \bar C^2 \delta_N^2
\end{align*}
 so that using also Fubini's theorem the last probability can be bounded by 
\begin{align*}
&P^N_{\theta_0}\left(\sqrt N \int \int_{B_N} \log \frac{p_{\theta_0}}{p_\theta}d\nu(\theta)d(P_N-P_{\theta_0}) \ge K \bar C^2\sqrt N\delta_N^2/2\right) \\
& + P^N_{\theta_0}\left(\sqrt N \int \log \frac{p_{\hat \theta_{MAP}}}{p_{\theta_0}}d(P_N-P_{\theta_0}) \ge K \bar C^2\sqrt N \delta_N^2/2\right).
\end{align*}
For the first probability we decompose as in (\ref{twoproc}) and consider $Z_N$ as in Lemma \ref{mixchain} for fixed $h_\theta$ equal to  either $h_1$ or $h_2$, where $$h_1(x) = \int_{B_N} (\mathcal G(\theta)(x)- \mathcal G(\theta_0)(x)) d\nu(\theta),~\text{ and}~h_2(x) =  \int_{B_N} (\mathcal G(\theta)(x)- \mathcal G(\theta_0)(x))^2 d\nu(\theta).$$ To each of these we apply Bernstein's inequality (\ref{bernstein}) with $x= N\sigma^2$ and $K$ large enough to obtain the desired exponential bound, using uniform boundedness $\|\mathcal G(\theta)-\mathcal G(\theta_0)\|_\infty \le 2U$ from (\ref{ubd}) and Jensen's inequality in the variance estimates $E^Xh^2_1(X) \le 2\bar C^2 \delta_N^2 \equiv \sigma^2$ in the first case and $$E^Xh^2_2(X) \le 4U^2 \int_{B_N}\|\mathcal G(\theta)-\mathcal G(\theta_0)\|_{L^2}^2 d\nu(\theta) \le 8U^2 \bar C\delta_N^2 \equiv \sigma^2$$ for the second case. [This already proves the case where $\hat \theta_{MAP}$ is replaced by $\theta_0$.]

For the second probability, restricting to the event in the supremum below, which has sufficiently high $P_{\theta_0}^N$-probability in view of  Proposition \ref{fwdrate}, it suffices to bound for some $C>0$,
$$P^N_{\theta_0}\left(\sup_{\|\theta\|_{h^\alpha} \le  2C, \|\mathcal G(\theta)-\mathcal G(\theta_0)\|^2_{L_2} \le  2C \delta^2_N} \sqrt N\Big|\int \log \frac{p_\theta}{p_{\theta_0}}d(P_N-P_{\theta_0}) \Big| \ge K \bar C^2\sqrt N\delta_N^2/2\right).$$ This term corresponds to the empirical process bounded in and after (\ref{peel}) for $s=1$. Choosing $K$ large enough the proof there now applies directly, giving the desired exponential bound. 
\end{proof}

\appendix

\section{Review of convergence guarantees for ULA}\label{app:ULA}
In this section we collect some key results (that were used in our proofs) about convergence guarantees for an Unadjusted Langevin Algorithm (ULA) for sampling from \textit{strongly log-concave target measures}, see \cite{D17, DM17, DM18} and also the classical reference \cite{RT96}.  Our presentation follows the recent article \cite{DM18}.

Suppose that $\mu$ is a Borel probability measure on $\R^D$ which has a Lebesgue density proportional to $e^{-U}$  for some potential $U:\R^D\to \R$, specifically
\begin{equation}\label{brillenschlange}
\mu(B)=\frac{\int_{B}e^{-U(\theta)}d\theta}{\int_{\R^D}e^{-U(\theta)}d\theta},~~~B \subseteq \R^D \text{ measurable}.
\end{equation}
Following \cite{DM18} (cf.~H1 and H2 there) we will assume that the potential $U$ has a $\Lambda$-Lipschitz gradient and is $m$-strongly convex.
\begin{ass}\label{ass:ULA}
	1. The function $U:\R^D\to \R$ is continuously differentiable and there exists a constant $\Lambda\ge 0$ such that for all $\theta,\bar\theta \in\R^D$,
	\[\|\nabla U(\theta)-\nabla U(\bar \theta)\|_{\R^D}\le \Lambda \|\theta-\bar\theta\|_{\R^D}. \] 
	\par 
	2. There exists a constant $0<m\le \Lambda$ such that for all $\theta,\bar \theta\in\R^D$, we have
	\[U(\bar \theta)\ge U(\theta) +\langle \nabla U(\theta),\bar\theta-\theta \rangle_{\R^D}+\frac m2 \|\theta-\bar\theta\|_{\R^D}^2. \]
\end{ass}

Under Assumption \ref{ass:ULA}, the potential $U$ has a unique minimiser over $\mathbb R^D$, which we shall denote by $\theta_{U}$. For the computation of $\theta_U$ via gradient descent methods, we have the following standard result from convex optimisation (see Theorem 1 in \cite{D17} and (9.18) in \cite{BV04}).

\begin{prop}\label{reiskorn}
	Suppose $U:\R^D\to \R$ satisfies Assumption \ref{ass:ULA}. Then the gradient descent algorithm given by
	\[ \vartheta_{k+1}=\vartheta_{k} - \frac{1}{2\Lambda} \nabla U(\vartheta_{k}),~~~ k=0,1,2,\dots, \]
	satisfies that
	\[ \|\vartheta_{k}-\theta_U\|_{\R^D}^2 \le \frac{2(U(\vartheta_0)-U(\theta_U))}{m}\big(1-\frac{m}{2\Lambda} \big)^k,~~~ k=0,1,2,\dots \]
\end{prop}

The results presented below establish corresponding geometric convergence bounds for \textit{stochastic} gradient methods which target the entire probability measure $\mu$ (instead of just its mode $\theta_U$). Define the continuous time Langevin diffusion process as the unique strong solution $(L_t: t \ge 0)$ of the stochastic differential equation
\begin{equation}\label{eq:langevin}
dL_t=-\nabla U(L_t)dt +\sqrt 2 dW_t, ~~~ t\ge 0,~L_t \in \R^D,
\end{equation}
where $(W_t:t\ge 0)$ is a $D$-dimensional standard Brownian motion. It is well known that the Markov process $(L_t: t \ge 0)$ has $\mu$ from (\ref{brillenschlange}) as its invariant measure. The Euler-Maruyama discretisation of the dynamics (\ref{eq:langevin}) gives rise to the discrete-time Markov chain $(\vartheta_k:k\ge 0)$,
\begin{equation}\label{eq:ULA}
\vartheta_{k+1}=\vartheta_{k}-\gamma \nabla U(\vartheta_k)+\sqrt{2\gamma}\xi_{k+1},~~~ k\ge 0,
\end{equation}
where $(\xi_{k}:k\ge 1)$ form an i.i.d.~sequence of $D$-dimensional standard Gaussian $N(0, I_{D\times D})$ vectors, and $\gamma>0$ is some fixed \textit{step size}. We will refer to $(\vartheta_k)$ as the unadjusted Langevin algorithm (ULA) in what follows. We denote by $\mathbf P_{\theta_{init}}, \mathbf E_{\theta_{init}}$ the law and expectation operator, respectively, of the Markov chain $(\vartheta_k:k \ge 1)$ when started at a deterministic point $\vartheta_0=\theta_{init}$. We also write $\mathcal L(\vartheta_k)$ for the (marginal) distribution of the $k$-th iterate $\vartheta_k$. 
\par 
For any measurable function $H:\R^D\to \R$ and any $J_{in},J\ge 0$, let us define the average of $H$ along an ULA trajectory after `burn-in' period $J_{in}$ by 
\[ \hat \mu_{J_{in}}^J(H) =\frac 1J\sum_{k=J_{in}+1}^{J_{in}+J}H(\vartheta_k). \]

\begin{prop}\label{prop:ULA:conc}
	Suppose that $U$ satisfies Assumption \ref{ass:ULA} and suppose $\gamma \le 2/(m+\Lambda)$. Then for all $J, J_{in}\ge 1, x>0$ and any Lipschitz function $H:\R^D\to \R$, we have the concentration inequality
	\[\mathbf P_{\theta_{init}}\Big(\hat \mu_{J_{in}}^J(H)-\mathbf E_{\theta_{init}}[\hat \mu_{J_{in}}^J(H)]\ge x\Big)\le \exp\Big(-\frac{J\gamma x^2m^2}{16\|H\|_{Lip}^2(1+2/(mJ\gamma))}\Big). \]
\end{prop}
\begin{proof}
	The statement follows directly from Theorem 17 of \cite{DM18}, noting that $\kappa=2m\Lambda /(m+\Lambda)\in [m,2m]$ and that the constant $v_{N,n}(\gamma)$ from (28) of \cite{DM18} can be upper bounded by
	\[ 1 + \frac{m^{-1}+2/(m+\Lambda)}{\gamma J}\le 1 + 2/ (m\gamma J). \]
\end{proof}

\begin{prop}\label{prop:ULA:wass}
	Suppose that $U$ satisfies Assumption \ref{ass:ULA} and let $\gamma, J_{in}, J$ and $H$ be as in Proposition \ref{prop:ULA:conc}. Then we have for $\mu$ as in (\ref{brillenschlange}) that 
	\begin{equation}\label{eq:ULA:wass}
	W^2_2(\mathcal L(\vartheta_k),\mu)\le 2\big(1-m\gamma /2 \big)^k\Big[ \|\theta_{init}-\theta_U\|^2_{\R^D}+\frac Dm \Big] + b(\gamma)/2,~~~~ k\ge 0,
	\end{equation}
	where
	\begin{equation}\label{eq:bgamma}
	b(\gamma)= 36 \frac{\gamma D \Lambda^2}{m^2} + 12 \frac{\gamma^2 D \Lambda^4}{m^3},
	\end{equation}
	as well as
	\begin{equation}\label{eq:ULA:bias}
	\Big(\mathbf E_{\theta_{init}} [\hat \mu_{J_{in}}^J(H)]-E_\mu H\Big)^2\le \|H\|_{Lip}^2~\frac{1}{J}~\sum_{k=J_{in}+1}^{J_{in}+J}W^2_2(\mathcal L(\vartheta_k),\mu).
	\end{equation}
\end{prop}
\begin{proof}
	The display (\ref{eq:ULA:bias}) is derived in (27) of \cite{DM18}.	The bound (\ref{eq:ULA:wass}) follows from an application of Theorem 5 in \cite{DM18} with fixed step size $\gamma>0$, where in our case, noting again that $\kappa \in [m,2m]$, the expression $u_n^{(1)}(\gamma)$ there is upper bounded by $2\big(1-m\gamma /2 \big)^k$ and the expression $u_n^{(2)}(\gamma)$ there is upper bounded by (using that $\gamma\le \min\{ 2/\Lambda, 1/m \}\le \min\{ 2/\Lambda, 2/\kappa \}$)
	\begin{equation*}
	\begin{split}
	&\Lambda ^2D\gamma ^2\big(\kappa^{-1}+\gamma\big) \Big(2+\frac{\Lambda^2\gamma}{m}+\frac{\Lambda^2\gamma^2}{6} \Big)\sum_{i=1}^k(1-\kappa \gamma/2)^{k-i}\\
	&~~~\le \Lambda ^2D\gamma ^2\big(\kappa^{-1}+\gamma\big) \Big(2+\frac{\Lambda^2\gamma}{m}+\frac{\Lambda^2\gamma^2}{6} \Big)\frac{2}{\kappa\gamma}\\
	&~~~\le \Lambda ^2 D\gamma \Big(\kappa^{-2}+\frac{\gamma}{\kappa}\Big) \Big(6+\frac{2\Lambda^2\gamma}{m}\Big)\\
	&~~~\le \Lambda ^2 D\gamma m^{-2}  \Big (18+\frac{6\Lambda^2\gamma}{m}\Big),
	\end{split}
	\end{equation*}
	which equals (\ref{eq:bgamma}).
\end{proof}

\section{Auxiliary results}\label{sec:aux}

\subsection{Analytical properties of Schr\"odinger operators and link functions}
Recall the inverse Schr\"odinger operators $V_f$ from (\ref{Vf}).

\begin{lem}\label{cpebach} There exists a constant $C>0$ such that for any $f\in C(\mathcal O)$ with $f\ge 0$, the following holds.
	\begin{enumerate}[label=\textbf{\roman*)}]
		\item We have the estimates
		\begin{equation}\label{eq:Vf:est}
			\begin{split}
			\|V_f[\psi]\|_{L^2}&\le C\|\psi\|_{L^2},~\psi\in L^2(\mathcal O),\\
			\|V_f[\psi]\|_{\infty}&\le C\|\psi\|_{\infty},~\psi\in C(\mathcal O).
			\end{split}
		\end{equation}
		 \item For any $\psi\in L^2(\mathcal O)$, we have that
		\begin{align}\label{H2L2}
			\|V_f[\psi]\|_{H^2}&\le C(1+\|f\|_\infty)\|\psi\|_{L^2},
		\end{align}
		as well as
		\begin{align}
			\frac{1}{C(1+\|f\|_\infty)}\|\psi\|_{(H^2_0)^*}\le \|V_f[\psi]\|_{L^2}&\le C(1+\|f\|_\infty)\|\psi\|_{(H^2_0)^*}. \label{eq:schr:lb}
		\end{align}	
%		\item For any $\psi \in C(\mathcal O)$,
%		\begin{equation}\label{eq:schr:C2}
%		\|V_f\psi\|_{\mathcal C^2}\leq C(1+\|f\|_\infty)\|\psi\|_\infty.
%		\end{equation}
		\item If also $d\le 3$, then for any $\psi\in L^2(\mathcal O)$ and any $f,\bar f\in C(\mathcal O)$ with $f,\bar f\ge 0$, we have that 
		\begin{equation}\label{eq:Vfdiff}
				\|V_f[\psi]-V_{\bar f}[\psi]\|_\infty \lesssim (1+\|f\|_\infty)\|\psi\|_{L^2} \|f-\bar f\|_\infty.
		\end{equation}
	\end{enumerate}
\end{lem}

\begin{proof}
	Part i) is a direct consequence of the Feynman-Kac formula for $V_f[\psi]$ from \cite{CZ95} (see also Lemma 25 in \cite{NVW18}). The upper bounds in part ii) likewise are proved by standard arguments for elliptic PDEs (see, e.g., Lemma 26 in \cite{NVW18}). In order to prove the lower bound in (\ref{eq:schr:lb}), let us denote the Schr\"odinger operator by $S_f[w]= \frac 12 \Delta w -fw$. Since $S_f:H^2_0\to L^2$ satisfies $S_f V_f[ 
\psi] = \psi$, it suffices to show that
	 \[ \|S_fw \|_{(H^2_0)^*}\lesssim (1+\|f\|_\infty)\|w\|_{L^2},~w \in H^2_0.\]
 Using the divergence theorem we have that for such $w$,
	 \begin{equation*}
	 	\begin{split}
	 		\|S_f w \|_{(H^2_0)^*}&=\sup_{\psi\in H^2_0:\|\psi\|_{H^2_0}\le 1}\Big| \int_{\mathcal O} \psi S_fw\Big|\\
	 		&=\sup_{\psi\in H^2_0:\|\psi\|_{H^2_0}\le 1}\Big| \int_{\mathcal O} w S_f\psi \Big|\le \|w\|_{L^2}\sup_{\psi\in H^2_0:\|\psi\|_{H^2_0}\le 1}\|S_f\psi\|_{L^2},
	 	\end{split}
	 \end{equation*}
	 and the term on the right hand side is further estimated by
	 \[\|S_f\psi\|_{L^2}\lesssim \|\Delta \psi\|_{L^2}+\|f\psi\|_{L^2}\lesssim 1+\|f\|_\infty\|\psi\|_{L^2}\le 1+\|f\|_\infty, \]
	 which proves (\ref{eq:schr:lb}). Finally, (\ref{eq:Vfdiff}) is proved by using a Sobolev embedding as well as (\ref{eq:Vf:est}), (\ref{H2L2}):
	 \begin{equation*}
	 	\begin{split}
	 		\|V_f[\psi]-V_{\bar f}[\psi]\|_\infty &\lesssim \|V_f[ (f-\bar f)V_{\bar f}[\psi]] \|_{H^2}\lesssim (1+\|f\|_\infty) \|(f-\bar f)V_f [\psi]\|_{L^2}\\
	 		&\lesssim (1+\|f\|_\infty) \|f-\bar f\|_{\infty}\|\psi \|_{L^2}.
	 	\end{split}
	 \end{equation*}
\end{proof}

For any normed vector spaces $(V,\|\cdot\|_V)$ and $(W,\|\cdot\|_W)$ let $L(V,W)$, denote the space of bounded linear operators $V\to W$, equipped with the operator norm. For $g\in C^\infty(\partial\mathcal O)$ and any $f\in C(\mathcal O)$ with $f>0$, there exists a unique (weak) solution $G(f)\in C(\mathcal O)$ of (\ref{eq:schr}), see Theorem 4.7 in \cite{CZ95}. We define the operators $DG_f\in L(C(\mathcal O),C(\mathcal O))$ and $D^2G_f\in L(C(\mathcal O),L(C(\mathcal O),C(\mathcal O)))$ as
\begin{equation}\label{eq:DG}
\begin{split}
D G_f[h_1 ]= V_f[h_1u_{f}],~~ (D^2 G_f[h_1])[h_2]=V_f[h_1D G_f[h_2]]+V_f[h_2D G_f[h_1]], ~~h_1,h_2\in C(\mathcal O).
\end{split}
\end{equation}
The next lemma establishes that these operators are suitable Fr\'echet derivatives of $G$ on the open subset $\{f \in C(\mathcal O), f>0\}$ of $C(\mathcal O)$.

\begin{lem}\label{lem:schr:deriv}
	\begin{enumerate}[label=\textbf{\roman*)}]
	\item For any $f\in C(\mathcal O)$ with $f>0$, we have $G(f)\in C(\mathcal O)$. Moreover there exists $C>0$ such that for any $f,\bar f\in C(\mathcal O)$ with $f,\bar f >0$,
	\begin{equation}\label{hummel}
			\|G (\bar f)- G(f)\|_\infty \le C\|\bar f-f\|_\infty,
	\end{equation}
	as well as 
	\begin{equation}\label{deriv}
	\begin{split}
	\| G (\bar f)- G(f)-D G_f[\bar f- f]\|_\infty&\le C\|\bar f-f\|_\infty^2,\\
	\|D G_{\bar f}-D G_f-D^2 G_f[\bar f-f] \|_{L(C(\mathcal O),C(\mathcal O))}&\le C\|\bar f-f\|_\infty^2.
	\end{split}
	\end{equation}
	\item For any integer $\alpha >d/2$ there exists a constant $C>0$ such that for all $f\in H^\alpha$ with $\inf_{x\in \mathcal O}f(x)>0$, we have
	\begin{align}
		\|G(f)\|_{H^{2}} &\le C(\|f\|_{L^2}+\|g\|_{C^{2}(\partial \mathcal O)}), \label{2bound}\\
	\|G(f)\|_{H^{\alpha+2}}&\le C(1+\|f\|_{H^\alpha}^{\alpha/2+1})\|g\|_{C^{\alpha+2}(\partial \mathcal O)}.\label{alphabound}
	\end{align}
	\end{enumerate}
\end{lem}
\begin{proof}
The estimate (\ref{hummel}) follows from the identity  $ G(\bar f)-G(f)=V_f[(\bar f-f)G(\bar f)]$, (\ref{eq:Vf:est}) and (\ref{ubd}). Arguing similarly and using (\ref{hummel}), we further obtain
\begin{equation*}
\begin{split}
\| G (\bar f)-G(f)-D G_f[\bar f-f]\|_\infty&=\|V_f[(\bar f-f)( G (\bar f)- G(f))]\|_\infty\\
&\lesssim \|(\bar f-f)( G (\bar f)- G(f))\|_\infty\lesssim \|\bar f-f\|_\infty^2,
\end{split}
\end{equation*}
which proves the first part of (\ref{deriv}). For the second part of (\ref{deriv}), we have for any $h\in C(\mathcal O)$ that
\begin{equation*}
\begin{split}
DG_{\bar f}[h]-D G_{f}[h]&= V_{\bar f}[hu_{\bar f}]-V_f[hu_f]\\
&= V_{\bar f}[h(u_{\bar f}-u_f)]+(V_{\bar f}-V_f)[hu_f]\\
&=V_f[hD G_f[\bar f-f]]+R_1+V_f[(\bar f-f)V_f[hu_f]] + R_2\\
&=(D^2G_{f}[\bar f-f])[h]+R_1+R_2,
\end{split}
\end{equation*}
with remainder terms $R_1,R_2$ given by
\begin{equation*}
\begin{split}
R_1&= [V_{\bar f}-V_f][h(u_{\bar f}-u_f)] + V_f[h(u_{\bar f}-u_f-DG[h])],\\
R_2&=[V_{\bar f}-V_f](hu_f)-V_f[(\bar f-f)V_f[hu_f]].
\end{split}
\end{equation*}
Using the identity $(V_{\bar f}-V_f)\psi =V_f[(\bar f-f)V_{\bar f}[\psi]]$ with $\psi = h(u_{\bar f}-u_f)$, Lemma \ref{cpebach} as well as the first part of (\ref{deriv}), we have
\begin{equation*}
\begin{split}
\|R_1\|_\infty &\lesssim \|\bar f-f\|_\infty \|h(u_{\bar f}-u_f)\|_\infty +\|h\|_\infty \|u_{f+h}-u_f-D\bar G[h]\|_\infty\lesssim \|\bar f-f\|_\infty^2\|h\|_\infty,
\end{split}
\end{equation*}
and arguing similarly,
\[ \|R_2\|_\infty = \|V_f[(\bar f-f)(V_{\bar f}-V_f)[hu_f]]\|_\infty \lesssim \|\bar f-f\|_\infty\|(V_{\bar f}-V_f)[hu_f]\|_\infty \lesssim \|\bar f-f\|_\infty^2\|h\|_\infty. \]
This completes the proof of (\ref{deriv}).

\smallskip 

To prove (\ref{2bound}), we use that $(\Delta,\text{tr}):H^2(\mathcal O)\to L^2\times H^{3/2}(\partial\mathcal O)$ [where $\text{tr}$ denotes the boundary trace operator for the domain $\mathcal O$] is a topological isomorphism, see Theorem II.5.4 in \cite{LM72}, such that in particular
\[ \|G(f)\|_{H^2}\lesssim  \|fu_f\|_{L^2}+\|g\|_{C^2(\partial \mathcal O)}\le \|f\|_{L^2} +\|g\|_{C^2(\partial \mathcal O)}. \]
where we also used (\ref{ubd}). Finally, (\ref{alphabound}) is proved in Lemma 27 in \cite{NVW18}.
\end{proof}

\subsubsection{Properties of the map $\Phi^*$}\label{ssec-link}

We summarise some properties of `regular' link functions from Definition \ref{def:reg:link}. We recall the notation $\Phi^*$ for the associated composition operator from (\ref{fifa}). For any $F\in  C(\mathcal O)$, define the operators $D\Phi^*_F\in L(C(\mathcal O),C(\mathcal O))$, $D^2\Phi^*_F\in L(C(\mathcal O),L(C(\mathcal O),C(\mathcal O)))$ by
\begin{equation}\label{Dphi}
	D\Phi^*_{F}[H]=H\Phi'\circ F,~~~~	(D^2\Phi^*_F[H])[J]=HJ\Phi''\circ F, ~~~ H,J\in C(\mathcal O). 
\end{equation}
Then for any $F,H,J \in C(\mathcal O)$ and $x\in \mathcal O$, a Taylor expansion immediately implies that, with $\zeta_x,\bar \zeta_x$ denoting intermediate points between $F(x)$ and $(F+H)(x)$,
\begin{equation*}
\begin{split}
	|(\Phi^*(F+H)-\Phi^*(F)-D\Phi^*_F[H])(x)|&= |H^2(x)\Phi''(\zeta_x)/2|\le \|H\|_\infty^2\sup_{t\in \R}|\Phi''(t)|,\\
	\big|\big(D\Phi^*_{F+H}-D\Phi^*_{F}-D^2\Phi^*_F[H]\big)[J](x)\big|&= \big| J(x) H^2(x)\Phi'''(\bar \zeta_x)/2 \big|\le \|J\|_\infty\|H\|_\infty^2 \sup_{t\in \R}|\Phi'''(t)|,
\end{split}
\end{equation*}
whence $D\Phi^*,D^2\Phi^*$ are the Fr\'echet derivatives of $\Phi^*:C(\mathcal O)\to C(\mathcal O)$.

\smallskip

We also need the basic fact that for any integer $\alpha>d/2$ there exists $C>0$ such that for all $F\in H^\alpha(\mathcal O)$,
\begin{equation}\label{links}
	\|\Phi\circ F\|_{H^\alpha}\le C(1+	\|\Phi\circ F\|_{H^\alpha}^\alpha),
\end{equation}
see Lemma 29 in \cite{NVW18}. Finally, note that by the definition of $\Phi$, there exists $C'>0$ such that for any $\bar F, F\in C(\mathcal O)$,
\begin{equation}\label{linklip}
	\|\Phi\circ \bar F-\Phi\circ F\|_{\infty} \le C\|\bar F-F\|_\infty,~~\|\Phi\circ \bar F-\Phi\circ F\|_{L^2} \le C\|\bar F-F\|_{L^2}.
\end{equation}

\subsubsection{Chain rule for Fr\'echet derivatives}
Let $U,V$ be normed vector spaces and $\mathcal D\subseteq U$ an open subset. For a map $T:\mathcal D\to V$ we denote by $DT_\theta \in L(U,V)$ and $D^2T_\theta \in L(U,L(U,V))$ the first and second order Fr\'echet derivatives at $\theta \in \mathcal D$, respectively, whenever they exist. The following basic lemma then follows directly from the chain rule.
\begin{lem}\label{lem:chain}
	Suppose $U,V,W$ are (open subsets of) normed vector spaces, and suppose that $A:U\to V$ and $B:V\to W$ are both twice differentiable in the Fr\'echet sense. Then for any $\theta\in U$ and $H_1,H_2\in U$, we have that $D(B\circ A)_\theta= DB_{A(\theta)}\circ DA_\theta$ and 
	\begin{align}
	\big(D^2(B\circ A)_\theta [H_1]\big)[H_2]=\big(D^2B_{A(\theta)}[DA_\theta[H_1]]\big)[DA_\theta[H_2]]+DB_{A(\theta)}\big[(D^2A_\theta[H_1])[H_2]\big].   \label{eq:sec:der}
	\end{align}
\end{lem}

\subsection{Proof of Proposition \ref{prop:log:conv}}
We first record the following basic lemma without proof.
\begin{lem}\label{lem:norm:der}
	Let $|\cdot|$ be an ellipsoidal norm on $\R^D$ with associated matrix $M$, $|\theta|^2=\theta^TM\theta$ and define the function $n: \theta\to |\theta|$. Then for any $\theta\neq 0$, we have
	\begin{equation}\label{eq:f:derivs}
	\begin{split}
	\nabla n(\theta)= \frac{M\theta}{|\theta|},~~~~~~\nabla^2 n(\theta)=\frac{M}{|\theta|}- \frac{M\theta(M\theta)^T}{|\theta|^3},
	\end{split}
	\end{equation}
	as well as the norm estimates
	\begin{align}
	\|\nabla n(\theta)\|_{\R^D}&\le \sqrt{\lambda_{max}(M)},      	\label{eq:grad:norm}\\
	\|\nabla^2 n(\theta)\|_{op}&\le 2\lambda_{max}(M)/|\theta|_1.     	\label{eq:hess:norm}
	\end{align}
\end{lem}

Using Lemma \ref{lem:norm:der}, we prove the following bounds on the cut-off function $\alpha_\eta$.

\begin{lem}\label{lem:alphaeta}
	If $|\cdot|_1$ is an ellipsoidal norm with associated matrix $M$, $|\theta|_1^2=\theta^TM\theta$, then the function $\alpha_\eta$ from (\ref{eq:alphaeta:def}) satisfies that for all $\theta\in \R^D$,
	\begin{equation*}
	\begin{split}
	\|\nabla \alpha_\eta(\theta)\|_{\R^D}\le \frac{\|\alpha\|_{C^1}\sqrt{\lambda_{max}(M)}}{\eta},~~~~~~~\|\nabla^2\alpha_\eta(\theta)\|_{op}\le \frac{4\|\alpha\|_{C^2}\lambda_{max}(M)}{\eta^2}.
	\end{split}
	\end{equation*}
\end{lem}

\begin{proof}
	We may assume w.l.o.g. that $\theta_{init}=0$ and we write $n(\theta)=|\theta|_1$. The gradient bound is obtained by the chain rule and (\ref{eq:grad:norm}):
	\begin{equation*}
	\begin{split}
	\|\nabla \alpha_\eta(\theta)\|_{\R^D}&= \big\|\eta^{-1}\alpha'\big(|\theta|_1/\eta\big)\nabla n(\theta)\big\|_{\R^D}\le \eta^{-1} \|\alpha\|_{C_1}\sqrt{\lambda_{max}(M)}.
	\end{split}
	\end{equation*}
	For the Hessian, we similarly employ the chain rule, (\ref{eq:grad:norm}), (\ref{eq:hess:norm}) as well as the fact that $\alpha'(t)=0$ when $t\in (0,3/4)$:
	\begin{equation*}
	\begin{split}
	\|\nabla^2\alpha_\eta(\theta)\|_{op}&\le \eta^{-2}\big\|\alpha''\big(|\theta|_1/\eta\big)\nabla n(\theta)\nabla n(\theta)^T\big\|_{op}+\eta^{-1}\big\|\alpha'\big(|\theta|_1/\eta\big)\nabla^2n(\theta)\big\|_{op}\\
	&\le\eta^{-2}\|\alpha\|_{C^2}\|\nabla n(\theta)\|_{\R^D}^2 +  \eta^{-1}\|\alpha\|_{C_1}\mathbbm 1_{\{|\theta|\ge 3\eta/4 \}}\cdot \frac{2\lambda_{max}(M)}{|\theta|_1} \\
	&\le 4 \eta^{-2}\|\alpha\|_{C^2}\lambda_{max}(M).
	\end{split}
	\end{equation*}
\end{proof}

We now turn to the proof of Proposition \ref{prop:log:conv}. Throughout, we work on the event $\mathcal E_{conv}\cap \mathcal E_{init}$ defined by (\ref{eq:E:event}),(\ref{initevent}); moreover we assume without loss of generality that $\theta_{init}=0$.

\begin{proof}[Proof of Proposition \ref{prop:log:conv}] We divide the proof into five steps.
\par 
\textbf{1. Local lower bound for $\alpha_\eta\ell_N$.}
For the set
\[ V:=\{\theta: |\theta|_1\le 3\eta/4 \},\]
by definition of $\mathcal E_{init}$, we have that $V\subseteq \mathcal B$. Thus using the definitions of $\mathcal E_{conv}$ and of $\alpha_\eta$, we obtain
\begin{equation}\label{eq:alphaeta:lb}
\inf_{\theta\in V}\lambda_{min} \big(-\nabla^2[\alpha_\eta\ell_N](\theta)\big)\ge Nc_{min}/2.
\end{equation}
\par 
\textbf{2. Upper bound for $\alpha_\eta\ell_N$.}  By the chain rule, Lemma \ref{lem:alphaeta}, the definition of $\mathcal E_{conv}$ and using that $\|\alpha\|_{C^2}\ge 1$, we obtain that for any $\theta\in\R^D$ and some $c=c(\alpha)$,
\begin{align}
\|\nabla^2[\alpha_\eta\ell_N&](\theta)\|_{op}\le |\ell_N(\theta)|\|\nabla^2 \alpha_\eta (\theta)\|_{op} +2 \|\nabla \alpha_\eta(\theta)\|_{\R^D}\|\nabla \ell_N(\theta)\|_{\R^D}+|\alpha_\eta(\theta)|\|\nabla^2 \ell_N(\theta)\|_{op}\notag\\
&\le 2\sup_{\theta\in\mathcal B}\Big(\big[|\alpha_\eta(\theta)| + \|\nabla \alpha_\eta(\theta)\|_{\R^D}+\|\nabla^2 \alpha_\eta(\theta)\|_{op}\big]\big[|\ell_N(\theta)| + \|\nabla \ell_N(\theta)\|_{\R^D}+\|\nabla^2 \ell_N(\theta)\|_{op}\big]\Big)\notag\\
&\le c\big( 1+\lambda_{max}(M)/\eta^2 \big)\cdot N(c_{max}+1).\label{eq:alphaeta:ub}
\end{align}

\textbf{3. Global lower bound for $\nabla^2 g_\eta$.} First we note that $g_\eta$ is convex on all of $\R^D$: Indeed, this follows from the the identity $\gamma_\eta=\tilde\gamma_\eta \ast \varphi_{\eta/8}$, the convexity of the functions $n:\theta\mapsto |\theta|_1$, $\tilde \gamma_\eta$ and the fact that convolution with the positive function $\varphi_{\eta/8}$ preserves convexity. As $g_\eta$ has $C^2$ regularity, it follows that $\nabla^2 g_\eta \succeq 0$ on all of $\R^D$.
\par
We next prove a quantitative lower bound for $\nabla^2 g_\eta$ on the set $V^c$. By the chain rule and Lemma \ref{lem:norm:der}, we have that for any $\theta\in\R^D$, writing $v=\nabla n(\theta)$,
\begin{equation}\label{eq:geta:est}
\begin{split}
\nabla^2 g_\eta(\theta) &= \gamma_\eta''(|\theta|_1)\nabla n(\theta)\nabla n(\theta)^T + \gamma_\eta'(|\theta|_1)\nabla^2 n(\theta)\\
&=\gamma_\eta''(|\theta|_1)vv^T + \frac{\gamma_\eta'(|\theta|_1)}{|\theta|_1}\big( M-vv^T \big)\\
&= \Big(\gamma_\eta''(|\theta|_1)-\frac{\gamma_\eta'(|\theta|_1)}{|\theta|_1}  \Big) vv^T+\frac{\gamma_\eta'(|\theta|_1)}{|\theta|_1}M\\
&=: A(|\theta|_1) vv^T+B(|\theta|_1)  M.
\end{split}
\end{equation}
To derive lower bounds for the functions $B(\cdot)$ and $A(\cdot)$, we first observe that by the symmetry of $\varphi_{\eta/8}$ around $0$, it holds for any $t\ge 3\eta/4$ that
\begin{equation}\label{eq:gammaeta:der}
\gamma_\eta'(t) = \int_{[-\eta/8,\eta/8]} \varphi_{\eta/8}(y)\cdot 2(t-y-5\eta/8)=2(t-5\eta/8).
\end{equation}
Thus the function $B(t)=\gamma_\eta'(t)/t$ strictly increases on $(3\eta/4,\infty)$, and for any $t\ge 3\eta/4$, we obtain
\begin{equation}\label{eq:B:lb}
B(t)\ge B(3\eta /4)=\frac {\gamma_\eta'(3\eta/4)}{3\eta/4}= 2\frac{3\eta/4-5\eta/8}{3\eta/4}=\frac 13.
\end{equation}
For the term $A(\cdot)$, we note that for any $t\ge 3\eta/4$, using that  $\gamma_\eta''(t)=2$ as well as (\ref{eq:gammaeta:der}), we have
\begin{equation}\label{eq:A:lb}
A(t)=2-\frac{2(t-5\eta/8)}{t}\ge 0.
\end{equation}
Combining the displays (\ref{eq:geta:est}), (\ref{eq:B:lb}), (\ref{eq:A:lb}), we have proved the lower bound
\begin{equation}\label{eq:geta:lb}
 \inf_{\theta\in V^c}\lambda_{min}\big(\nabla^2 g_\eta(\theta) \big)\ge \lambda_{min}(M)/3,~~~.
\end{equation}
\par
\textbf{4. Global upper bound for $\nabla^2 g_\eta$.} We note that the functions $A(\cdot)$, $B(\cdot)$ from (\ref{eq:geta:est}) satisfy
\[ \sup_{t\in(0,\infty)} |A(t)|\le  \sup_{t\in(0,\infty)} |\gamma_\eta'(t)/t|+|\gamma_\eta''(t)|\le 4,~~~\sup_{t\in(0,\infty)} |B(t)|\le  \sup_{t\in(0,\infty)} |\gamma_\eta'(t)/t|\le 2.\]
Hence, by (\ref{eq:geta:est}) and  Lemma \ref{lem:norm:der}, we obtain that
\begin{equation}\label{eq:geta:ub}
\|\nabla^2 g_\eta(\theta)\|_{op}\le 4\|vv^T\|_{op} + 2\|M\|_{op}\le 6\lambda_{max}(M),~~~  \theta \in\R^D.
\end{equation}
\par
\textbf{5. Combining the bounds.}
Combining the estimates (\ref{eq:alphaeta:lb}), (\ref{eq:alphaeta:ub}) and (\ref{eq:geta:lb}), we obtain that 
\begin{equation}
\begin{split}
\inf_{\theta\in V}\lambda_{min}\big(- \nabla^2\tilde \ell_N(\theta) \big)&\ge \frac{Nc_{min}}2,\\
\inf_{\theta\in V^c}\lambda_{min}\big(- \nabla^2\tilde \ell_N(\theta) \big)&\ge \frac{K\lambda_{min}(M)}{3}- c \big(1+\lambda_{max}(M)/\eta^2\big) N(c_{max}+1).
\end{split}
\end{equation}
In particular, there exists $C\ge 3$ such that for any $K$ satisfying (\ref{eq:K:cond}), we have
\begin{align*}
\inf_{\theta\in \R^D}\lambda_{min}\big(-\nabla^2 \tilde\ell_N(\theta)\big)\ge \min\Big\{\frac{Nc_{min}}2, \frac{K\lambda_{min}(M)}6 \Big\}=Nc_{min}/2,
\end{align*}
which completes the proof of (\ref{eq:str:conv}). To prove (\ref{eq:grad:lip}), we use (\ref{eq:alphaeta:ub}), (\ref{eq:geta:ub}) and (\ref{eq:K:cond}) to obtain that for all $\theta\neq \bar\theta \in\R^D$,
\begin{equation*}
\begin{split}
\frac{\|\nabla \tilde\ell_N(\theta)-\nabla \tilde\ell_N(\bar \theta)\|_{\R^D}}{\|\theta-\bar\theta\|_{\R^D}}&\le \sup_{\theta\in\R^D} \|\nabla^2 \tilde\ell_N(\theta)\|_{op}\\
&\le c\|\alpha\|_{C^2} \big(1+\lambda_{max}(M)/\eta^2\big) N(c_{max}+1)+6K\lambda_{max}(M)\\
&\le 7K\lambda_{max}(M).
\end{split}
\end{equation*}
\end{proof}

\subsection{Initialisation} \label{sec:initpfs}

In this section we prove the existence of polynomial time `initialiser' $\theta_{init}=\theta_{init}(Z^{(N)}) \in \mathbb R^D$ (that lies in the region $\mathcal B_{1/\log N}$ from (\ref{eq:B:def}) of strong log-concavity of the posterior measure with high $P_{\theta_0}^N$-probability, when $\alpha>6$), in the Schr\"odinger model. 
\begin{thm}\label{triebelei}
Suppose $\theta_0 \in h^\alpha(\mathcal O)$ for some $\alpha>2+d/2, d \le 3$. Then there exists a measurable function $\theta_{init} \in \mathbb R^D$ of the data $Z^{(N)}$ from (\ref{ZN}) and large enough $M'>0$ such that for all $N,D\in \N$ and some $\bar c>0$,
\begin{equation*}
P_{\theta_0}^N\big(\|\theta_{init} - \theta_{0,D}\|_{\R^D}  >M' N^{-(\alpha-2)/(2\alpha+d)} \big) \lesssim e^{-\bar c N^{d/(2\alpha+d)}}.
\end{equation*} Moreover $\theta_{init}$ is the output of a polynomial time algorithm involving $O(N^{b_0}), b_0>0,$  iterations of gradient descent (each requiring a multiplication with a fixed $D' \times D'$ matrix, $D' \lesssim N^{d/(2\alpha+d)}$).
\end{thm}

\begin{proof}
\textbf{Step I.}
To start, consider the wavelet frame $$\big\{\phi_{l,r}, 1 \le r \le N_l, l \in \mathbb N \big\}, N_l \lesssim 2^{ld},$$ of $L^2(\mathcal O)$ constructed in Theorem 5.51 in \cite{T08}. Then for data arising from (\ref{model}), choosing $$ ~2^J \simeq N^{1/(2\alpha+d)}= (N \delta^2_N)^{1/d},~\delta_N = N^{-\alpha/(2\alpha+d)},~n_J \equiv \sum_{l \le J} N_l \lesssim 2^{Jd},$$
and for multiscale vectors $(\lambda_{l,r}) \in \R^{n_J}$, define
\begin{equation}\label{global}
\hat \lambda = \arg\min_{\lambda \in \R^{n_J}} \left[\frac{1}{N}\sum_{i=1}^N\big(Y_i - \sum_{l \le J, r} \lambda_{l,r} \phi_{l,r}(X_i)\big)^2 +  \delta_N^2 \|\lambda\|^2_{h^\alpha} \right],~~~ \|\lambda\|_{h^\alpha}^2=\sum_{l,r}2^{2l\alpha}\lambda_{l,r}^2.
\end{equation} 
Next we set $$\hat u = \hat u(Z^{(N)}) = \sum_{l \le J, r} \hat \lambda_{l,r} \phi_{l,r},~~ u_{f_0,J} = \sum_{l \le J,r} \lambda_{0,l, r} \phi_{l,r},$$ where the $\lambda_{0,l,r} \in h^{\alpha+2}$ are frame coefficients of $u_{f_0}=\mathcal G(\theta_0) \in H^{\alpha+2}$ furnished by Theorem 5.51 in \cite{T08} and the elliptic regularity estimate (\ref{alphabound}). In particular by the Sobolev embedding $h^{\alpha+2} \subset b^\alpha_{\infty \infty}$ ($d<4$) and again Theorem 5.51 in \cite{T08} we can prove
\begin{equation}\label{bias}
\|u_{f_0} -u_{f_0,J}\|_{L^2} \lesssim \|u_{f_0}-u_{f_0,J}\|_{\infty} \lesssim 2^{-J\alpha} \lesssim \delta_N.
\end{equation}
We now apply a standard result from $M$ estimation \cite{V00, V01}, with empirical norms $$\|u\|^2_{(N)}= \frac{1}{N} \sum_{i=1}^N u^2(X_i),$$ conditional on the design $X_1, \dots, X_n$, to obtain the following bound.
\begin{prop}
We have for $\alpha>d/2$, all $N$ and some constant $c>0$,
\begin{equation}\label{mongolei}
P_{\theta_0}^N\big(\|\hat u - u_{f_0}\|^2_{(N)} + \delta_N^2 \|\hat \lambda\|_{h^{\alpha}}^2 > \|u_{f_0} - u_{f_0,J}\|_{(N)}^2 + \delta_N^2 \|\lambda_{0, l,r}\|_{h^{\alpha}}^2 | (X_i)_{i=1}^N \big) \le e^{-cN \delta_N^2}.
\end{equation}
\end{prop}
\begin{proof}
We apply Theorem 2.1 in \cite{V01}. We can bound the $\|\cdot\|_\infty$ and then also $\|\cdot\|_{(N)}$-metric entropy of the class of functions $$\Big\{u: u = \sum_{l \le J, r}  \lambda_{l,r} \phi_{l,r}; \|\lambda\|^2_{h^{\alpha}} \le m\Big\},~~m>0,$$ by the metric entropy of a ball of radius $m$ in a $H^{\alpha}$-Sobolev space, which by (4.184) in \cite{GN16} is of order $H(\tau) \lesssim (m/\tau)^{d/\alpha}$ for every $m>0$. Then arguing as in Section 3.1.1 in \cite{V01} (the only notational difference being that here $d>1$), the result follows. 
\end{proof}

\medskip

This implies in particular, using $\|u\|_{(N)} \le \|u\|_\infty$, (\ref{bias}), $\lambda_{0,l,r} \in h^{\alpha+2}$ and Theorem 5.51 in \cite{T08}, that for some $C, C'>0$,
\begin{equation}\label{regul}
P_{\theta_0}^N\big(\|\hat u\|_{H^{\alpha}}^2 > C\big) \le P_{\theta_0}^N\big(\|\hat \lambda\|_{h^{\alpha}}^2 > C'\big) \le \exp\{-cN\delta_N^2\}.
\end{equation}
as well as
\begin{equation}\label{rate}
P_{\theta_0}^N\big(\|\hat u - u_{f_0,J}\|^2_{(N)}  > C \delta_N^2 \big) \le  \exp\{-cN \delta_N^2\}.
\end{equation}
In Step IV below we establish the following restricted isometry type bound
\begin{align} \label{ripbd}
P_{\theta_0}^N \left(\Big|\frac{\|\hat u - u_{f_0,J}\|^2_{(N)}}{\|\hat u - u_{f_0,J}\|^2_{L^2}} - 1\Big| \le \frac{1}{2}\right) \ge 1- c''e^{-c' N\delta_N^2}
\end{align}
for some constants $c',c''>0$ so that in particular 
$$P_{\theta_0}^N \left(\frac{1}{2} \le \frac{\|\hat u - u_{f_0,J}\|^2_{(N)}}{\|\hat u - u_{f_0,J}\|^2_{L^2}} \le \frac{3}{2}\right) \ge 1- c''e^{-c' N\delta_N^2}.$$
On the event $\mathcal A_N$ in the last probability we can write, using again (\ref{bias}) and (\ref{rate}), for $M$ large enough,
\begin{align*}
& P_{\theta_0}^N\big(\|\hat u - u_{f_0}\|^2_{L^2}  > M \delta_N^2  \big) \le P_{\theta_0}^N\big(\|\hat u - u_{f_0,J}\|^2_{L^2}  > (M/2) \delta_N^2  \big) \\
& \le P_{\theta_0}^N\left(\frac{\|\hat u - u_{f_0,J}\|^2_{L^2}}{\|\hat u - u_{f_0,J}\|^2_{(N)}}\|\hat u - u_{f_0,J}\|^2_{(N)}  > (M/2) \delta_N^2, \mathcal A_N \right) +c''e^{-c'N\delta_N^2} \\
& \le P_{\theta_0}^N\left( \|\hat u - u_{f_0,J}\|^2_{(N)} > (M/4) \delta_N^2 \right)  + c''e^{-c'N\delta_N^2} \lesssim e^{-cN \delta_N^2} + e^{-c'N \delta_N^2}.
\end{align*}
Overall what precedes implies that we can find $M$ large enough such that for some constants $\bar c, \bar c'>0$,
\begin{equation}\label{finalbd}
P_{\theta_0}^N\big(\|\hat u - u_{f_0}\|^2_{L^2}  \le M\delta_N^2 \text{ and } \|\hat u\|_{H^{\alpha}}^2 \le M \big)  \ge 1- \bar c' e^{-\bar c N \delta_N^2}.
\end{equation}

\textbf{Step II.}
By definition of the $\|\cdot\|_{h^\alpha}$-norm, the objective function minimised in (\ref{global}) over $\R^{n_J}$ is $m$-strongly convex with convexity bound $m \ge \delta_N^2$. Moreover, noting that the sum-of-squares term $Q_N$ appearing in (\ref{global}) satisfies
\[ \frac{\partial Q_N}{\partial \lambda_{l',r'}}(\lambda)= -\frac{2}{N}\sum_{i=1}^N \big[ Y_i-\sum_{l\le J,r}\lambda_{l,r}\phi_{l,r}(X_i) \big]\phi_{l',r'}(X_i), ~~~ l'\le J,~1\le r'\le N_{l'}, \]
we can deduce that the gradient of the objective function is globally Lipschitz with constant at most of order $O(2^{Jd})=O(N\delta_N^2)$, using standard properties of the wavelet frame from Definition 5.25 in \cite{T08}. Using (\ref{ubd}), (\ref{bernstein}) and a standard tail inequality for $\chi^2$-random variables (Theorem 3.1.9 in \cite{GN16}), one shows further that for some $\bar C>0$ and on events of sufficiently high $P_{\theta_0}^N$-probability,
\[ Q_N(0)=\frac 1N \sum_{i=1}^N \big( \eps_i^2 + 2\eps_i u_{f_0}(X_i) +u_{f_0}^2(X_i) \big) \le \bar C. \]
By Proposition \ref{reiskorn} and using the standard sequence norm inequality
\[ \|v\|_{h^\beta} \le 2^{J\beta} \|v\|_{\ell^2} \lesssim  N^{\frac{\beta}{2\alpha +d}} \|v\|_{\ell^2} ,~~ v\in \R^{n_J}, ~\beta \ge 0,\]
we deduce that on preceding events and for any fixed $p>0$ there exists $b_0 >0$ such that the output $\lambda_{init}\in \R^{n_J}$ from $O(N^{b_0})$ iterations of gradient descent satisfies $\|\lambda_{init}-\hat\lambda\|_{h^\alpha}\le N^{-p}.$ In particular we can choose $p$ such that, denoting $$u_{init} := \sum_{l \le J,r} \lambda_{init,l,r} \phi_{l,r},$$ we have that $\|\hat u - u_{init}\|_{H^\alpha}\lesssim \|\hat \lambda - \lambda_{init}\|_{h^\alpha} = o(\delta_N)$; hence by virtue of (\ref{finalbd}), we may restrict the rest of the proof to an event of sufficiently large probability where $u_{init}$ satisfies
\begin{equation}\label{guetersloh}
\|u_{init} - u_{f_0}\|^2_{L^2} + \delta_N^2 \|u_{init}\|_{H^{\alpha}}^2 \le  (2M+1)\delta_N^2.
\end{equation}

\textbf{Step III.}
From the interpolation inequality for Sobolev norms from Section \ref{notatio} and (\ref{guetersloh}) we now obtain, with sufficiently high $P_{\theta_0}^N$-probability,
\begin{equation}
\|u_{init} - u_{f_0}\|_{H^2} \leq \bar M N^{-(\alpha-2)/(2\alpha+d)}
\end{equation}
and the Sobolev imbedding ($d<4$) further implies $\|u_{init} - u_{f_0}\|_\infty \to 0$ as $N \to \infty$ so that we deduce from (\ref{leivand}) $\hat u \ge u_{f_0}/2 \ge c>0$ with sufficiently high $P_{\theta_0}^N$-probability. So on these events we can define a new estimator
\begin{equation}
f_{init} = \frac{\Delta u_{init}}{2 u_{init}}, ~~\text{ noting that } f_0 = \frac{\Delta u_{f_0}}{2 u_{f_0}}.
\end{equation}
For $F_{init} = \Phi^{-1}\circ f_{init}$, using also the regularity of the inverse link function (\ref{linklip}), we then see $$\|F_{init}-F_{\theta_0}\|_{L^2} \lesssim \|f_{init} - f_0\|_{L^2} \lesssim \|u_{init} - u_{f_0}\|_{H^2},$$ and hence for some $M'>0$,
$$P_{\theta_0}^N\big(\|F_{init} - F_{\theta_0}\|_{L^2}  \le M' N^{-(\alpha-2)/(2\alpha+d)}\big) \ge 1- \bar c' e^{-\bar c N\delta_N^2}.$$ 
We finally define $\theta_{init}$ as 
\[\theta_{init} = (\langle F_{init}, e_k \rangle_{L^2}: k \le D) \in \R^D,~~D \in \mathbb N,\]
the vector of the first $D$ `Fourier coefficients' of $F_{init}$. Then we obtain from Parseval's identity that $\|\theta_{init} - \theta_{0,D}\|_{\R^D} \le  \|F_{init} - F_{\theta_0}\|_{L^2}$, which combined with the last probability inequality establishes convergence rate desired in Theorem \ref{triebelei}. 

\smallskip

\textbf{Step IV. Proof of (\ref{ripbd}).}
Let us introduce the symmetric $n_J \times n_J, n_J \lesssim 2^{Jd},$ matrices $$\hat \Gamma_{(l,r), (l',r')} = \frac{1}{N} \sum_{i=1}^N \phi_{l,r}(X_i) \phi_{l',r'}(X_i),~~\Gamma_{(l,r), (l', r')} = \int_\mathcal O \phi_{l,r}(x) \phi_{l',r'}(x)dP^X(x),$$ and vectors $(\hat \lambda=\hat \lambda_{l,r}), (\lambda_0=\lambda_{0,l,r}) \in \mathbb R^{n_J}$. Then we can write 
$$\|\hat u - u_{f_0,J}\|^2_{(N)}-\|\hat u - u_{f_0,J}\|^2_{L^2(\mathcal O)} = (\hat \lambda-\lambda_0)^T (\hat \Gamma - \Gamma) (\hat \lambda-\lambda_0)$$
and hence (one minus the) probability relevant in (\ref{ripbd}) can be bounded as
\begin{align*}
\Pr \left(\Big|\frac{(\hat \lambda-\lambda_0)^T (\hat \Gamma - \Gamma) (\hat \lambda-\lambda_0)}{(\hat \lambda-\lambda_0)^T  \Gamma (\hat \lambda-\lambda_0)}\Big| > 1/2\right) \le \Pr \left(\sup_{v \in \mathbb R^{n_J}: v^T\Gamma v \le 1}\big|v^T(\hat \Gamma - \Gamma) v \big| > 1/2\right) . 
\end{align*}
We also note that by the frame property of the $\{\phi_{l,r}\}$, specifically from (5.252) in \cite{T08} with $s=0, p=q=2$, for any $u_v=\sum_{l \le J,r} v_{l,r} \phi_{l,r}$ we have the norm equivalence
\begin{equation}\label{frame}
\|v\|^2_{\mathbb R^{n_J}} \simeq \|u_v\|^2_{L^2} = \sum_{l, l'\le J,r, r'} v_{l,r} v_{l',r'} \Gamma_{(l,r), (l', r')} = v^T \Gamma v  =: \|v\|^2_{\Gamma},
\end{equation}
with the constants implied by $\simeq$ independent of $J$. Next for any $\kappa>0$ let $$\{v_m, m=1, \dots, M_{J,\kappa}\},~ M_{J,\kappa} \lesssim (3/\kappa)^{n_J}$$ denote the centres of balls of $\|\cdot\|_\Gamma$-radius $\kappa$ covering the unit ball $V_\Gamma$ of $(\mathbb R^{n_J}, \|\cdot\|_{\Gamma})$ (e.g., as in Prop. 4.3.34 in \cite{GN16} and using (\ref{frame})). Then using the Cauchy-Schwarz inequality
\begin{align*}
|v^T(\hat \Gamma - \Gamma) v|& = |(v-v_m+v_m)^T(\hat \Gamma - \Gamma) (v-v_m+v_m)| \\
&\le \|v-v_m\|_\Gamma^2 \sup_{v \in V_\Gamma}|v^T(\hat \Gamma - \Gamma) v \big| + 2\|v-v_m\|_{\Gamma} \|(\hat\Gamma-\Gamma) v\|_{\Gamma} + |v_m^T (\hat \Gamma- \Gamma) v_m| \\
&\le (\kappa^2 + 2\kappa)\sup_{v \in V_\Gamma}|v^T(\hat \Gamma - \Gamma) v \big| +  |v_m^T (\hat \Gamma- \Gamma) v_m|
\end{align*}
so choosing $\kappa$ small enough so that $\kappa^2+2\kappa <1/4$ we obtain
\begin{equation}
\sup_{v \in V_\Gamma}|v^T(\hat \Gamma - \Gamma) v \big| \le (4/3) \max_{m=1, \dots, M_J} |v_m^T (\hat \Gamma- \Gamma) v_m|,~~M_J \equiv M_{J, \kappa}.
\end{equation}
In particular, using also that $M_J \lesssim e^{c_02^{Jd}} \le e^{c_1 N\delta_N^2}$, the last probability is thus bounded by 
\begin{equation}\label{vershy}
\Pr\Big(\max_{m=1, \dots, M_J} |v_m^T(\hat \Gamma- \Gamma) v_m| >1/4 \Big) \le  e^{c_1 N\delta_N^2} \max_m \Pr\Big(|v_m^T(\hat \Gamma- \Gamma) v_m| >1/4 \Big).
\end{equation}
Each of the last probabilities can be bounded by Bernstein's inequality (Prop.~3.1.7 in \cite{GN16}) applied to $$v_m^T(\hat \Gamma- \Gamma) v_m=\frac{1}{N}\sum_{i=1}^N Z_{i}-EZ_i,$$ with i.i.d.~variables $Z_i=Z_{i,m}$ given by
\begin{equation}
Z_i = \sum_{l,l'\le J,r, r'} v_{m,l,r}v_{m, l', r'} \phi_{l,r}(X_i) \phi_{l',r'}(X_i)= \sum_{l\le J,r} v_{m,l,r} \phi_{l,r}(X_i) \sum_{l' 
\le J, r'} v_{m, l', r'}\phi_{l',r'}(X_i),
\end{equation}
wit vectors $v_m$ all satisfying $\|v_m\|_{\Gamma} \le 1$. For these variables we have from the Cauchy-Schwarz inequality
\begin{equation*}
|Z_i| \le \Big|\sum_{l \le J,r} v_{m,l,r} \phi_{l,r}(\cdot)\Big|^2 \le \|v_m\|_{\mathbb R^{n_J}}^2 \sum_{l\le J,r}(\phi_{l,r}(\cdot))^2 \leq c 2^{Jd} \equiv U
\end{equation*}
where the constant $c$ depends only on the wavelet frame (cf.~(\ref{frame}) and also Definition 5.25 in \cite{T08}). Similarly, using the previous estimate, we can bound
\begin{align*}
EZ_i^2 &= E\Big[\sum_{l\le J,r} v_{m,l,r} \phi_{l,r}(X_i)\Big]^4 \le U  \int_\mathcal O \Big[\sum_{l\le J,r} v_{m,l,r} \phi_{l,r}(x)\Big]^2 dx= U \|v_m\|^2_\Gamma \le U.
\end{align*}
Now Proposition 3.1.7 in \cite{GN16} implies for some constant $c_0>0$
$$\Pr\Big(N|v_m (\hat \Gamma- \Gamma) v_m| >N/4 \Big) \le 2\exp \Big\{-\frac{N^2/16}{2NU+(2/12) NU}\Big\} \le 2e^{-c_0/\delta_N^2} $$ since $U=c2^{Jd} \simeq N \delta_N^2$. Now since $\alpha>d/2$ we have $\delta_N^2 = o(1/\sqrt N)$ and thus $(1/\delta_N^2) \gg N \delta_N^2$ which means that the r.h.s~in (\ref{vershy}) is bounded by a constant multiple of $e^{-c'N\delta_N^2}$ for some $c'>0$, completing the proof.

\end{proof}

\bibliography{bib_2019_12}
\bibliographystyle{abbrv}
\end{document}